\documentclass[10pt]{amsart}

\usepackage{a4wide}
\usepackage{enumerate}
\usepackage{amsmath, amsthm, amsfonts, amssymb,accents,a4}
\usepackage{dsfont}
\usepackage{srcltx, color}

\usepackage[shadow, textsize = small]{todonotes}

\usepackage{tikz}
\usetikzlibrary{patterns, arrows}

\theoremstyle{plain}
\newtheorem{theorem}{Theorem}[section]
\newtheorem{lemma}[theorem]{Lemma}
\newtheorem{corollary}[theorem]{Corollary}
\newtheorem*{assumption*}{Assumption}
\newtheorem{assumption}{Assumption}

\newtheorem{assumpt}{Assumption}

\theoremstyle{remark}
\newtheorem{remark}[theorem]{Remark}
\newtheorem{convention}[theorem]{Convention}

\numberwithin{equation}{section}

\def\a{\alpha}
\def\b{\beta}
\def\om{\omega}
\def\e{\varepsilon}
\renewcommand{\epsilon}{\varepsilon}

\def\g{\gamma}
\def\G{\Gamma}
\def\l{E}
\def\lmb{\lambda}
\def\p{\partial}
\def\D{\Delta}
\def\Om{\Omega}
\def\z{\zeta}
\def\vp{\varphi}
\def\vr{\varrho}

\def\d{\delta}
\def\L{\Lambda}

\def\H1{W^{1,2}}
\def\H2{W^{2,2}}

\def\di{\,d}
\def\iu{\mathrm{i}}

\def\cI{\mathcal{I}}

\newcommand{\EE}{\mathbb{E}}
\newcommand{\NN}{\mathds{N}}

\newcommand{\PP}{\mathbb{P}}

\newcommand{\RR}{\mathds{R}}
\newcommand{\ZZ}{\mathds{Z}}

\def\cL{\mathcal{L}}
\def\Op{\mathcal{H}}
\def\cS{\mathcal{S}}
\def\Dom{\mathfrak{D}}

\DeclareMathOperator{\supp}{supp}

\def\cB{\mathcal{B}}

\def\spec{\sigma}

\def\cA{\mathcal{A}}
\def\cR{\mathcal{R}}

\def\cT{\mathcal{T}}

\def\cP{\mathcal{P}}

\def\cU{\mathcal{U}}
\def\Ups{\Upsilon}

\newcommand{\vecfield}{\mathfrak{A}}

\renewcommand\e{\varepsilon}

\newcommand{\Tr}{\operatorname{Tr}}
\newcommand{\weyl}{\mathrm{Weyl}}
\newcommand{\drm}{\mathrm{d}}
\newcommand{\1}{\mathbf{1}}

\DeclareMathOperator{\RE}{Re}
\DeclareMathOperator{\dist}{dist}
\DeclareMathOperator{\IM}{Im}
\DeclareMathOperator{\sgn}{sign}
\DeclareMathOperator{\Div}{div}

\renewcommand{\leq}{\leqslant}
\renewcommand{\geq}{\geqslant}

\begin{document}

\allowdisplaybreaks

\title[]{Quantum Hamiltonians with weak random abstract perturbation. II. Localization in the expanded spectrum}
\author[Borisov, T\"aufer, and Veseli\'c]{Denis Borisov$^{1,2,3}$, Matthias T\"aufer$^{4}$, and  Ivan Veseli\'c$^{5}$}
\today

\thanks{\jobname.tex}

\begin{abstract}
We consider multi-dimensional  Schr\"odinger operators  with a weak random perturbation
distributed in the cells of some periodic lattice.
In every cell the perturbation is described by the translate of a fixed abstract operator depending on a random variable.
The random variables, indexed by the lattice,  are assumed to be independent and identically distributed according to an absolutely continuous probability density.
A small global coupling constant tunes the strength of the perturbation.
We treat analogous random Hamiltonians defined on multi-dimensional  layers, as well.
For such models we determine the location of the almost sure spectrum and its dependence on the global coupling constant.
In this paper we concentrate on the case that the spectrum expands when the perturbation is switched on.
Furthermore, we derive a Wegner estimate and an initial length scale estimate, which together with Combes--Thomas estimate
allows to invoke the multi-scale analysis proof of localization.
We specify an energy region, including the bottom of the almost sure spectrum, which exhibits spectral and dynamical localization.
Due to our treatment of general, abstract perturbations our results apply at once to many interesting examples both known and new.
\end{abstract}

\maketitle

\begin{quote}
\begin{itemize}
\small
\item[1)] \emph{Department of Differential Equations, Institute of Mathematics with Computer Center, Ufa Scientific Center, Russian Academy of Sciences, Chernyshevsky. st.~112, Ufa, 450008,
Russia, Email: borisovdi@yandex.ru}

\item[2)] \emph{Faculty of Physics and Mathematics, Bashkir
State Pedagogical University, October rev. st.~3a, Ufa, 450000,
Russia}

\item[3)] \emph{Faculty of Science, University of Hradec Kr\'alov\'e, Rokitansk\'eho 62, 500 03, Hradec Kr\'alov\'e, Czech Republic}

\item[4)] \emph{School of Mathematical Sciences, Queen Mary University of London, United Kingdom}

\item[5)] \emph{Fakult\"at f\"ur Mathematik, Technische Universit\"at Dortmund,
44227 Dortmund, Germany}

\end{itemize}
\end{quote}

\section{Introduction}
Anderson localization refers to the physical phenomenon that waves do not propagate in certain types of disordered media.
The occurrence of this phenomenon depends on the energy associated with the wave and the specifics of the randomness present in the medium.
It was first studied in the context of propagation or diffusion of electron wavepackets in quantum mechanical models for condensed matter,
but can also occur in models for the propagation of classical waves.
More specifically, in the mathematical literature, \emph{spectral localization} in a certain energy region $I\subset \RR$ means that  the Hamiltonian governing the equation of motion of the waves has neither singular continuous nor absolutely continuous spectrum in $I$.
Thus, any spectrum in $I$ is pure point spectrum.
The term \emph{Anderson localization} describes the situation of spectral localization in $I$ where all eigenfunctions with eigenvalues in $I$ are exponentially decaying, even though this additional requirement is usually no restriction, in the sense that it occurs in most cases when spectral localization does.
Furthermore, there is the concept of \emph{dynamical localization} referring to the phenomenon that space-localized initial states do not spread, but remain localized uniformly in time when evolving as determined by the
Hamiltonian.
In most situations, whenever one can prove spectral localization one can prove dynamical localization as well. Concerning the other direction, spectral localization is a direct consequence of dynamical localization.
For a class of random Schr\"odinger operators the equivalence of these two notions has been proved in \cite{GerminetK-04}.
In order to understand the challenges posed by the models considered in this work let us emphasize that the random perturbation operators considered in \cite{GerminetK-04} are linear and monotone in the random variables and consist of multiplication operators whereas we shall consider a {more general} class of operators.

As a motivation for our results, let us start by noting that strictly speaking, the above mentioned notions of localization, which are ubiquitous in the literature, do not exclude the trivial situation where $I$ is a subset of the resolvent set.
At first glance, this seems to be a technical subtlety.
On the one hand, methods to exclude any spectrum in an interval $I$ are typically much simpler than methods to prove that the entire spectrum in $I$ is pure point.
On the other hand, proofs of localization often only apply in certain specific energy intervals, which are usually near the edges of the spectrum of a non-random operator and which are sometimes subject to technical conditions.
Showing that this region of localization indeed contains spectrum is a non-trivial task.
We have to perform a detailed analysis of the spectrum as a set, as well as of the region of localization, in order to show that their intersection is non-empty -- at least for certain parameter choices.

The operators which we consider in this work are random perturbations of a Schr\"odinger operator.
The perturbation however is general. It may be a differential, integral, multiplication or other operator, as long as it satisfies certain reasonable regularity conditions.
Furthermore the functional dependence on the random parameters may be non-linear and non-monotone.

Such perturbations are more difficult to analyse than the original Anderson or alloy type models with fixed-sign single site potentials, for which the theory of localization has been developed first.
Consequently, in subsequent developments, extensions of the methods have been invented in order to be able to deal with non-monotone, non-linear and non-potential perturbations.
The first step in this direction was the study of alloy type potentials with sign-changing single site potentials, as carried out in for instance in
\cite{Klopp-95a}, \cite{Stolz-00}, \cite{Veselic-01}, \cite{Klopp-02c}, \cite{Veselic-02a},
\cite{HislopK-02}, \cite{KostrykinV-06}, \cite{KloppN-09a}, \cite{Veselic-10a}.
In fact, the Wegner estimate which we prove in this paper relies on the vector field method of \cite{Klopp-95a,HislopK-02}.
More recently also the discrete analogue of this model was studied in
\cite{ElgartTV-10}, \cite{Veselic-10b}, \cite{ElgartTV-11},
\cite{TautenhahnV-10}, \cite{Krueger-12}, \cite{CaoE-12},
\cite{ElgartSS-14}, \cite{TautenhahnV-15}.
Another relevant model without direct monotonicity is a random potential given by a Gaussian stochastic
field with sign-changing covariance function,
c.f.~\cite{HupferLMW-01a}, \cite{Ueki-04}, \cite{Veselic-11}, \cite{Tautenhahn-19}.
Electromagnetic Schr\"odinger operators with random magnetic field
\cite{Ueki-94}, \cite{Ueki-00}, \cite{HislopK-02}, \cite{KloppNNN-03}, \cite{Ueki-08},
\cite{Bourgain-09},
\cite{ErdoesH-12a}, \cite{ErdoesH-12c}, \cite{ErdoesH-12b},
as well as Laplace-Beltrami operators with random metrics \cite{LenzPV-04}, \cite{LenzPPV-08},
\cite{LenzPPV-09} exhibit a non-monotonous parameter dependence which
affects the higher order terms of the differential operator.
Another model where geometric randomness enters naturally  is the random displacement model, cf.{} e.g.{}
\cite{Klopp-93,BakerLS-08,GhribiK-10,KloppLNS-12}. This is also the case for random waveguide Laplacians
which we discuss in the next paragraph.

The paper builds upon, improves and complements several earlier papers of two of us (partially with other coauthors).
In \cite{BorisovHEV-16} and \cite{BorisovHEV-18}
we have analysed what type of expansion rate for the spectrum is possible
for weak disorder random Hamiltonians.
For a large class of random operators we were able to show that the expansion is either
linear or quadratic, i.~e.~that other rates do not occur.
In the present paper we cover both of these two scenarios.
Initial length scale estimates for randomly bent and randomly curved  waveguides
and more general geometric perturbations have been proved in
\cite{BorisovV-11}, \cite{BorisovV-13}, and \cite{BorisovGV-16}
as well as  \cite{Borisov-17}, respectively.
However none of these paper proved Wegner estimates, which we provide here.

Let us list the main achievements of the paper together with the theorems where they are spelled out and the sections where they are proven.

\begin{enumerate}[(i)]
  \item
A characterisation of the almost sure spectrum as a set and its minimum, for a class of general weak disorder Hamiltonians.
(Theorem \ref{th2.1}, proof in Section~\ref{s:minimum}.)
  \item
Characterisation of linear and quadratic movement of the spectral minimum as a function of the weak disorder parameter.
(Corollary \ref{c:expansion})
  \item
An improved variational lower bound for the principal eigenvalue of Hamiltonians on large, finite segments
or cubes: The deviation of the random parameter configuration from the optimal one determines
explicitly how much the eigenvalue is lifted.
(Theorems \ref{thm:lower_bound_ground_state} and \ref{th4.1}, proof in Section~\ref{s:initial}.)
  \item
A Wegner estimate for general weak disorder random Hamiltonians
specifying an energy $\times$ disorder regime where it holds.
(Theorem~\ref{thm:Wegner_intro} and \ref{thm:Wegner}, proof in Section~\ref{s:Wegner}.)
  \item
Two results with explicit estimates on the location and size of a localization regime in the spectrum.
One of them corresponds to models with linear shift of the spectral minimum, the other
to models with a quadratic one, with respect to the small disorder parameter.
(Theorem~\ref{thm:localization_C1} and \ref{thm:localization_C2}.)
  \item
An broad list of examples which are covered by the general class or random Hamiltonians which we consider,
presented in Section~\ref{s:entire-space} and \ref{s:examples}.
\end{enumerate}

One of the complementary motivations of this paper is to understand \cite[Theorem 6.1]{HislopK-02}, where a Wegner estimate in the weak disorder regime is presented.
While we were able to generalize the proof to our situation,
we learned that it has to be complemented by a detailed analysis of the region in the energy$\times$disorder plane where the Wegner bound holds and an analysis about the expansion of the spectrum.
For certain models, like weak disorder magnetic fields
with no magnetic field in the unperturbed operator, the considered method of proof yields a Wegner estimate only in the resolvent set!
We shall discuss this in more detail for the specific example of a compactly supported magnetic field single site potential in  Section~\ref{ss:simple-magnetic} and Appendix \ref{ap:CaseIII}.

An important technical assumption in our models is that each single site perturbation acts on functions supported in one individual
periodicity cell. Overlapping single site perturbations would lead to additional
higher order terms, which would change the perturbation analysis which we carry out.

The proof of our theorems on localization is split into two parts.
First, one identifies an energy interval $I$, where it is possible to prove an appropriate Wegner estimate and an initial length scale estimate.
Thus in this region one can perform the multi-scale analysis.
Consequently, the spectrum in this energy region must be pure point spectrum with exponentially decaying eigenfunctions.
However, this does note yet exclude the situation that $I$ is part of the resolvent set.
For this reason we prove separately in Section~\ref{sec:MSA} that a nonempty subset $I_0\subset I$ belongs to the almost sure spectrum $\Sigma_\e$.
This is not trivial, because one has to analyse how fast the spectrum expands with emerging disorder $\e$.
At the same time, the method of proof we use for the Wegner estimate requires that one has a certain security distance away from the spectrum of the unperturbed operator $\Op^0$.
Thus, $I_0\subset \Sigma_\e$ is a set which is close, but not too close to the unperturbed spectrum $\Sigma_0$.

\section{Definitions and results for operators on multidimensional layers}
\label{s:results}

We shall first explain the abstract framework and our main results for operators on multidimensional layers before treating operators on $\RR^n$ in Section~\ref{s:entire-space}.
We fix a dimension $n \in \NN$ and write $x'=(x_1,\ldots,x_n)$ and $x=(x',x_{n+1})$ for Cartesian coordinates in $\RR^n$ and $\RR^{n+1}$, respectively.
For $d > 0$, we define $\Pi:=\{x:\, 0<x_{n+1}<d\} \subset \RR^{n+1}$, the multidimensional layer of width $d$.
Let $\G$ be a periodic lattice in $\RR^n$ with basis $e_1$, \ldots, $e_n$.
Its elementary cell is $\square':=\{x': x'=\sum\limits_{i=1}^{n} a_i e_i,\; a_i\in(0,1)\}$ and we also define $\square:=\square'\times(0,d) \subset \Pi$.

Let $\Pi\ni x\mapsto V_0(x)=V_0(x_{n+1})\in \RR$ be a measurable and bounded
potential depending only on the transversal variable $x_{n+1}$.
We consider the self-adjoint operator $\Op^0:=-\D +V_0$ in $L^2(\Pi)$ with
either Dirichlet or Neumann boundary conditions on $\p\Pi$, abbreviated by
\begin{equation}\label{eq:boundary-condition-B}
\cB u=0 \quad \text{where}\quad \cB u=u\quad \text{or}\quad \cB u=\frac{\p u}{\p x_{n+1}}.
\end{equation}
It is possible to choose different types of boundary conditions on the two (``upper'' and ``lower'') components of $\p\Pi$. The domain of this operator is
\begin{equation}\label{2.0}
\Dom(\Op^0)
:=\{u\in {\H2}(\Pi):\, \text{(\ref{eq:boundary-condition-B}) is satisfied on $\p\Pi$}\}.
\end{equation}

Let $T > 0$ be a fixed, positive parameter.
We introduce a family of linear operators $\cL(t)$, $t\in[-T,T]$.
These operators are assumed to act from $\H2(\square)$ into $L^2(\square)$ and are defined as
\begin{equation*}
\cL(t):=t \cL_1 + t^2 \cL_2 + t^3 \cL_3(t),
\end{equation*}
where $\cL_i\colon {\H2}(\square)\to L^2(\square)$ are bounded symmetric linear operators.
The operator $\cL_3(t)$ and its derivative with respect to $t$ are assumed to be bounded uniformly in $t\in[-T,T]$.
Recall that $\H2(\square)$ and $L^2(\square)$ are canonically embedded into $\H2(\Pi)$ and $L^2(\Pi)$, respectively.
For each $u\in H^2(\Pi)$, the restriction of this function on $\square$ is an element of $u\in H^2(\square)$.
Hence, the function $\cL_i u$ is a well-defined element of $L^2(\square)$ and
we may extend this function by zero in $\Pi\setminus\square$.
After the extension this function is an element of $L^2(\Pi)$.
In view of the described continuation,
in what follows, we regard the operators $\cL_i$ as unbounded operators in $L^2(\Pi)$ with domain $\H2(\Pi)$.

Let $\{ \xi_k \}_{k \in\ZZ^n}$ be a sequence of numbers with values in $[-1,1]$,
and denote by $\cS(k)$ the shift operator in $L^2(\Pi)$, i.e. $(\cS(k)u)(x)=u(x'+k,x_{n+1})$.
For $\e \in (0, T]$, we define the unbounded operator
\begin{equation}\label{2.2}
\Op^\e(\xi):=-\D +V_0+ \cL^\e(\xi),\quad \cL^\e(\xi):=\sum\limits_{k\in\G} \cS(k) \cL(\e \xi_k) \cS(-k),
\end{equation}
on the domain $\Dom(\Op^\e(\xi))=\Dom(\Op^0) \subset L^2(\Pi)$.

By $\|\cdot\|_{X\to Y}$ we denote the norm of a bounded operator acting from a Banach space $X$ into a Banach space $Y$.

\begin{convention}
To ensure a number of further properties of the operator $\Op^\e(\xi)$
it will be necessary to chose $\e$ from a smaller interval than $(0,T]$.
There is a (small) number $0<t_0<1$ depending only on
$n$, $\Gamma$, $\cB$, $V_0$,
$ \| \cL_1\|_{{\H2}(\square)\to L^2(\square)}$,
$ \| \cL_2\|_{{\H2}(\square)\to L^2(\square)}$,
$ \sup\limits_{-T\leq t\leq T} \| \cL_3(t)\|_{{\H2}(\square)\to L^2(\square)}$,
$ \sup\limits_{-T\leq t\leq T} \| \frac{\partial}{\partial t} \cL_3(t)\|_{{\H2}(\square)\to L^2(\square)}$
and the measure $\mu$ (which will be defined shortly) such that all results in
the paper  hold for $\e \in (0,t_0]$.
In proofs, the specific values of $t_0$ may change from line to line,
nevertheless we shall \emph{not} write each time \emph{possibly decreasing $t_0$ further}.
The main point is that $t_0$ does not depend on $\epsilon$ nor on the scale $N$ appearing below.
\end{convention}

If $t_0 > 0$ is sufficiently small, then
for every $\e \in (0, t_0]$ the operator $\cL^\e$ is
relatively bounded with respect to the Laplacian on $\Dom(\Op^\e(\xi))$ with relative bound smaller than one.
As a consequence, the operator $\Op^\e(\xi)$ is lower semi-bounded and, by the Kato-Rellich theorem, self-adjoint.

Let now $\om:=\{\om_k\}_{k\in\G}$ be a sequence of independent and identically distributed random variables on the probability space $(\Omega, \PP)$, where $\Om:=\times_{k\in\G} [-1,1]$, and $\PP=\bigotimes_{k\in\G} \mu$.
The probability measure $\mu$ corresponding to each $\omega_j$ is assumed to have an absolutely continuous probability density $h_0$ with respect to the Lebesgue measure and to satisfy
\begin{equation*}
-1\leqslant b=\min\supp\mu<\max\supp\mu=1.
\end{equation*}
We shall write $\EE$ for the expectation with respect to $\PP$.
Throughout the paper, the symbol $\Op^\e(\om)$ will denote a realization of our random operator corresponding to the configuration $\om\in \Omega$.
The notation $\Op^\e(\xi)$ will be used if we study properties which are valid for every single $\xi\in \Omega$.

We now define the auxiliary operator
\begin{equation*}
\Op^\d_\square:=-\Delta + V_0|_{\square} +\cL(\d),\quad \d\in[-t_0,t_0],
\end{equation*}
in $L^2(\square)$ with domain consisting of all functions in the Sobolev space $\H2(\square)$ satisfying the boundary condition (\ref{eq:boundary-condition-B}) on $\p\square\cap\p\Pi$ and periodic boundary conditions on  $\g:=\p\square\setminus\p\Pi$.
The operator $\Op^\d_\square$ is again self-adjoint.
For $\delta \in [- t_0, t_0]$, we 
denote by $\L^\d$ the lowest eigenvalue of the operator $\Op^\d_\square$ and by $\Psi^\d$ an associated, appropriately normalized eigenfunction.

In Lemma~\ref{lem:Taylor}, we prove that the eigenvalue $\L^\d$ and the eigenfunction $\Psi^\d$ {are, in a specific sense,} twice differentiable as functions of $\d$ in a neighbourhood of $\d = 0$ with Taylor expansions
\begin{equation}
 \label{eq:Taylor}
 \L^\d
 =
 \L_0 + \d \L_1 + \d^2 \L_2 + O(\d^3),
 \quad
 \text{and}
 \quad
 \Psi^\d
 =
 \Psi_0 + \d \Psi_1 + \d^2 \Psi_2 + O(\d^3),
\end{equation}
where $\Lambda_i$ and $\Psi_i$, $i \in \{0,1,2\}$, are explicitly given, cf. Lemma~\ref{lem:Taylor} for details.

In the present paper we cover the following two scenarios:
\begin{enumerate}\def\theenumi{\textbf{Case \Roman{enumi}}}
\item \label{C1}  The constant $\L_1$ is strictly negative.

\item \label{C2} The constant $\L_1$ is zero and
\begin{equation}\label{eq:definition-eta}
b>-1,\quad
 \eta:=-(b+1) \L_2 + (1-b)\big( \Phi_1-\Psi_1, \cL_1\Psi_0\big)_{L^2(\square)}>0.
\end{equation}
where $\Phi_1\in \H2(\square)$ is the unique solution to the problem
\begin{equation}\label{2.12b}
(-\Delta+V_0-\L_0)\Phi_1=-\cL_1\Psi_0\quad \text{in}\quad \square,\qquad \cB\Phi_1=0\quad\text{on}\quad \p\square\cap\p\Pi,\qquad \frac{\p\Phi_1}{\p\nu}=0\quad\text{on}\quad\g,
\end{equation}
orthogonal to $\Psi_0$ in $L^2(\square)$.
\end{enumerate}

\begin{remark}
  \label{r:CaseII}
  \begin{enumerate}
    \item
    \label{rem2.2}
    Note that \eqref{C2} implies that the vector $\cL_1 \Psi_0$ is orthogonal to $\Psi_0$ itself,
    cf.~Lemma~\ref{lem:Taylor},
    the unique ground state of $\Op^0_\square=-\Delta+V_0$.
    Hence, $\cL_1 \Psi_0$ is orthogonal to the entire spectral subspace of $\Lambda_0$.
    The operator $\Op^0_\square$ commutes with this subspace and is invertible on the orthogonal complement.
    Thus, there is indeed a unique solution $\Phi_1$ to (\ref{2.12b}).
    \item
    \label{rem2.3}
    It was shown in \cite{Borisov-17} that
    \begin{equation*}
    \big(\RE (\Phi_1-\Psi_1), \RE \cL_1\Psi_0\big)_{L^2(\square)}
    =
    \big(\Phi_1-\Psi_1, \cL_1\Psi_0\big)_{L^2(\square)},
    \end{equation*}
    thus the second inequality in condition \eqref{eq:definition-eta} can be rewritten as
    \begin{equation*}
    \eta:=-(b+1) \L_2 + (1-b)\big(\RE (\Phi_1-\Psi_1), \RE \cL_1\Psi_0\big)_{L^2(\square)}>0,
    \end{equation*}
    \item
    \label{rem2.4}
    Condition (\ref{eq:definition-eta}) is in some sense equivalent to $\L_2<0$. Namely, it was shown  in \cite{Borisov-17} that
    \begin{equation*}
    \big( \RE(\Phi_1-\Psi_1), \RE\cL_1\Psi_0\big)_{L^2(\square)} \leqslant 0.
    \end{equation*}
    Hence, the second term in the definition of $\eta$ is negative and we can satisfy (\ref{eq:definition-eta}) only if $\L_2<0$.
    On the other hand, if $\L_2<0$, then condition (\ref{eq:definition-eta}) is satisfied at least for $b$ close enough to $1$.
    \item In a companion paper we plan to discuss the case $\Lambda_1>0$.
    Currently we do not know how to treat the cases $\Lambda_1=0$, $\Lambda_2<0$ and $b=-1$, $\Lambda_1>0, \Lambda_2>0$.
  \end{enumerate}
\end{remark}

In the situation of~\eqref{C1} and~\eqref{C2} and provided that 
we establish in Lemma~\ref{lm:minimal-value} the following minimization property
\begin{equation*}
\forall \ \e \in (0, t_0]: \quad
\min\limits_{\d\in[\e b,\e]} \L^\d= \L^{\e}.
\end{equation*}
We 
use the corresponding ground state $\Psi^{\epsilon}$ to define a function $\rho^{\epsilon}$ on the lateral boundary $\g$, given by
\begin{equation*}
\rho^{\epsilon}:=\frac{1}{\Psi^{\epsilon}} \frac{\p\Psi^{\epsilon}}{\p\nu},
\end{equation*}
where $\nu$ is the outward normal vector.
Note that a priori, one might have to verify whether this definition of $\rho^{\epsilon}$ is admissible.
For instance, one needs to exclude that $\Psi^{\epsilon}$ vanishes on a nonempty open subset of $\gamma$.
The issue of well-definedness of $\rho^{\epsilon}$ will be addressed by Assumption~\ref{A1} below.
We shall also assume that for all $\e \in (0, t_0]$, the boundary term $\rho^{\epsilon}$ can be uniformly approximated by polynomial expressions in boundary terms derived from $\Psi_0$, $\Psi_1$, and $\Psi_2$, that appear in the Taylor expansion of Lemma \ref{lem:Taylor}.
This is made precise in the following assumption, that is satisfied for many examples as discussed in Section~\ref{s:examples}.

\setcounter{assumption}{12}
\begin{assumption}
  \label{A1}
  For all $\e \in [0,t_0]$ the function $\rho^{\epsilon}$ is piecewise continuous on $\g$.
  Furthermore,
  \begin{itemize}
   \item
   If \eqref{C1} holds, then the asymptotic identity
    \begin{equation}\label{2.12c}
    \sup\limits_{\overline{\g}} \big|\rho^{\epsilon}-\epsilon\rho_1\big|=O(\e^2)
    \end{equation}
    holds true, where
    \begin{equation}\label{2.12d}
    \rho_1:=\frac{1}{\Psi_0} \frac{\p\Psi_1}{\p\nu}
    \end{equation}
   is a piecewise continuous function on $\g$.
   \item
    If \eqref{C2} holds, then the asymptotic identity
    \begin{equation*}
    \sup\limits_{\overline{\g}} \big|\rho^{\epsilon}-\epsilon\rho_1-\epsilon^2\rho_2\big|=O(\e^3)
    \end{equation*}
    holds true, where $\rho_1$ is given by (\ref{2.12d}),
    \begin{equation*}
    \rho_2:=\frac{1}{\Psi_0} \frac{\p\Psi_2}{\p\nu} - \frac{\Psi_1}{\Psi_0^2} \frac{\p\Psi_1}{\p\nu},
    \end{equation*}
    and
    the functions $\rho_1$, $\rho_2$ are piecewise continuous on $\g$.
   \end{itemize}
\end{assumption}

Since $\Psi^{\epsilon}$ satisfies periodic boundary conditions, the function $\rho^{\epsilon}$ inherits this periodicity.
More precisely, after periodically extending $\rho^{\epsilon}$ from $\gamma$ to $\bigoplus_{k \in \Gamma} S(k) \g $ the values of $\rho^{\epsilon}$ on two touching boundaries of elementary cells of $\Gamma$ will sum up to zero.
The function $\rho^{\epsilon}$ will be used to define restrictions of the operator $\Op^\e(\om)$ onto finite parts of the layer $\Pi$ with so-called \emph{Mezincescu boundary conditions}.
These are a special choice of Robin boundary conditions, which ensures that the ground state of finite volume restrictions coincides with the infimum of the spectrum of the infinite volume operator.

Our first result describes the location of the spectrum of $ \spec(\Op^\e(\om))$.

\begin{theorem}\label{th2.1}
Let $t_0 > 0$ be sufficiently small.
For all $\e \in [0, t_0]$ there exists a closed set $\Sigma_\e\subset\RR$ such that
\begin{equation*}
\spec(\Op^\e(\om))=\Sigma_\e\quad\PP-a.s.
\end{equation*}
This set $\Sigma_\e$ is equal to the closure of the union of spectra for all periodic realizations of $\Op^\e(\cdot)$:
\begin{equation}\label{2.19}
\Sigma_\e=\overline{\bigcup\limits_{N\in \NN} \quad \bigcup\limits_{\xi \ \text{is $2^N \G$-periodic}} \spec\big(\Op^\e(\xi)\big)},
\end{equation}
where the second union is taken over all sequences $\xi: \G\to\supp \mu $, which are periodic with respect to   the sublattice $2^N \G:=\{2^N q:\, q\in\G\}$.

In particular, if we are in \eqref{C1} or in \eqref{C2} and Assumption~\ref{A1} holds, then, assuming $t_0 > 0$ sufficiently small,
we have
\begin{equation}\label{2.20}
\min\Sigma_\e=\L^{\epsilon}.
\end{equation}
\end{theorem}

A consequence of this theorem and Lemmata \ref{lem:Taylor} and \ref{lm:minimal-value} is
\begin{corollary}\label{c:expansion}
Let $t_0 > 0$ be sufficiently small.
If \eqref{C1} holds, then for all $\e \in (0, t_0]$
\[
  \min\Sigma_\e =\L^\e\leq \Lambda_0 - c \e, \quad \text{where}\quad c = - \Lambda_1/2>0 .
\]
If \eqref{C2} holds, then for all $\e \in(0,t_0]$ we have
\[
  \min\Sigma_\e=\L^\e \leq \Lambda_0 - c \e^2,\quad \text{where}\quad c = - \Lambda_2/2>0 .
\]
\end{corollary}

This ensures that the random perturbation will indeed create new spectrum below the one of the unperturbed operator $\Op^0$.
The newly created spectrum expands at least by order $\e$ or $\e^2$ away from the original spectrum, respectively.
This is crucial to  make sure that our results below on localization are not trivial.

Having proved a lower bound on the expansion of the spectrum below $\L_0$, we then identify a region where we can prove localization.
The next theorem is the first ingredient for this purpose.
It provides a lower bound on the distance of the ground state energy of finite volume restrictions $\Op^{\epsilon}(\xi)$ (with Mezincescu boundary conditions) to the  minimum
over \emph{all} configurations $\Lambda^{\epsilon}$ in terms of an average over expressions containing the $\xi_k$.

\begin{theorem}[See Theorem \ref{th4.1}] \label{thm:lower_bound_ground_state}
For $N\in\NN$, we denote by  $\Pi_{N}$ the open interior of the set
\begin{align*}
&\tilde\Pi_{N}:=\{x\in \RR^{n+1}:\, x'=\sum\limits_{j=1}^{n} a_j e_j,\ a_j\in[0,N),\ 0<x_{n+1}<d\},
\end{align*}
$\G_{N}:= \tilde\Pi_{N}\cap \Gamma$, and by $\L_{N}^\e(\xi)$ the lowest eigenvalue of the operator
\begin{equation*}
\Op_{N}^\e(\xi):=-\D +V_0+ \cL_{N}^\e(\xi),\quad \cL_{N}^\e(\xi):=\sum\limits_{k\in\G_{N}} \cS(k) \cL(\e\xi_k) \cS(-k)
\end{equation*}
on $\Pi_{N}$ subject to boundary condition (\ref{eq:boundary-condition-B}) on $\p\Pi\cap\p\Pi_{N}$ and to the Mezincescu boundary condition
$\frac{\p u}{\p\nu}= \rho^{\epsilon}u$ on $ \g_{N}:=\p\Pi_{N}\setminus\p\Pi$.

Let $t_0$ be sufficiently small.
Then there exist constants $N_1\in \NN$ and $c_0>0$,
depending exclusively on the operators $\cL(t)$, $t\in[-T,T]$,
and on $V_0$,
 such that for all $N \in \NN$ with $N \geq N_1$,
$\xi\in \Omega$, and $\e \in (0, t_0]$, we have
\begin{align*}
&\L_{N}^\e(\xi)-\L^{\epsilon}\geqslant \frac{\e |\L_1|}{4 N^n} \sum\limits_{k\in\G_{N}} (1-\xi_k) \quad \text{for}\ \e<c_0 N^{-2}\ \text{in \eqref{C1}},
\\
&\L_{N}^\e(\xi)-\L^{\epsilon}\geqslant \frac{\eta \e^2}{N^n} \sum\limits_{k\in\G_{N}} (1-\xi_k) \quad\hphantom{11} \text{for}\ \e<c_0 N^{-4}\ \text{in \eqref{C2}},
\end{align*}
where $\L_1$ is given in Lemma~\ref{lem:Taylor} and $\eta$ is given in~\eqref{eq:definition-eta}.
\end{theorem}

It is a canonical step to combine this theorem with large deviation estimates, and so-called Combes-Thomas estimates to deduce a so-called initial scale estimate.
Initial scale estimates are probabilistic bounds on the decay of the resolvent of finite volume restrictions of $\Op^\e(\om)$.
They serve as an induction anchor in the multi-scale analysis.
The initial scale estimates can be found in Theorem~\ref{thm:ISE} in Section~\ref{s:initial}.

The second ingredient in a multi-scale analysis proof of localization are Wegner estimates.
They bound the probability of finding an eigenvalue of a finite volume restriction of $\Op^{\e}(\om)$ in an interval and serve as a non-resonance condition in the induction step in the multi-scale analysis.
The next theorem is a Wegner estimate in an $\e$-dependent regime below $\L_0$.
\begin{theorem}[See Theorem~\ref{thm:Wegner_with_epsilon}]
\label{thm:Wegner_intro}
  Assume that $t_0 > 0$ is sufficiently small.
  There are $D$ depending exclusively on the operators $\cL(t), t\in [-T,T]$,
  $C_{n,h_0}$ depending exclusively on $n$ and $h_0$,
  as well as $C_{n,V_0}$ depending merely on $n$, $V_0$, and the lattice $\Gamma$,
  such that for all $\e \in (0, t_0]$, all $\alpha \in \Gamma$, and all $N \in \NN$
  the following hold:
\begin{enumerate}[(i)]
  \item
  For all $E \leq \Lambda_0 - D \e^2$ and  all $\kappa \leq D \e^2/ 4$ we have
  \begin{align*}
    \PP ( \dist(\sigma(\Op_{\a,N}^\e(\omega)), E) \leq \kappa )
    \leq
    \frac{C_{n,h_0}}{D \e^2}  [1+C_{n,V_0}|\square| N^n]\cdot \kappa \, N^{n}.
  \end{align*}
  \item
  Assume that $\cL_2 \leq 0$.
  Then for all $E \leq \Lambda_0 - D \e^3$ and all $\kappa \leq D \e^3/ 4$,  we have
  \begin{equation*}
    \PP ( \dist(\sigma(\Op_{\a,N}^\e(\omega)), E) \leq \kappa )
    \leq
    \frac{C_{n,h_0}}{D \e^3}  [1+C_{n,V_0}|\square| N^n]\cdot \kappa \, N^{n}.
  \end{equation*}
 \end{enumerate}
\end{theorem}

It is paramount to notice that in Theorem~\ref{thm:Wegner_intro} the intervals $[\Lambda^{\e}, \Lambda_0 - D \e^2]$ or $[\Lambda^{\e}, \Lambda_0 - D \e^3]$, respectively,  from which $E$ is chosen, must have a sufficiently large distance to $\L_0$.
These intervals can only be non-empty if $\e\mapsto\L^{\e}$ expands (anti-tonically) at a sufficient fast rate below $\L_0$.
It is Corollary \ref{c:expansion} that ensures this.

The initial scale estimate and the Wegner estimate are the crucial ingredients to start the multi-scale analysis as in~\cite{GerminetK-03} and conclude spectral and dynamical localization.
To implement this one needs actually a number of several additional, rather general properties. We show in Section~\ref{sec:MSA} that they hold, provided the following assumption is satisfied.

\setcounter{assumption}{17}
\begin{assumption}
\label{B1+2}
\begin{enumerate}\def\theenumi{\textbf{R.{\arabic{enumi}}}}
\item
\label{B1}
There exists a constant $c_6>0$ such that for any function $\vp\colon \overline{\square'} \to [0,1]$, $\vp\in C^\infty(\overline{\square'})$, there exists a constant $c_7 = c_7(\|\vp\|_{C^2(\overline{\square})})$
satisfying the estimate
\begin{multline*}
\text{for all }\ u\in \H2(\square), i=1,2,3:\\
\sup_{t\in[-T,T]}\big|(\cL_i(t) u,\vp u)_{L^2(\square)}\big|\leqslant c_6 (\vp \nabla u,\nabla u)_{L^2(\square)} + c_7 \|u\|_{L^2(\square)}^2.
\end{multline*}
For $i=1,2$, the above supremum can be omitted.
The constant $c_7$ is allowed to depend on $\vp$ only through its norm $\|\vp\|_{C^2(\overline{\square})}$ in such a way that
$\|\vp\|_{C^2(\overline{\square})}\mapsto c_7 $  is non-decreasing.
\item
\label{B2}
For any function $\vp\colon \overline{\square'} \to [0,1]$, $\vp\in C^\infty(\overline{\square'})$, there exists a constant
$c_8 = c_8(\|\vp\|_{C^2(\overline{\square})})$ such that the estimate
\begin{equation*}
\text{for all } \ u\in \H2(\square), i=1,2,3: \
\sup_{t\in[-T,T]}\|\cL_i(t) \vp u -\vp \cL_i(t) u\|_{L^2(\square)}\leqslant c_8\|u\|_{W^{1,2}(\square)}
\end{equation*}
holds true, where for  $i=1,2$, the above supremum can be omitted.
The constant $c_8$ is allowed to depend on $\vp$ only through its norm $\|\vp\|_{C^2(\overline{\square})}$ in such a way that
$\|\vp\|_{C^2(\overline{\square})}\mapsto c_8 $  is non-decreasing.
\end{enumerate}
\end{assumption}
Finally, this allows us formulate our results on Anderson localization in a small neighbourhood of $\Sigma_\e$.

\begin{theorem}\label{thm:localization_C1}
Assume that we are in \eqref{C1}, Assumptions~\ref{A1}, \ref{B1+2} hold, and $t_0 > 0$ is sufficiently small.
Then there exists $C > 0$ such that for all $\e\in(0,t_0]$ the set
$\Sigma_\e\cap[\min\Sigma_\e,\min\Sigma_\e+C \e]$
is almost surely non empty and exhibits dynamical localization.
\end{theorem}

\begin{theorem}\label{thm:localization_C2}
Assume that we are in \eqref{C2}, Assumptions~\ref{A1}, \ref{B1+2} hold, and $t_0 > 0$ is sufficiently small.
Assume also that the operator $\cL_2\leqslant 0$ is non-positive.
Then there exists $C > 0$ such that for all $\e\in(0,t_0]$ the set $\Sigma_\e\cap[\min\Sigma_\e,\min\Sigma_\e+C \e^2]$
is almost surely non empty and exhibits dynamical localization.
\end{theorem}

 The proof of Theorems~\ref{thm:localization_C1} and~\ref{thm:localization_C2} follows the classical strategy of multi-scale analysis proofs of localization, but some extra attention is required due to the interplay between the scale $N$ and the disorder $\e$.
 These details are explained in Section~\ref{sec:MSA}.

\section{Reformulation for operators acting on the entire space}
\label{s:entire-space}

We now explain the necessary modifications in order to also treat operators living on the whole space $\RR^n$.
We use the same approach as in \cite[Sect. 3.7]{BorisovGV-16}.
Let us define the operators
\begin{equation*}
\cL'(t):=t \cL'_1 + t^2 \cL'_2 + t^3 \cL'_3(t),
\end{equation*}
where $\cL'_i: {\H2}(\square')\to L^2(\square')$ are bounded symmetric linear operators and the operator $\cL'_3(t)$ as well as its derivative with respect to $t$ are bounded uniformly in $t\in[-t_0,t_0]$.
We then introduce in $L^2(\RR^n)$ the operator
\begin{equation*}
{\Op'}^\e(\xi):=-\D_{x'}+ \sum\limits_{k\in\G} \cS '(k) \cL'(\e \xi_k) \cS '(-k)
\end{equation*}
where $\D_{x'}$ denotes the Laplacian in $\RR^n$ and $\cS' (k)$ is the shift operator: $(\cS' (k)u)(x')=u(x'+k)$.
The operator ${\Op'}^\e(\om)$ is a random self-adjoint operator in $L^2(\RR^n)$.

In order to use results of the previous section, we extend this construction from $\RR^n$ to the layer $\Pi \subset \RR^{n+1}$ where we set its height $d = \pi$ and impose Neumann boundary conditions on $\p\Pi$, i.e.
\begin{equation*}
\mathcal{B}u=\frac{\p u}{\p x_{n+1}}.
\end{equation*}
More precisely, we extend the operators $\cL_i'$ constantly to the {$(n+1)$th} dimension, i.e. we extend them to operators $\cL_i: {\H2}(\square)\to L^2(\square)$ via
\begin{equation*}
(\cL_i u)(x',x_{n+1})=\big(\cL'_i u(\cdot,x_{n+1})\big)(x').
\end{equation*}
This allows to introduce $\Op^\e(\xi)$ in $L^2(\Pi)$ as before by (\ref{2.2}) with $V_0 = 0$.
Now, we have in particular that the lowest eigenvalue $\L_0$ of
\[
\Op^\bot:=\frac{\drm^2\hphantom{x_{n+}}}{\drm x_{n+1}^2}
 \quad
 \text{on}
 \quad
 (0,\pi)
\]
with Neumann boundary condition is $\L_0=0$, see Section~\ref{s:minimum}.
The associated eigenfunction $\Psi_0$, extended to $\square$, and with the normalization as in Section~\ref{s:minimum} is simply a constant function
\begin{equation*}
\Psi_0=\frac{\mathbf{1}}{\sqrt{\pi|\square'|}},\quad \mathbf{1}(x')\equiv 1.
\end{equation*}

The spectral properties of
$\Op^\e(\xi)$ can be analysed by separation of variables ($x'$ and $x_{n+1}$).
More precisely we have
\begin{equation}\label{5.24b}
\Op^\e(\xi)={\Op'}^\e\otimes\Op^\bot,
\end{equation}
This implies
\begin{equation}\label{5.24a}
\spec(\Op^\e(\xi))=\bigcup\limits_{m=0}^{\infty} \spec({\Op'}^\e(\xi))+m^2.
\end{equation}
and that $\Op^\e(\xi) \chi_I(\Op^\e(\xi) )$
and ${\Op'}^\e(\xi) \chi_I({\Op'}^\e(\xi))$ are unitarily equivalent
for
\[
I:=[\min\spec({\Op'}^\e(\xi)),\min\spec({\Op'}^\e(\xi))+1].
\]
Since Theorem~\ref{th2.1}, Lemma~\ref{lem:Taylor}, and Lemma~\ref{lm:minimal-value}
concern only properties in a small neighbourhood of the spectral bottom, they
immediately imply analogous statements for the operator ${\Op'}^\e(\xi)$.

More precisely, the analogue of the auxiliary operator $\Op_\square^\d$ is introduced as
\begin{equation*}
{\Op'}^\d_{\square'}:=-\Delta + \cL'(\d),\quad \d\in[-\e,\e],
\end{equation*}
in $L^2(\square')$ subject to periodic boundary conditions on $\p\square'$.
By $\L'_\delta$ we denote again the lowest eigenvalue of the operator ${\Op'}^\d_\square$ and by ${\Psi'}^\d$ an appropriately normalized associated eigenfunction.
Similarly as in the previous section, we can expand them in a Taylor series
\[
 {\L'}^\d
 =
 \d {\L_1'} + \d^2 {\L_2'} + O(\d^3)
 \quad
 \text{and}
 \quad
 \Psi^\d(x')=\frac{1}{\sqrt{|\square'|}} \left( \mathbf{1}+\d\Psi'_1(x')+\d^2\Psi'_2(x')\right)+O(\d^3)
\]
around $\d = 0$.
The expressions $\L_i'$ and $\Psi_i'$, $i \in \{1,2\}$ are given by analogous formulas as the corresponding $\L_i$ and $\Psi_i$ in Lemma~\ref{lem:Taylor}.
For convenience, we explicitly formulate these expressions in Lemma~\ref{lm3.2.1}.

Analogously to~\eqref{C1} and~\eqref{C2} we treat the following cases:
\begin{enumerate}\def\theenumi{\textbf{Case \Roman{enumi}'}}
\item \label{C1'}  The constant $\L'_1$ is strictly negative.

\item \label{C2'} The constant $\L'_1$ is zero and
\begin{equation*}
b>-1,\quad \eta:=-(b+1) \L'_2 + (1-b) \big(\RE (\Phi'_1-\Psi'_1), \RE \cL'_1\mathbf{1}\big)_{L^2(\square')}>0,
\end{equation*}
where $\Phi_1' \in \H2(\square')$ is the unique solution to the problem
\begin{equation*}
-\Delta\Phi_1' = -\cL_1 \mathbf{1}\quad \text{in}\quad \square',\qquad \frac{\p\Phi'_1}{\p\nu}=0\quad\text{on}\quad\p\square',
\end{equation*}
orthogonal to $\mathbf{1}$ in $L^2(\square')$.
\end{enumerate}

Lemma~\ref{lm:minimal-value} carries over verbatim which means that for sufficiently small $t_0$
we have
\[
 \forall \e \in (0, t_0] \colon
 \quad
 \min_{\delta \in [\e b, \e]} {\L'}^\d = {\L'}^\e.
\]
We also introduce on the lateral boundary $\p\square'$ the analogue of the function $\rho^{\epsilon}$:
\begin{equation*}
{\rho'}^{\epsilon}:=\frac{1}{{\Psi'}^{\epsilon}} \frac{\p{\Psi'}^{\epsilon}}{\p\nu},
\end{equation*}
where $\nu$ is the outward normal.

\setcounter{assumpt}{12}
The place of Assumption~\ref{A1} will be taken by
\begin{assumpt}
  \label{A1'}
  For all $\e \in [0,t_0]$ the function ${\rho'}^{\epsilon}$ is piecewise continuous on $\p\square'$.
  Furthermore,
  \begin{itemize}
   \item
    If \eqref{C1'} holds, then the asymptotic identity
    \begin{equation*}
    \sup\limits_{\overline{\p\square'}} \big|{\rho'}^{\epsilon}-\epsilon\rho'_1\big|=O(\e^2)
    \end{equation*}
    holds true, where
    \begin{equation}\label{3.2.12d}
    \rho'_1:=\frac{\p\Psi'_1}{\p\nu},
    \end{equation}
    {is piecewise continuous on $\p\square'$.}
   \item
    If \eqref{C2'}  holds, then the asymptotic identity
    \begin{equation*}
    \sup\limits_{\overline{\p\square'}} \big|{\rho'}^{\epsilon}-\epsilon\rho'_1-\epsilon^2\rho'_2\big|=O(\e^3)
    \end{equation*}
    holds true, where $\rho'_1$ is given by (\ref{3.2.12d}),
    \begin{equation*}
    \rho_2:= \frac{\p\Psi'_2}{\p\nu} - \Psi'_1 \frac{\p\Psi'_1}{\p\nu},
    \end{equation*}
    and $\rho'_1$, $\rho'_2$ are piecewise continuous on $\p\square'$.
   \end{itemize}
\end{assumpt}

\setcounter{assumpt}{17}
\begin{assumpt}\label{primeB1+2}
Assumptions~\ref{B1} and~\ref{B2} remain almost unchanged: we just replace $\square$ by $\square'$ in their formulation.
\end{assumpt}

Combining identity~\eqref{5.24a} with the results of the previous section, we arrive at the following analogous results on the almost sure spectrum and on Anderson localization near the bottom of the spectrum of ${\Op'}^\e(\om)$.

\begin{theorem}
Let $t_0 > 0$ be sufficiently small.
Then, for all $\e \in [0, t_0]$ there exists a closed set $\Sigma'_\e$ such that
\begin{equation*}
\spec({\Op'}^\e(\om))=\Sigma'_\e\quad\PP-a.s.
\end{equation*}
The set $\Sigma'_\e$ is equal to the closure of the union of spectra for all periodic realizations of ${\Op'}^\e(\om)$:
\begin{equation*}
\Sigma'_\e=\overline{\bigcup\limits_{N\in \NN} \quad \bigcup\limits_{\xi \ \text{is $2^N \G$-periodic}} \spec\big({\Op'}^\e(\xi)\big)},
\end{equation*}
where the second union is taken over all sequences $\xi: \G\to\supp \mu $, which are periodic  with respect to   the sublattice $2^N \G:=\{2^N q:\, q\in\G\}$.

Assume that we are in \eqref{C1'} or in \eqref{C2'} and Assumption~\ref{A1'} holds, then, assuming $t_0 > 0$ sufficiently small, we have
\begin{equation*}
\min\Sigma'_\e={\L'}^{\epsilon}.
\end{equation*}
\end{theorem}

\begin{corollary}
Let $t_0 > 0$ be sufficiently small.
If \eqref{C1'} holds, then for all $\e \in (0, t_0]$
\[
   \inf \Sigma'_\e \leq \inf \Sigma'_0 - c \e,
   \quad
   \text{where}
   \quad
   c = - \L'_1/2 > 0.
\]
If \eqref{C2'} holds, then for all $\e \in(0,t_0]$ we have
\[
 \inf \Sigma'_\e \leq \inf \Sigma'_0 - c \e^2
   \quad
   \text{where}
   \quad
   c = - \L'_2/2 > 0.
\]
\end{corollary}

We also obtain the following analogue of Theorem~\ref{thm:lower_bound_ground_state}:

\begin{theorem}
 \label{thm:lower_bound_ground_state'}
For $N\in\NN$, we denote by  $\square'_N$ the open interior of the set
\begin{align*}
&\square'_N:=\{x'\in \RR^{n}:\, x'=\sum\limits_{j=1}^{n} a_j e_j,\ a_j\in[0,N)\},
\end{align*}
$\G_{N}:= {\square'}_{N} \cap \Gamma$, and by ${\L'}_{N}^\e(\xi)$ the lowest eigenvalue of the operator
\begin{equation*}
{\Op'}_{N}^\e(\xi):=-\D_{x'} {\cL'}_{N}^\e(\xi),\quad {\cL'}_{N}^\e(\xi):=\sum\limits_{k\in\G_{N}} \cS(k) \cL'(\e\xi_k) \cS(-k)
\end{equation*}
on $\square'_N$ subject to the Mezincescu boundary condition
$\frac{\p u}{\p\nu}= \rho'^{\epsilon}u$ on $\p \square'_N$.

Let $t_0 > 0$ be sufficiently small. Then
there exist constants $N_1\in \NN$ and $c_0>0$,
depending exclusively on the operators $\cL'(t)$,  $t \in [-T,T]$, such that for all $N \in \NN$ with $N \geq N_1$, $\xi\in \Omega$, and $\e \in (0, t_0]$,  we have
\begin{align}
&{\L'}_{N}^\e(\xi)-\L'^{\epsilon}\geqslant \frac{\e |\L_1|}{4 N^n} \sum\limits_{k\in\G_{N}} (1-\xi_k) \quad \text{for}\ \e<c_0 N^{-2}\ \text{in \eqref{C1}},
\\
&{\L'}_{N}^\e(\xi)-\L'^{\epsilon}\geqslant \frac{\eta \e^2}{N^n} \sum\limits_{k\in\G_{N}} (1-\xi_k) \quad\hphantom{11} \text{for}\ \e<c_0 N^{-4}\ \text{in \eqref{C2}}
\end{align}
\end{theorem}

Now, we can again combine Theorem~\ref{thm:lower_bound_ground_state'} with Combes-Thomas estimates and deduce a corresponding initial scale estimate for the operator ${\Op'}^\e(\omega)$.
Finally, we note that the Wegner estimate of
Theorem~\ref{thm:Wegner_intro} can also be proved for the operator ${\Op'}^\e(\omega)$.
This is clear from the proof in Section~\ref{s:Wegner} since the fact that
Theorem~\ref{thm:Wegner_intro} is formulated on the layer $\Pi$ only enters in the volume bound.
We conclude:

\begin{theorem}
  Assume that $t_0 > 0$ is sufficiently small.
  There are $D$ depending exclusively on the operators $\cL'(t)$, $t\in [-T,T]$,
    $C_{n,h_0}$ depending exclusively on $n$, and $h_0$,
  as well as $C_{n}$ depending merely on $n$, and the lattice $\Gamma$,
  such that for all $\e \in (0, t_0]$, all $\alpha \in \Gamma$, and all $N \in \NN$
  the following hold:
\begin{enumerate}[(i)]
  \item
  For all $E \leq \Lambda_0 - D \e^2$ and  all $\kappa \leq D \e^2/ 4$ we have
  \begin{align*}
    \PP ( \dist(\sigma(\Op_{\a,N}^\e(\omega)), E) \leq \kappa )
    \leq
    \frac{C_{n,h_0}}{D \e^2}  [1+C_{n}|\square'| N^n]\cdot \kappa \, N^{n}.
  \end{align*}
  \item
  Assume that $\cL_2 \leq 0$.
  Then for all $E \leq \Lambda_0 - D \e^3$ and all $\kappa \leq D \e^3/ 4$,  we have
  \begin{equation*}
    \PP ( \dist(\sigma(\Op_{\a,N}^\e(\omega)), E) \leq \kappa )
    \leq
    \frac{C_{n,h_0}}{D \e^3}  [1+C_{n}|\square'| N^n]\cdot \kappa \, N^{n}.
  \end{equation*}
 \end{enumerate}
\end{theorem}

In conclusion, we have again all ingredients needed to start the multi-scale analysis and prove Anderson localization near the bottom of $\Sigma'_\e$ for sufficiently small $\e > 0$.

\begin{theorem}\label{thm:localization_C1'}
Assume that we are in \eqref{C1'}, Assumptions~\ref{A1'} and \ref{primeB1+2} hold,  and $t_0 > 0$ is sufficiently small.
Then there exists $C > 0$ such that for all $\e\in(0,t_0]$ the set
$\Sigma'_\e\cap[\L'^{\epsilon},\L'^{\epsilon}+C\e]$ is almost surely non empty and only contains pure point spectrum with exponentially decaying eigenfunctions.
\end{theorem}
\begin{theorem}\label{thm:localization_C2'}
Assume that we are in \eqref{C2'}, Assumptions~\ref{A1'} and \ref{primeB1+2} hold,
and $t_0 > 0$ is sufficiently small.
Assume also that the operator $\cL'_2 $ is non-positive.
Then there exists $C > 0$ such that for all $\e\in(0,t_0]$ the set $\Sigma'_\e\cap[{\L'}^{\epsilon},{\L'}^{\epsilon}+C\e^2]$ is almost surely non empty and only contains pure point spectrum with exponentially decaying eigenfunctions.
\end{theorem}

\section{Examples}
\label{s:examples}

All examples we present, concern operators defined on the layer
as in Section~\ref{s:results}, but immediately extended by formula (\ref{5.24b})
to the full space setting on $\RR^{n}$ as in Section~\ref{s:entire-space}.

\subsection{Random potential} Our first example is the classical perturbation by a random potential. Here the operator $\cL(t)$ is the multiplication by a potential $W(t)$, which is defined as
\begin{equation*}
W(t)=t W_1 + t^2 W_2
\end{equation*}
where $W_i=W_i(x)$ are bounded measurable real functions.
>From Lemma~\ref{lem:Taylor}, we infer that the constant $\L_1$ is given by
\begin{equation*}
\L_1=\int\limits_{\square} W_1(x)|\Psi_0(x_{n+1})|^2\di x.
\end{equation*}
If $\L_1<0$, we are in \eqref{C1}.

If $\L_1 = 0$, then we infer from Lemma~\ref{lem:Taylor} that
\[
 W_1 \Psi_0
 =
 - (\Op_\square^0 - \L_0) \Psi_1
\]
whence
\begin{align*}
\L_2
&=
\int\limits_{\square} W_2(x)\Psi_0^2(x_{n+1})\di x
+
\int\limits_{\square} W_1(x) \Psi_0 \Psi_1 \di x\\
&=
\int\limits_{\square} W_2(x)\Psi_0^2(x_{n+1})\di x
-
\int\limits_{\square} \Psi_1  (\Op_\square^0 - \L_0) \Psi_1 \di x
<
\int\limits_{\square} W_2(x)\Psi_0^2(x_{n+1})\di x
\end{align*}
since $\Op_\square^0 - \L_0$ is a nonnegative operator and $\Psi_1$ is not its (unique) ground state.
We conclude that if we assume $W_2 \leq 0$, as needed in Theorem~\ref{thm:localization_C2}, then we automatically have $\L_2 < 0$ and are in \eqref{C2} provided
\[
\eta
=
-(b+1)\L_2+(1-b)\int\limits_{\square} W_1(x)(\Phi_1(x)-\Psi_1(x))\Psi_0(x_{n+1})\di x > 0.
\]

Assumption~\ref{A1} is satisfied once the potentials $V_0$ and $W_i$ are smooth enough, for instance $V_0, W_i\in C^q(\overline{\square})$ for some $q\in(0,1)$. In this case by standard Schauder estimates we get that $\Psi^\d$, $\Psi_1$, $\Psi_2$ belong to the H\"older space $C^{2+q}(\overline{\square})$ and the asymptotics (\ref{2.9})  holds in the sense of the norm in this space.
 Assumptions~\ref{B1},~\ref{B2} are obviously true for the considered example.

\subsection{Integral operator} Our next example is a random integral operator corresponding to
\begin{equation*}
(\cL_i u)(x)=\int\limits_{\square} K_i(x,y)u(y)\di y,\quad i=1,2,
\end{equation*}
where $K_i\in L^2(\square\times\square)$ are symmetric kernels:
\begin{equation*}
K_i(y,x)=\overline{K_i(x,y)}.
\end{equation*}
Here, the constant $\L_1$ is given by the integral
\begin{equation*}
\L_1=\int\limits_{\square\times\square} K_1(x,y)\Psi_0(x_{n+1})\Psi_0(y_{n+1})\di x\di y,
\end{equation*}
and if $\L_1<0$, we are in \eqref{C1}.

If $\L_1=0$, conditions (\ref{2.12}), (\ref{2.12b}) take the form
\begin{equation*}
(\Op^0_\square-\L_0)\Psi_1=-\int\limits_{\square} K_1(\cdot,y)\Psi_0(y_{n+1})\di y,
\end{equation*}
and
\begin{align*}
&(-\Delta+V_0-\L_0)\Phi_1=-\int\limits_{\square} K_1(\cdot,y)\Psi_0(y_{n+1})\di y \quad \text{in}\quad \square,
\\
&\cB\Phi_1=0\quad\text{on}\quad \p\square\cap\p\Pi,\qquad \frac{\p\Phi_1}{\p\nu}=0\quad\text{on}\quad\g.
\end{align*}
The constants $\L_2$ and $\eta$ read in this case
\begin{align*}
&\L_2=\int\limits_{\square\times\square} K_2(x,y) \Psi_0(x_{n+1})\Psi_0(y_{n+1})\di x\di y,
\\
&\eta=-(b+1)\L_2 + (1-b)\int\limits_{\square\times\square} K_1(x,y) (\Phi_1(y)-\Psi_1(y))\Psi_0(x_{n+1})\di x\di y.
\end{align*}
For Theorem~\ref{thm:localization_C2}, the kernel $K_2$ needs to satisfy furthermore the inequality
\begin{equation} \label{eq:positive-definite-K}
\int\limits_{\square\times\square} K_2(x,y)u(x)\overline{u(y)}\di x\di y\leqslant 0\quad\text{for all}\quad u\in L^2(\square).
\end{equation}

A classical case where the last condition is satisfied is the following:
If $K_2(x,y) = K_2(x-y)$ is translation invariant, continuous  and satisfies the normalization $K_2(0) = 1$.
Bochner's theorem tells us that \eqref{eq:positive-definite-K} holds if and only if there is a probability measure $\nu$ on $\RR^n$, such that
 \[
  K_2(x) = - \int \exp(i \xi \cdot x) \mathrm{d} \nu (\xi).
 \]
Other sufficient conditions for~\eqref{eq:positive-definite-K} can be found in \cite{BergCR-84}.

Assumption~\ref{A1} can be again ensured by sufficient smoothness of the kernels $K_i$, e.~g.~$K_i\in C^q(\overline{\square\times\square})$. In this case the term $\cL(\d)\Psi^\d$ in the eigenvalue equation for $\Psi^\d$ is holomorphic in $\d$ in the sense of the $C^q(\overline{\square})$-norm. By Schauder estimates this implies that $\Psi^\d\in C^{2+q}(\overline{\square})$ and $\Psi^\d$ is holomorphic in the sense of the norm in this space. This ensures that Assumption~\ref{A1} holds. Assumptions~\ref{B1+2} hold, since $K_i$ are Hilbert-Schmidt operators.

\subsection{Random second order differential operator}\label{SSe:DO}
It is possible to treat the case where $\cL_i$ are
general symmetric second order differential operators:
 \begin{equation}\label{3.1}
\cL_i=\sum\limits_{k,j=1}^{n+1} \frac{\p}{\p x_k} Q_{kj}^{(i)}\frac{\p}{\p x_j} + \iu \sum\limits_{j=1}^{n+1}\left( Q_j^{(i)}\frac{\p}{\p x_j} + \frac{\p}{\p x_j} Q_j^{(i)}\right) + Q_0^{(i)},
\end{equation}
where $Q_{kj}^{(i)}$, $Q_j^{(i)}$  are real-valued piecewise continuously differentiable functions, $Q_0^{(i)}$ is a real measurable bounded function. We assume that

\begin{equation}\label{3.7}
Q_{jk}^{(i)}=Q_{kj}^{(i)}\quad\text{in}\quad\square\quad\text{and}\quad Q_{kj}^{(i)}=Q_k^{(i)}=0\quad\text{on}\quad\p\square\setminus\p\Pi.
\end{equation}
The latter condition is sufficient for $\cL_i$ to be symmetric.

An integration by parts gives the formula
\begin{equation}\label{3.2}
\L_1=
- \int\limits_{\square}
  Q_{n+1\,n+1}^{(1)}(x)\big(\Psi_0(x_{n+1})\big)^2\di x +\int\limits_{\square} Q_0^{(1)}(x)\Psi_0^2(x_{n+1})\di x.
\end{equation}

If $\L_1<0$, we are in \eqref{C1}.

If $\L_1=0$, then
\begin{equation}\label{3.4}
\cL_1 \Psi_0=\sum\limits_{k=1}^{n+1} \frac{\p\ }{\p x_k} Q_{k\,n+1}^{(1)}\Psi_0 + \iu Q_{n+1}^{(1)}\Psi_0 + \iu \sum\limits_{j=1}^{n+1} \frac{\p\ }{\p x_j} (Q_j^{(1)}\Psi_0) + Q_0^{(1)}\Psi_0
\end{equation}
and this formula is to be substituted into the right hand side of the equations in (\ref{2.12}), (\ref{2.12b}). The constant $\L_2$ is given by the identity similar to (\ref{3.2}):
\begin{equation}\label{3.5}
\L_2=
- \int\limits_{\square}
  Q_{n+1\,n+1}^{(2)}(x)\big(\Psi_0(x_{n+1})\big)^2\di x +\int\limits_{\square} Q_0^{(2)}(x)\Psi_0^2(x_{n+1})\di x.
\end{equation}

The formula for $\eta$ is rewritten as
\begin{equation}
\begin{aligned}
\eta:=&-(b+1)\L_2
\\
&+ (1-b) \left(-\sum\limits_{k=1}^{n+1} \left(\frac{\p\ }{\p x_k}(\Phi_1-\Psi_1), Q_{k\,n+1}^{(1)}\Psi_0\right)_{L^2(\square)} + \iu (\Phi_1-\Psi_1,Q_{n+1}^{(1)}\Psi_0)_{L^2(\square)}\right.
\\
&\hphantom{+(1-b)\Bigg(.}\left.-\iu\sum\limits_{j=1}^{n+1} \left(\frac{\p\ }{\p x_j}(\Phi_1-\Psi_1), Q_j^{(1)}\Psi_0 \right)_{L^2(\square)}
+ (\Phi_1-\Psi_1,Q_0^{(1)}\Psi_0)_{L^2(\square)}
\right).
\end{aligned}\label{3.6}
\end{equation}

Assumption~\ref{A1} again follows from Schauder estimates as above.

Assumption~\ref{B1} can be checked easily. Indeed, by integration by parts we  get that
\begin{align*}
(\cL_i u,\vp u)_{L^2(\square)}=&- \sum\limits_{k,j=1}^{n+1} \left(
Q_{kj}^{(i)}\frac{\p u}{\p x_j}, \frac{\p \vp u}{\p x_i}\right)_{L^2(\square)}
+ \iu \sum\limits_{j=1}^{n+1}\left( Q_j^{(i)}\frac{\p u}{\p x_j},\vp u\right)_{L^2(\square)}
\\
&+\iu \sum\limits_{j=1}^{n+1} \left( Q_j^{(i)}u,\frac{\p\vp u}{\p x_j}\right)_{L^2(\square)} + (Q_0^{(i)}u,\vp u)_{L^2(\square)}.
\end{align*}
Substituting here the identities
\begin{equation*}
\frac{\p \vp u}{\p x_i}=\vp\frac{\p u}{\p x_i} + u \frac{\p \vp }{\p x_i},
\end{equation*}
after obvious estimates we are led to Assumption~\ref{B1}.

 The proof of Assumption~\ref{B2} is even simpler. It is clear that in the expression $\vp \cL_i u-\cL_i \vp u$ all second order derivatives of $u$ cancel out and hence  Assumption~\ref{B2} holds true.

\subsection{Random electro-magnetic field}
\label{ss:simple-magnetic}

Our next example is a random magnetic field, which is in fact a particular case of the previous example. The random operator is introduced as
\begin{gather*}
\Op^\e(\xi):=\big(\iu \nabla + A^\e(\xi)\big)^2+W^\e(\xi),
\\
A^\e(\xi):=\e\sum\limits_{k\in\G} \xi_k \cS(k) A \cS(-k),\quad W^\e(\xi):=\e\sum\limits_{k\in\G} \xi_k \cS(k) (W_1+\e\xi_k W_2) \cS(-k).
\end{gather*}
Here $A\colon \square \to \RR^{n+1}$, is a real-valued magnetic potential, $W_i\colon \square \to\RR$ are bounded measurable real-valued potentials. We suppose that $A$ is piecewise continuously differentiable and it vanishes on the lateral boundary of $\square$. The corresponding operators $\cL_i$ are given by the identities
\begin{equation*}
\cL_1=2\iu A\cdot \nabla + \iu \Div A+W_1,\quad \cL_2=|A|^2+W_2,\quad \cL_3=0.
\end{equation*}

The case $W_0=W_1=0$ was already considered in \cite[Sect. 3.3]{BorisovGV-16} and it was shown that in this case $\L_1=0$ and $\L_2>0$.
While this was stated in~\cite{BorisovGV-16} for $V_0 \equiv 0$, this also holds for non-trivial $V_0$.
An explicit calculation is provided in Appendix~\ref{appendix:V_0}.
In view of Remark~\ref{r:CaseII} \eqref{rem2.4}, this implies that we are neither in \eqref{C1} nor \eqref{C2}. For this reason random magnetic fields are a very instructive example to illustrate why  the
case $\L_1=0$ and $\L_2>0$ is harder to treat than the two cases that we do.
We devote Appendix \ref{ap:CaseIII} to discuss this example and the results of Section~6 of~\cite{HislopK-02}.

To fit our present framework we therefore assume that at least one of the functions $W_0$ and $W_1$ is non-zero.
It was shown in \cite[Sect. 3.3]{BorisovGV-16} that
\begin{equation*}
\int\limits_{\square} \Psi_0  \big(2A\cdot \nabla + \Div A\big)\Psi_0\di x=0
\end{equation*}
and therefore,
\begin{equation*}
\L_1=\int\limits_{\square} W_1\Psi_0^2\di x.
\end{equation*}
If $\L_1<0$, we are in \eqref{C1}.

If $\L_1=0$, then
\begin{equation*}
\cL_1\Psi_0=2\iu A_{n+1}\Psi_0' + \iu\Psi_0\Div A + W_1 \Psi_0,
\end{equation*}
and
\begin{equation*}
\L_2=\int\limits_{\square} (|A|^2+W_2)\Psi_0^2\di x + \int\limits_{\square} \Psi_1 (2\iu A_{n+1}\Psi_0' + \iu\Psi_0\Div A + W_1 \Psi_0)\di x.
\end{equation*}
The constant $\eta$ is given by the identity
\begin{equation*}
\eta=-(b+1) \L_2 + (1-b)\int\limits_{\square} \RE (\Phi_1-\Psi_1) W_1\Psi_0\di x.
\end{equation*}
The condition $\cL_2\leqslant 0$
holds if we have the inequality
\begin{equation*}
|A|^2+W_2<0.
\end{equation*}
Finally, Assumption~\ref{A1} is again guaranteed via the Schauder estimates, assuming sufficient smoothness of $A$ and $W_i$,
for instance $A\in C^{1+q}(\overline{\square})$, $W_i\in C^q(\overline{\square})$, $q\in(0,1)$.
Since the considered example is a particular case of a second order differential operator treated in the previous example, Assumptions~\ref{B1},~\ref{B2} hold as well.

\subsection{Random metric} One more particular case of the above described random second order differential operator is provided by a random metric.
The corresponding random operator is introduced as
\begin{equation*}
\Op^\e(\xi)=-\D+V_0 +  \sum\limits_{k\in\G} \cS(k) \sum\limits_{j,k=1}^{n+1} \frac{\p\ }{\p x_j}
(\e\xi_k Q_{jk}^{(1)}+\e^2\xi_k^2 Q_{jk}^{(2)})\frac{\p\ }{\p x_k}
\end{equation*}
and the associated operators $\cL_i$, $i=1,2$, are given by (\ref{3.1}) with $Q_j^{(i)}=Q_0^{(i)}=0$, $i=1,2$, $j=1,\ldots,n+1$. The functions $Q_{jk}^{(i)}$ are assumed to be piecewise continuously differentiable in $\overline{\square}$, real-valued and obeying condition (\ref{3.7}). The operator $\cL_3$ is supposed to be zero.

Formula (\ref{3.2}) for $\L_1$ takes the form
\begin{equation*}
\L_1=-\int\limits_{\square}
  Q_{n+1\,n+1}^{(1)}(x)\big(\Psi_0(x_{n+1})\big)^2\di x.
\end{equation*}
If $\L_1<0$, we are in \eqref{C1}.

If $\L_1=0$, we just need to appropriately adapt formulae (\ref{3.4}), (\ref{3.5}), (\ref{3.6}).
To identify whether we are in \eqref{C2} let us restrict here our attention to the particular case
\begin{equation*}
Q_{n+1\,n+1}^{(1)}=Q_{k\,n+1}^{(1)}=Q_{n+1\,k}=0.
\end{equation*}
Then $\L_1=0$ and moreover, by (\ref{3.4}), we also have $\cL_1\Psi_0=0$. This implies immediately $\Psi_1=\Phi_1=0$ and the formulae for $\L_2$ and $\eta$ become
\begin{equation*}
\L_2=- \int\limits_{\square}
  Q_{n+1\,n+1}^{(2)}(x)\big(\Psi_0(x_{n+1})\big)^2\di x,\quad \eta=-(b+1)\L_2.
\end{equation*}
In order to satisfy simultaneously the condition $\cL_2\leqslant0$, we impose the assumption that
\begin{equation*}
\sum\limits_{j,k=1}^{n+1}  Q_{jk}^{(2)}(x)z_j z_k\geqslant 0,\quad z_j\in\mathds{R},\quad\text{and}\quad Q_{n+1\,n+1}^{(2)}>0,\quad x\in\overline{\square}.
\end{equation*}
Then $\L_2<0$ and $\eta>0$ and supposing also that $Q_{jk}^{(i)}\in C^{1+q}(\overline{\square})$, $q\in(0,1)$, we satisfy all assumptions of Theorem~\ref{thm:localization_C2}.

\subsection{Random delta-potential}
Our next example is a random delta interaction.
The results of this paragraph have been presented in the announcement~\cite{BorisovTV-18}.

We introduce the random delta interaction as follows:
Let $M_0$ be a closed bounded $C^3$ manifold in $\square \subset \RR^n$ of codimension one. The outward normal vector to $M_0$ is denoted by $\nu_0$. The manifold $M_0$ is assumed to be separated
from the boundary $\p\square$ by a positive distance.
By $M_k$, $k\in\G$, we denote the translate of $M_0$ along $\G$:
\begin{equation*}
M_k:=\{x\in\Pi:\, x-k\in M_0\}.
\end{equation*}
We also let $M:=\bigcup\limits_{k\in\G} M_k$.

By $y=(y_1,\ldots,y_{n-1})$ we denote some local coordinates on $M_0$, while $\vr$ is the distance in $\square$ from a point to $M_0$ measured along $\nu_0$. Since $M_k$ are translations of $M_0$, the coordinates $(y,\vr)$ are well-defined in every $\square_k$, $k\in\G$, and hence, in the whole of $\Pi$.

By $b_0=b_0(y)$ we denote a real function on $M_0$ assuming that $b_0\in C^3(M_0)$. We extent $b_0$ periodically to the entire set $M_0$.

Our random operator is introduced as the negative Laplacian
\begin{equation*}
\tilde{\Op}^\e(\xi):=-\D\quad\text{in}\quad L^2(\Pi),
\end{equation*}
whose domain consists of the functions $u\in \H2(\Pi
\setminus M)$ satisfying the boundary conditions
\begin{equation*}
[u]_{M_k}=0,\quad \left[\frac{\p u}{\p\varrho}\right]_{M_k}=-\e b_0 \xi_k u\big|_{M_k},\quad [u]_{M_k}:=u\Big|_{\genfrac{}{}{0 pt}{}{\vr=+0}{y\in M_k}}-u\big|_{\genfrac{}{}{0 pt}{}{\vr=-0}{y\in M_k}}.
\end{equation*}
The introduced random operator $\tilde{\Op}^\e(\xi)$ cannot be represented as (\ref{2.2}) since the domain of $\tilde{\Op}^\e(\xi)$ is not a subspace in $\H2(\Pi)$.
However, here we apply the approach proposed in \cite[Sect. 8.5]{Borisov-07-MPAG}, see also \cite[Sect. 8.5]{Borisov-07-AHP}. This approach will allow us to transform our operator to another one obeying the assumptions of the present work. Let us describe this approach.

First we denote by $\mathcal{P}_1$ the mapping describing the change of the variables $x\mapsto (y,\vr)$: $(y,\vr)=\mathcal{P}_1(x)$. This mapping is well-defined in a small neighbourhood of each $M_k$, $k\in \G$, and the shape of this neighbourhood is independent of $k$. Then in these neighbourhoods we introduce one more mapping:
\begin{equation*}
\mathcal{P}_2(x,t):=\mathcal{P}_1^{-1}\left(y,\vr+\tfrac{1}{2}\vr|\vr|t b(\xi)\right),
\end{equation*}
where $t\in[-t_0,t_0]$, and $t_0$ is the constant used in the definition of the operator $\cL(t)$.

Let $\chi=\chi(z)$ be an infinitely differentiable cut-off function vanishing for $|z|>2$ and equalling one for $|z|<1$.
We define
\begin{equation*}
\mathcal{P}_3(x,t)=\left(1-\chi\left(\frac{\vr}{\d_0}\right)\right) x + \chi\left(\frac{\vr}{\d_0}\right)\mathcal{P}_2(x,t),
\end{equation*}
where $\d_0>0$ is a sufficiently small fixed number. In \cite[Sect. 8.5]{Borisov-07-MPAG}, the following properties of $\mathcal{P}_3$ were proved:
For sufficiently small $\d_0$, the mapping $\mathcal{P}_3$ is a $C^1$-diffeomorphism, maps $\overline{\square}$ onto itself  and it acts as the identity mapping outside some small neighbourhood of $M_0$. We define a similar mapping on $\RR^n$ as
\begin{equation*}
\mathcal{P}_4(x,\e\xi):=\mathcal{P}_3(x,\e \xi_k)\quad\text{on}\quad \square_k.
\end{equation*}
In view of the aforementioned properties of $\mathcal{P}_3$, the mapping $\mathcal{P}_4$ is a $C^1$-diffeomorphism, maps $\Pi$
onto itself, and it acts as the identity mapping outside a small neighbourhood of $M$.  It also follows from \cite[Sect. 8.5]{Borisov-07-MPAG} that the operator
\begin{equation*}
(\mathcal{U}u)(x):=\mathrm{p}^{-\frac{1}{2}}(x)u\big(\mathcal{P}_4^{-1}(x,\e\xi)\big)
\end{equation*}
is unitary in $L^2(\RR^n)$ and
\begin{equation}\label{3.13}
\Op^\e(\xi):=\mathcal{U}\tilde{\Op}^\e(\xi)\mathcal{U}^{-1}=-\mathrm{p}^{\frac{1}{2}} \Div_x \mathrm{p}^{-1} \mathrm{P}^\intercal \mathrm{P} \nabla_x \mathrm{p}^{\frac{1}{2}},
\end{equation}
where $\mathrm{P}$ is the Jacobian matrix formed by the derivatives $\frac{\p\mathcal{P}_4}{\p x_i}$, and $\mathrm{p}:=\det \mathrm{P}$ is the associated Jacobian.

The domain of the operator $\Op^\e(\xi)$ coincides with space (\ref{2.0}).
The matrix $\mathrm{P}$ and the function $\mathrm{p}$ does not have continuous derivatives on $M$,
however these derivatives are well-defined on $\Pi\setminus M$ and thus have limits on both the inner and outer side of $M$.
This is why the action of the  differential expression in the right hand side in (\ref{3.13}) should be treated as follows: this expression is applied to a function in the domain of $\Op^\e(\xi)$ and the result is calculated in the sense of usual derivatives in $\Pi\setminus M$; the values of a zero measure $M$ are neglected.

The operator $\Op^\e(\xi)$ satisfies (\ref{2.2}); the corresponding operator $\cL(t)$  acts as
\begin{equation}\label{3.14}
\cL(t):=-\mathrm{p}^{\frac{1}{2}} \Div_x \mathrm{p}^{-1} \mathrm{P}^\intercal \mathrm{P} \nabla_x \mathrm{p}^{\frac{1}{2}} + \D_x,
\end{equation}
where in the right hand side we let $\e:=1$, $\xi_k:=t$.
The operators $\cL_i$ can be obtained by expanding the right hand side of (\ref{3.14}) into the Taylor series as $t\to0$. The spectra of the operators $\Op^\e(\xi)$ and $\tilde{\Op}^\e(\xi)$ coincide and therefore, the stated localization of the spectrum for $\Op^\e(\xi)$
implies the same for that of $\tilde{\Op}^\e(\xi)$ provided we can satisfy the required assumptions.

Reproducing literally the calculations of Section~8.5 in \cite{Borisov-07-AHP}, one can check easily that
\begin{equation*}
\L_1=\int\limits_{M_0} b_0(\xi)\Psi_0^2(x)\di x.
\end{equation*}
If $\L_1<0$, we are in \eqref{C1}. Assumption~\ref{A1} then holds due to standard smoothness improving theorems.
Indeed, since the perturbation is localized on the manifold $M_0$, which is separated from $\g$, asymptotics (\ref{2.9}) holds true in the vicinity of $\g$ in $H^p$-norm for each $p\geqslant 1$. This implies immediately the asymptotics (\ref{2.12c}).

Condition~\ref{B1+2} here can be checked in the same way as this was done for a second order differential operator in Section~\ref{SSe:DO}.

\section{Perturbation theory and minimum of the spectrum} \label{s:minimum}
In this section we provide details on the Taylor approximations~\eqref{eq:Taylor} for $\L^\d$ and $\Psi^\d$, prove some preliminary perturbation estimates and prove Theorem~\ref{th2.1}.
\subsection{Taylor expansion of the ground state}
Let $\L_0$ be defined as the smallest eigenvalue of the operator
\begin{equation*}
-\frac{d^2}{dx_{n+1}^2}+V_0\quad \text{on}\quad (0,d)
\end{equation*}
subject to boundary condition (\ref{eq:boundary-condition-B}).

The function $\Psi_0$ is defined as follows:
Let $\Psi_0=\Psi_0(x_{n+1})$ be the unique positive eigenfunction corresponding to $\L_0$ with normalization chosen such that
\begin{equation*}
\|\Psi_0\|_{L^2(0,d)}=\frac{1}{\sqrt{|\square'|}}.
\end{equation*}
We extend $\Psi_0$ to $\square$ by $\Psi_0(x',x_{n+1})=\Psi_0(x_{n+1})$ and use the same symbol for this extension.
The resulting function then belongs to $\H2(\square)$ and is the unique, non-negative, normalized ground state of $\Op^0_\square$.
Furthermore, we note that the function $\Psi_0$ satisfies Neumann as well as periodic boundary conditions on $\g = \p \square \backslash \p \Pi$.

The first lemma describes some properties of $\L^\d$ and $\Psi^\d$ (lowest eigenvalue and eigenfunction of the operator $\Op^\d_\square$).

\begin{lemma}\label{lem:Taylor}
The eigenvalue $\L^\d$ is simple and twice continuously differentiable with respect to sufficiently small $\d$.
The associated eigenfunction $\Psi^\d$ can be normalized to obey $(\Psi^\d,\Psi_0)_{L^2(\square)}=1$ and under such a normalization it is twice differentiable with respect to $\d$ in the norm of $\H2(\square)$.
The first terms of the Taylor expansions for $\L^\d$ and $\Psi^\d$ are
\begin{align}\label{2.7}
&\L^\d=\L_0+\d \L_1 +\d^2 \L_2 +O(\d^3),
\\
\label{2.9}
&\Psi^\d(x)=\Psi_0(x)+\d\Psi_1(x)+\d^2\Psi_2(x)+O(\d^3),
\end{align}
where $\L_i$, $\Psi_i$, $i=1,2$, are uniquely determined by the conditions
\begin{align}
&\L_1:=(\cL_1\Psi_0,\Psi_0)_{L^2(\square)},\nonumber
\\
&\L_2:=(\cL_2\Psi_0,\Psi_0)_{L^2(\square)}  + (\Psi_1,\cL_1\Psi_0)_{L^2(\square)},\nonumber
\\
\label{2.12}
&(\Op^0_\square-\L_0)\Psi_1=-\cL_1\Psi_0 + \L_1\Psi_0, && (\Psi_1,\Psi_0)_{L^2(\square)}=0,
\\
&(\Op^0_\square-\L_0)\Psi_2=-\cL_1\Psi_1 -\cL_2\Psi_0+ \L_1\Psi_1+\L_2\Psi_0, && (\Psi_2,\Psi_0)_{L^2(\square)}=0. \nonumber
\end{align}
\end{lemma}

\begin{proof}
The proof of this lemma is based on regular perturbation theory, see also the proof of Lemma~2.1 in \cite{Borisov-17}.
Namely, as $\d\to0$, the lowest eigenvalue of $\Op_\square^\d$ converges to $\L_0$ and is simple.
The eigenfunction associated with $\L^\delta$  can be chosen so that it converges to $\Psi_0$ in $\H2(\square)$
and we normalize $\Psi^\d$ as it is stated in the formulation of the lemma.
Expansions (\ref{2.7}), (\ref{2.9}) and the stated formulae for their coefficients are implied directly by standard perturbation theory.
\end{proof}

\begin{lemma}\label{lm:minimal-value}
Assume that we are either in  \eqref{C1} or \eqref{C2}.
Let $t_0 > 0$ be sufficiently small.
Then, for all $\e \in (0, t_0]$,
the value $\d \in [b \e, \e]$ which minimizes $[b \e, \e] \ni \d \mapsto \L^\d$
is given by $\d=\e $.
\end{lemma}
\begin{proof}
We differentiate formula (\ref{2.7}):
\begin{equation*}
\frac{d\L^\d}{d\d} = \L_1 + O(\d)
\end{equation*}
and we see that in \eqref{C1}  the sign of $\frac{d\L^\d}{d\d}$ coincides with that of $\L_1$. Hence, for $\L_1>0$, the minimum of $\L^\d$ is attained at $\d=\e b$, while for $\L_1<0$, it is attained at $\d=\e$.

In \eqref{C2}  we have
\begin{equation*}
\frac{d\L^\d}{d\d}=2\L_2\d + O(\d^2)
\end{equation*}
Recall from Remark~\ref{rem2.2} and \cite{Borisov-17} that in \eqref{C2} the constant $\L_2$ is strictly negative and therefore,
\begin{equation*}
\sgn \frac{d\L^\d}{d\d}=-\sgn \d.
\end{equation*}
Hence, the minimum is attained either at $\d=\e$ or at $\d=\e b$.
Since $b>-1$, due to the asymptotics (\ref{2.7}), we conclude that the minimum is attained at $\d=\e$ (decreasing possibly $t_0$ further).
\end{proof}

In the situation of operators on the whole space as discussed in Section~\ref{s:entire-space}, we have the following analogue of Lemma~\ref{lem:Taylor}.
The proof carries over verbatim.

\begin{lemma}\label{lm3.2.1}
The eigenvalue ${\L'}^\d$ is simple and twice continuously differentiable with respect to sufficiently small $\d$. The associated eigenfunction can ${\Psi'}^\d$ can be normalized as
\begin{equation*}
\frac{1}{\sqrt{|\square'|}}\int\limits_{\square'} {\Psi'}^\d\di x'=1
\end{equation*}
and under such a normalization it is twice continuously differentiable  with respect to $\d$ in the norm of $\H2(\square')$. The first terms of the Taylor expansions for ${\L'}^\d$ and ${\Psi'}^\d$ are
\begin{align*}
&{\L'}^\d=\d {\L'}_1 +\d^2 {\L'}_2 +O(\d^3),
\\
&\Psi^\d(x')=\frac{1}{\sqrt{|\square'|}} \left( \mathbf{1}+\d\Psi'_1(x')+\d^2\Psi'_2(x')\right)+O(\d^3),
\end{align*}
where $\Lambda_i$, $\Psi'_i$, $i=1,2$, are uniquely determined by the conditions
\begin{align*}
&\L_1:=\frac{1}{|\square'|}\int\limits_{\square'}\cL'_1\mathbf{1}\di x',
\\
&\L_2:=\frac{1}{|\square'|}\int\limits_{\square'} \cL_2\mathbf{1}\di x'  + \frac{1}{|\square'|}(\Psi'_1,\cL_1\mathbf{1})_{L^2(\square')}, 
\\
&{\Op'}^0_{\square'}\Psi'_1=
-{\cL'}_1\mathbf{1} + {\L'}_1\mathbf{1}, && \int\limits_{\square'}\Psi'_1\di x'=0,
\\
&{\Op'}^0_{\square'}\Psi'_2=-\cL'_1\Psi'_1 -\cL'_2\mathbf{1}+ \L'_1\Psi'_1+\L'_2\mathbf{1}, &&
\int\limits_{\square'}\Psi'_2\di x'=0. 
\end{align*}
\end{lemma}

\subsection{Mezincescu boundary condition} In this subsection we discuss the restrictions of the operator $\Op^\e(\xi)$ on large bounded subdomains of the layer $\Pi$ with special boundary conditions. Given $\a\in\G$, $N\in\NN$, we denote
\begin{align*}
&\Pi_{\a,N}:=\{x:\, x'=\a+\sum\limits_{j=1}^{n} a_j e_j,\ a_j\in(0,N),\ 0<x_{n+1}<d\},
\\
&\G_{\a,N}:=\bigg\{x'\in\G:\, x'=\a+\sum\limits_{j=1}^{n} a_j e_j,\ a_j=0,1,\ldots,N-1\bigg\}
\end{align*}
and we obtain the obvious identity
\begin{equation*}
\Pi_{\a,N}=\bigcup\limits_{k\in\G_{\a,N}} \square_k
\end{equation*}
up to a zero measure set. By $\Op_{\a,N}^\e(\xi)$ we denote the operator
\begin{equation*}
\Op_{\a,N}^\e(\xi):=-\D +V_0+ \cL_{\a,N}^\e(\xi),\quad \cL_{\a,N}^\e(\xi):=\sum\limits_{k\in\G_{\a,N}} \cS(k) \cL(\e\xi_k) \cS(-k)
\end{equation*}
on $\Pi_{\a,N}$ subject to boundary condition (\ref{eq:boundary-condition-B}) on $\p\Pi\cap\p\Pi_{\a,N}$ and to boundary condition
\begin{equation}\label{4.1}
\frac{\p u}{\p\nu}= \rho^{\epsilon}u\quad\text{on}\quad\g_{\a,N}:=\p\Pi_{\a,N}\setminus\p\Pi.
\end{equation}
In the context of random operators, this boundary condition was first used by Mezincescu \cite{Mezincescu-87}. This is why in what follows we refer to it as Mezincescu boundary condition.

\subsection{Proof of Theorem~\ref{th2.1}\label{ss-identify_minimum}} The larger part of Theorem~\ref{th2.1} is a particular case of Theorem~2.1 in \cite{BorisovHEV-18} and it only remains to prove identity (\ref{2.20}).

Identity (\ref{2.19}) and formula~(5.17) in \cite{BorisovHEV-18} with $\theta_0=0$ yield
\begin{equation}\label{4.3}
\min\Sigma_\e\leqslant \frac{(\Op_\square^{\epsilon}\Psi^{\epsilon},\Psi^{\epsilon})_{L^2(\square)}} {\|\Psi^{\epsilon}\|_{L^2(\square)}^2}=\L^{\epsilon}.
\end{equation}

To obtain the converse estimate, let us introduce the quadratic forms
\begin{align*}
\Dom\big(h_{\e,\xi}\big) \ni v &\mapsto h_{\e,\xi}(v)
 :=\|\nabla v\|_{L^2(\Pi)}^2 + (V_0 v,v)_{L^2(\Pi)} +  (\cL^\e(\xi)v,v)_{L^2(\Pi)}
\\
\Dom\big(h_{\e,\xi, \a, N}\big) \ni v &\mapsto h_{\e, \xi, \a, N}(v)
 :=\|\nabla v\|_{L^2(\Pi_{\a,N})}^2 + (V_0 v,v)_{L^2(\Pi_{\a,N})}
  \\
  &\hphantom{  \mapsto h_{\e, \xi, \a, N}(v)
 :=}
 +  (\cL^\e(\xi)v,v)_{L^2(\Pi_{\a,N})}-(\rho^{\epsilon}v,v)_{L^2(\g_{\a,N})}
 \end{align*}
corresponding to $\Op^\e(\xi)$ and $\Op_{\a,N}^\e(\xi)$, respectively.
To specify the corresponding domains let us denote by $\partial_D \Pi$ the part of  $\partial \Pi$
where  (\ref{eq:boundary-condition-B}) corresponds to Dirichlet boundary conditions. Set
\[
\Dom\big(h_{\e,\xi}\big):= \overline{\{f\in C^\infty(\Pi) \cap W^{1,2}(\Pi)\mid  \dist(\supp f , \partial_D\Pi) >0\}}
\]
 where the closure is taken w.~r.~t.~the norm of $W^{1,2}(\Pi)$.
 Furthermore for $\alpha\in \Gamma$ and $ n \in \NN$ let
 \[
 \Dom\big(h_{\e,\xi,\a, N}\big):= \{f\in W^{1,2}(\Pi_{\a,N})\mid  \text{there exists} \   g \in \Dom\big(h_{\e,\xi}\big) \text{ such that } f = g \1_{\Pi_{\a,N}} \}
 \]
By the variational characterisation of the infimum of the spectrum
\begin{align*}
&\inf\spec(\Op^\e(\xi))
=\inf\limits_{0\neq u\in\Dom(h_{\e,\xi}))} \frac{h_{\e,\xi}(u)}{\|u\|_{L^2(\Pi)}^2}
=\inf\limits_{0\neq u\in\Dom(h_{\e,\xi})}
\frac{\sum\limits_{\a\in N\G} h_{\e,\xi, \a, N}(u) + (\rho^{\epsilon}u,u)_{L^2(\g_{\a,N})}} {\sum\limits_{\a\in N\G}\|u\|_{L^2(\Pi_{\a,N})}^2}
\end{align*}
Observe that by cancellation at the interfaces between elementary cells of $N \Gamma$ we have
\begin{equation*}
\sum\limits_{\a\in N \G} (\rho^{\epsilon}u,u)_{L^2(\g_{\a,N})}=0,
\quad \text{ for } \quad u\in \Dom\big(h_{\e,\xi}\big) =\bigoplus\limits_{\a\in N\Gamma}  \Dom\big(h_{\e,\xi, \a, N}\big),
\end{equation*}
and for each $u_\a \in  \Dom\big(h_{\e,\xi, \a, N}\big) $
\[
h_{\e, \xi, \a, N}(u_\a)
\geq
\|u_\a\|_{L^2(\Pi_{\a,N})}^2
\inf_{0\neq v \in  \Dom\big(h_{\e,\xi, \a, N}\big)}   \frac{h_{\e, \xi, \a, N}(v)} {\|v\|_{L^2(\Pi_{\a,N})}^2}
=
\|u_\a\|_{L^2(\Pi_{\a,N})}^2   \L_{\a,N}^\e(\xi),
\]
where $\L_{\a,N}^\e(\xi)$ the lowest eigenvalues of the operator $\Op_{\a,N}^\e(\xi)$.
Hence, $\inf\spec(\Op^\e(\xi)$ is lower bounded by
\begin{equation*}
\inf\limits_{0\neq u\in\Dom(h_{\e,\xi})}
\frac{\sum\limits_{\a\in N\G} h_{\e,\xi, \a, N}(u)} {\sum\limits_{\a\in N\G}\|u\|_{L^2(\Pi_{\a,N})}^2}
\geqslant
\inf\limits_{0\neq u=(u_\a)_{\a\in N\Gamma}\in \bigoplus\limits_{\a\in N\Gamma} \Dom\big(h_{\e,\xi, \a, N}\big)} \quad
\frac{\sum\limits_{\a\in N\G} \|u_\a\|_{L^2(\Pi_{\a,N})}^2 \ \L_{\a,N}^\e(\xi)} {\sum\limits_{\a\in N\G}\|u_\a\|_{L^2(\Pi_{\a,N})}^2}.
\end{equation*}
We now specialize to $N=1$, observe that  $\L_{0,1}^{\e}(\xi_0)=\L^{\e \xi_0}$
and $\L^{\epsilon}\leq \min_{\xi_0\in [b, 1]}  \L_{0,1}^{\e}(\xi_0)$ by Lemma~\ref{lm:minimal-value},
hence conclude
\begin{equation}\label{4.2}
\inf\spec(\Op^\e(\xi))
\geq
\L^{\epsilon}
\end{equation}
Minimizing over the configuration,  we obtain
\begin{equation*}
\min\Sigma_\e\geqslant \L^{\epsilon}.
\end{equation*}
This inequality and \eqref{4.3} prove \eqref{2.20}.

\section{Initial length scale estimate} \label{s:initial}

In this section we formulate and prove a slight generalization of Theorem \ref{thm:lower_bound_ground_state}.
We follow the scheme proposed in \cite{BorisovV-11}, \cite{BorisovV-13}, \cite{BorisovGV-16}, \cite{Borisov-17} for proving similar eigenvalue bounds from below.
Thereafter we deduce an initial length scale estimate,
which is the induction anchor for the multi-scale induction proof of Anderson localization.

\begin{theorem}\label{th4.1}
Let Assumption \ref{A1} hold.
There exist constants $N_1\in \NN$ and $c_0>0$, depending exclusively on the operators $\cL(t), t\in[-T,T],$ and on $V_0$,
such that for all $\a\in\G$, $N \in \NN$ with $N \geq N_1$
and $\xi\in \Omega$ the inequalities hold
\begin{align}
&\L_{\a,N}^\e(\xi)-\L^{\epsilon}\geqslant \frac{\e |\L_1|}{4 N^n} \sum\limits_{k\in\G_{\a,N}} (1-\xi_k) \quad \text{for}\ \e<c_0 N^{-2}\ \text{in \eqref{C1}},\label{4.2a}
\\
&\L_{\a,N}^\e(\xi)-\L^{\epsilon}\geqslant \frac{\eta \e^2}{N^n} \sum\limits_{k\in\G_{\a,N}} (1-\xi_k) \quad\hphantom{11} \text{for}\ \e<c_0 N^{-4}\ \text{in \eqref{C2}}.\label{4.2b}
\end{align}
\end{theorem}

To prove this theorem, we shall need some preliminary notations and lemmata.
By $\mathds{1}$ we denote the constant sequence $\mathds{1}=\{1\}_{k\in\G}$. Let $\hat{\L}_\e$ be the second eigenvalue of the operator $\Op_{\a,N}^\e(\mathds{1})$.
Since the operator $\cL_3(t)$ has uniformly bounded derivative with respect to $t$, it satisfies the estimate
\begin{equation}\label{2.1a}
\big\|\big(\cL_3(t_1)-\cL_3( t_2 )\big)u\big\|_{L^2(\square)}\leqslant C|t_1-t_2|\|u\|_{\H2(\square)}
\end{equation}
for all $t_1,t_2\in[-t_0,t_0]$ and $u\in \H2(\square)$, where $C$ is a constant independent of $t_1$, $t_2$, $u$.
Reproducing literally the proof of Lemma~5.1 in \cite{Borisov-17}, one proves the following lemma.

\begin{lemma}\label{lm4.1}
For sufficiently large $N$ and small $t_0>0$
there exist constants $C_1$,  $C_2$, $C_3$,
depending exclusively on the operators
 $\cL_1$, $\cL_2$, $\cL_3(t)$, $t\in[-T,T]$
  and on $V_0$ such that for all $\a \in \G$ and all $\e \in (0, t_0]$ we have
\begin{align}
& |\hat{\L}_\e-\L_0-C_1 N^{-2}|\leqslant C_2\e,\nonumber
\\
&
\L_{\a,N}^\e(\xi)\leqslant \L_0+C_3\e\quad\text{in \eqref{C1}},\qquad
\L_{\a,N}^\e(\xi)\leqslant \L_0+C_3\e^2\quad\text{in \eqref{C2}}.
\label{4.7}
\end{align}
\end{lemma}

\begin{proof}[Proof of Theorem~\ref{th4.1}]
We choose sufficiently large $N$ and by  $T_N$ we denote the circle in the complex plane of radius $\frac{C_1}{2N^2}$ centered at the origin, where the constant $C_1$ comes from Lemma~\ref{lm4.1}. %
Due to our assumption on $\varepsilon$, asymptotics (\ref{2.7}) and Lemma~\ref{lm4.1}, this circle contains no spectral points of the operator $\Op_{\a,N}^\e(\mathds{1})$ except  for $\L^{\epsilon}$ and
\begin{equation*}
 \dist\big(\l,\spec(\Op_{\a,N}^\e(\mathds{1})) \setminus\{\L^{\epsilon}\}\big) \geqslant \frac{C_1}{3 N^2}.
\end{equation*}
for all $\l\in T_N$.
We rewrite the operator $\Op_{\a,N}^\e(\xi)$ as
\begin{equation}\label{eq:Lhat-operator}
\begin{aligned}
& \Op_{\a,N}^\e(\xi)= \Op_{\a,N}^\e(\mathds{1})+\hat{\cL}_{\a,N}^\e(\xi),
\\
&\hat{\cL}_{\a,N}^\e(\xi):= \cL^\e(\xi)-\cL^\e(\mathds{1})= \sum\limits_{k\in\G_{\a,N}} \cS(k)\hat{\cL}(\e\xi_k)\cS(-k),
\\
&\hat{\cL}(\e \z):=\cL(\e \z)-\cL(\e )=\e(\z-1)(\cL_1+\e(\z+1)\cL_2) + \e^3 \big(\z^3\cL_3(\e \z)-\cL_3(\e )\big).
\end{aligned}
\end{equation}
We observe that due to (\ref{2.1a}), for arbitrary  $u\in \H2(\square)$ we have
\begin{equation}
\begin{aligned}
\|\e^3(\z^3\cL_3(\e \z)-\cL_3(\e ))u\|_{L^2(\square)}&\leqslant \e^3\|(\z^3-1)\cL_3(\e\z)u\|+\e^3     \|(\cL_3(\e \z)-\cL_3(\e ))u\|_{L^2(\square)}
\\
&\leqslant C\e^3|\z-1|\|u\|_{\H2(\square)}
\end{aligned}\label{4.7a}
\end{equation}
with a constant $C$ depending only on the norms of $\cL_3(t)$ and its $t$-derivative, for $t\in [-T,T]$
(in particular, independent of $u$, $\varepsilon$ and $\zeta$).

Proceeding as in \cite[Eqs. (5.19)--(5.26)]{Borisov-17}, one can check easily that $\L_{\a,N}^\e(\xi)$ solves the equation
\begin{align}\label{4.10}
&\L_{\a,N}^\e(\xi)-\L^{\epsilon}
=
\frac{\big(\cA_{\a,N}^\e(\xi)   \hat{\cL}_{\a,N}^\e(\xi)
\Psi^{\epsilon},\Psi^{\epsilon}\big)_{L^2(\Pi_{\a,N})}} {N^n\|\Psi^{\epsilon}\|_{L^2(\square)}^2},
\\
&\cA_{\a,N}^\e(\xi):=\big(I+\hat{\cL}_{\a,N}^\e(\xi) \hat{\cR}_{\a,N}^\e(\L_{\a,N}^\e(\xi))\big)^{-1}, \nonumber
\end{align}
where $\Psi^{\epsilon}$
stands for the periodic extension of the function $\Psi^\d$ in Lemma \ref{lem:Taylor} with $\d=\e$ and
$\hat{\cR}_{\a,N}^\e(\l)$ is the reduced resolvent of the operator $\Op_{\a,N}^\e(\mathds{1})$ for  $\l$ in the vicinity of $\L^{\epsilon}$, specifically, for $\l\in T_N$.
The operator $\hat{\cR}_{\a,N}^\e(\l)$ is bounded as a map from $L^2(\Pi_{\a,N})$ into $\H2(\Pi_{\a,N})$ is holomorphic in $\l\in T_N$ and it satisfies the estimate
\begin{equation}\label{4.11}
\|\hat{\cR}_{\a,N}^\e(\l)f\|_{L^2(\Pi_{\a,N})}\leqslant CN^2 \|f\|_{L^2(\Pi_{\a,N})}
\end{equation}
for all $f\in L^2(\Pi_{\a,N})$, where $C$ is a constant independent of $f$, $\e$, $\a$, $N$.

In \eqref{C2}, equation (\ref{4.10}) was analysed in \cite[Sect. 5]{Borisov-17} and inequality \eqref{4.2b} was proved there.
Hence, in the following we consider only \eqref{C1}.
The operator $\cA_{\a,N}^\e(\xi)$ can be represented as
\begin{equation*}
\cA_{\a,N}^\e(\xi) = I - \hat{\cL}_{\a,N}^\e(\xi) \hat{\cR}_{\a,N}^\e(\L_{\a,N}^\e(\xi)) \cA_{\a,N}^\e(\xi).
\end{equation*}
We substitute this representation into the right hand side of (\ref{4.10}) and obtain
\begin{align*}
&\L_{\a,N}^\e(\xi)-\L^{\epsilon}=S_1+S_2,\quad S_1:=\frac{ (\hat{\cL}_{\a,N}^\e(\xi)\Psi^{\epsilon},\Psi^{\epsilon}\big)_{L^2(\Pi_{\a,N})}} {N^n\|\Psi^{\epsilon}\|_{L^2(\square)}^2},
\\
&S_2:=-\frac{\big(\hat{\cL}_{\a,N}^\e(\xi) \hat{\cR}_{\a,N}^\e(\L_{\a,N}^\e(\xi))
\cA_{\a,N}^\e(\xi)  \hat{\cL}_{\a,N}^\e(\xi)\Psi^{\epsilon},\Psi^{\epsilon}\big)_{L^2(\Pi_{\a,N})}} {N^n\|\Psi^{\epsilon}\|_{L^2(\square)}^2}.
\end{align*}
Our next step is to estimate $S_1$ and $S_2$.

First we observe that according Lemma~\ref{lem:Taylor}, the function $\Psi^\d$ is twice differentiable in $\d$ in the $\H2(\square)$-norm and hence, the norm $\|\Psi^\e\|_{\H2(\square)}$ is bounded uniformly in $\e$.

We begin by estimating the term $S_1$. We substitute the asymptotics (\ref{2.9}) for $\Psi^{\epsilon}$ and
formula (\ref{eq:Lhat-operator}) into the definition of $S_1$ and employ estimate (\ref{4.7a}).
This leads us to the identity
\begin{equation*}
S_1=\frac{1}{N^n}\sum\limits_{k\in\G_{\a,N}} \e (\xi_k-1) \big(\L_1 + S_3(\e,N,\xi_k,\a)\big),
\end{equation*}
where $S_3$ is some function satisfying the estimate
\begin{equation*}
|S_3(\e,N,\xi_k,\a)|\leqslant C\e.
\end{equation*}
Here $C$ is a constant depending only on the norms of $\cL_1$, $\cL_2$, $\cL_3(t)$, $\partial_t\cL_3(t)$,  $t\in[-T,T]$,
in particular independent of $\e$, $N$, $\xi$, $\a$, $k$. Hence, since $\L_1<0$, for sufficiently small $\e$ and all $N$, $\xi$, $\a$, $\xi_k$ we have
\begin{equation*}
\L_1+S_3(\e,N,\xi_k,\a)\leqslant \frac{\L_1}{2}.
\end{equation*}
Thus, we obtain
\begin{equation}\label{4.14}
S_1\geqslant \frac{|\L_1|}{2N^n} \sum\limits_{k\in\G_{\a,N}} \e|\xi_K-1|.
\end{equation}

We estimate $S_2$ in a rather rough way. Namely, in view of the definition of $\cA_{\a,N}^\e$ and $\hat{L}_{\a,N}^\e$, the relative boundedness of the operators $\cL_i(t)$, and estimate (\ref{4.11}), we get that for each $u\in L^2(\Pi_{\a,N})$
\begin{equation*}
\|\cA_{\a,N}^\e(\xi)u\|_{L^2(\Pi_{\a,N})}\leqslant C\|u\|_{L^2(\Pi_{\a,N})}
\end{equation*}
with a constant $C$ independent of $\e$, $\a$, $N$, $\xi$, $u$. Thus, again by the definition of $\hat{\cL}_{\a,N}^\e(\xi)$, (\ref{eq:Lhat-operator}), (\ref{4.7a}), (\ref{4.7}), (\ref{4.11}), we get
\begin{align*}
\big|\big(\hat{\cL}_{\a,N}^\e(\xi) & \hat{\cR}_{\a,N}^\e(\L_{\a,N}^\e(\xi))
\cA_{\a,N}^\e(\xi)  \hat{\cL}_{\a,N}^\e(\xi)\Psi^{\epsilon},\Psi^{\epsilon}\big)_{L^2(\Pi_{\a,N})}
\big|
\\
=&\big|\big( \hat{\cR}_{\a,N}^\e(\L_{\a,N}^\e(\xi))
\cA_{\a,N}^\e(\xi)  \hat{\cL}_{\a,N}^\e(\xi)\Psi^{\epsilon},\hat{\cL}_{\a,N}^\e(\xi)\Psi^{\epsilon}\big)_{L^2(\Pi_{\a,N})}
\big|
\\
\leqslant & C N^2 \|\hat{\cL}_{\a,N}^\e(\xi)\Psi^{\epsilon}\|_{L^2(\Pi_{\a,N})}^2=
C N^2 \sum\limits_{k\in\G_{\a,N}} \|\hat{\cL}_{\a,N}^\e(\xi)\Psi^{\epsilon}\|_{L^2(\square_k)}^2
\\
\leqslant & C N^2 \e^2 \sum\limits_{k\in\G_{\a,N}}|\xi_k-1|^2
\leqslant 4 C N^2 \e^2 \sum\limits_{k\in\G_{\a,N}}|\xi_k-1|.
\end{align*}
Hence,
\begin{equation*}
|S_2|\leqslant \frac{C\e^2}{N^{n-2}} \sum\limits_{k\in\G_{\a,N}}|\xi_k-1|.
\end{equation*}
Combining this estimate with (\ref{4.14}), we arrive at (\ref{4.2a}). This completes the proof.
\end{proof}

Now we are in the position to derive the desired initial length scale estimate
combining the Combes--Thomas bound with an elementary probabilistic estimate.
For this purpose one employs Theorem~\ref{th4.1}, estimate (\ref{4.2}) and proceeds as in \cite{BorisovV-11} \cite{BorisovGV-16}, \cite{Borisov-17}, to obtain the next theorem.

\begin{theorem}
 \label{thm:ISE}
 Let $\b_1,\,\b_2\in\G_{\a,N}$, $m_1,\,m_2>0$ be such that $B_1:=\Pi_{\b_1,m_1}\subset\Pi_{\a,N}$, $B_2:=\Pi_{\b_2,m_2}\subset\Pi_{\a,N}$.
Let $\tau \in \NN$ satisfy $\tau \geq 5$ in~\eqref{C1} or $\tau \geq 17$ in~\eqref{C2}.
Let $N_1$ and $c_0$ be as defined in Theorem~\ref{th4.1}. Define for $N \geq N_1$  the intervals
 \begin{align*}
  J_N
  &:=
  \left[\frac{8}{\sqrt{|\L_1|\mathbb{E}(|\om_0|)}} \frac{1}{N^{\frac{1}{2}}}, \frac{c_0}{N^{\frac{2}{\tau}}}\right]
  \quad
  \text{in~\eqref{C1}}
  \\
  J_N
  &:=
  \left[
  \frac{\sqrt{2}}{\sqrt{\eta\mathbb{E}(|\om_0|)}} \frac{1}{N^{\frac{1}{4}}},
  \frac{c_0}{N^{\frac{4}{\tau}}}
  \right]
  \quad\hphantom{1^{\frac{2}{\tau}}}
  \text{in~\eqref{C2}}.
 \end{align*}
Then there is a constant $c_1$, depending on the measure $\mu$ only, $c_2>0$ independent of $\e$, $\a$, $N$, $\b_1$, $\b_2$, $m_1$, $m_2$ such that for all $N \geqslant \max\{N_1^\tau, K_1^\tau\}$ and for all $\e \in J_N$,
where
\begin{align*}
 K_1
 &:=
 \left(\frac{8}{c_0 \sqrt{|\L_1| \mathbb{E}(|\om_0|)}}\right)^{\frac{2}{\tau-4}}
 \quad
 \text{in~\eqref{C1}}\\
 K_1
 &:=
 \left(\frac{1}{c_0}\sqrt{\frac{2}{\eta \, \mathbb{E}(|\om_0|)}}\right)^{\frac{4}{\tau-16}}
 \quad\hphantom{_1}
 \text{in~\eqref{C2}}
\end{align*}
we have the estimate
\begin{equation}\label{4.12a}
\begin{aligned}
 \mathbb{P}&\left(\forall \l\leqslant \L^{\epsilon}+\frac{1}{2\sqrt{N}}:\, \|\chi_{B_1}(\Op_{\a,N}^\e(\xi)-\l)^{-1}\chi_{B_2}\|\leqslant 2\sqrt{N}
 \exp\left(-c_2 \frac{\dist(B_1,B_2)}{\sqrt N}\right)\right)
 \\
 &\geqslant 1-N^{n\left(1-\frac{1}{\tau}\right)}   \exp\left(-c_1 N^{\frac{n}{\tau}}\right).
\end{aligned}
\end{equation}
This estimate is also valid if $\Op_{\a,N}^\e(\om)$ is equipped with Dirichlet boundary conditions on $\g_{\a,N}$.
\end{theorem}

\begin{remark}
Theorem~\ref{thm:ISE} will be used to start the multi-scale analysis which proves spectral and dynamical localization for all disorders $\e \in J_N$ and all energies $E$ in some energy interval.
A careful analysis is will be required to ensure that the perturbed operator actually has any spectrum in the energy regions appearing in the probability in
\eqref{4.12a}.
Such a discussion will be performed in Section~\ref{sec:MSA}.

Furthermore, we emphasize that \emph{for fixed $N$}, Theorem~\ref{thm:ISE} only allows for choices of $\e$ in the interval $J_N$, i.e. there is a lower bound on the disorder $\e$.
In order to have a meaningful statement for \emph{all sufficiently small $\e$}, we shall have to choose initial scales $N_0:=N_0(\e)$, depending on $\e$.

Both issues will be discussed in Section~\ref{sec:MSA}.
\end{remark}

\section{Wegner estimate}
\label{s:Wegner}

In this section we prove a Wegner estimate, Theorem~\ref{thm:Wegner}, for the operator $\Op^\e(\omega)$ defined in \eqref{2.2}.
It holds for restrictions $\Op_{\a,N}^\e(\omega)$ of $\Op^\e(\omega)$ to $\Pi_{\a,N}$ with Mezincescu boundary conditions on $\g_{\a,n}=\partial \Pi_{\a,N} \setminus \partial \Pi$. The results of this section hold equally
if Mezincescu boundary conditions are replaced by Dirichlet boundary conditions on $\g_{\a,n}$.
For convenience we assume in this section that $\e > 0$ is fixed and
write for brevity $\Op(\omega)$, $\cL(\omega)$, $\Op_{\a,N}(\omega)$ and $\cL_{\a,N}(\omega)$
instead of  $\Op^{\e}(\omega)$, $\cL^\e(\omega)$, $\Op_{\a,N}^\e(\omega)$ and $\cL_{\a,N}^\e(\omega)$, respectively.

Note that in the case of Mezincescu boundary conditions on $\g_{\a,n}$, we have
\[
 \Lambda_0
 =
 \inf \sigma(\Op)
 =
 \inf \sigma(\Op_{\a,N})
\]
for all $\alpha \in \Gamma$ and $N \in \NN$, cf. for instance~\cite[Section 2]{KirschV-10}, and for Dirichlet boundary conditions
\[
 \Lambda_0
 \leq
 \inf \sigma(\Op_{\a,N}).
\]
Motivated by \cite{Klopp-95a}, we define a vector field $\vecfield$ acting on the probability space $\Om = \times_{k\in\G} [-1,1]$ via
\begin{equation}
 \label{eq:definition_vector_field}
 \vecfield
 :=
 \sum_{k \in \Gamma} \omega_k \frac{\partial}{\partial \omega_k}.
\end{equation}
Our Wegner estimate reads as follows, where we refer to Section~\ref{ss-identify_minimum} for the notation on quadratic forms:
\begin{theorem}
 \label{thm:Wegner}
 Let $E_0 < \Lambda_0$ and assume that there is $C_b > 0$ such that for all
$\omega \in \Om$, $\a \in \Gamma$, $N \in \NN$, and all $\phi \in \Dom\big(h_{\e,\xi,\a, N}\big)$ we have
 \begin{align}
 \label{eq:relatively_form_bounded_s}
  \left\langle \phi, \left( \vecfield \cL_{\a,N}(\omega) - \cL_{\a,N}(\omega) \right) \phi \right\rangle
  &\leq
  \frac{1}{4}
  \left\langle \phi, (\Op_{\alpha,N} - E_0  ) \phi \right\rangle,
  \quad
  \text{and}
  \\
    \label{eq:relatively_form_bounded_Weyl}
    \lvert \left\langle \phi, \cL_{\a, N}(\omega) \phi \right\rangle \rvert
    &\leq
    \frac{1}{4}
    \left\langle \phi, \Op_{\a, N} \phi \right\rangle + C_b \lVert \phi \rVert_{L^2(\Pi_{\a,N})}^2.
 \end{align}
 Then there is $C_{\text{Weyl}}\in(0,\infty)$, depending only the dimension,  such that for all $E \in (- \infty, E_0]$, all $\kappa \leq (\Lambda_0 - E_0 ) / 4$, all $\alpha \in \Gamma$, and all $N \in \NN$, we have
 \begin{equation}
 \label{eq:Wegner}
\begin{aligned}
  \PP ( \dist(\sigma(\Op_{\a,N}(\omega)), E) \leq & \kappa )
  \\
  \leq &
  24 C_{\mathrm{Weyl}}(n) \frac{\Vert h_0\Vert_{W^{1,1}(\RR)}}{ \Lambda_0 - E_0}
 \left[1 + (4C_b+ \Vert V_0\Vert_\infty)^{\frac{n+1}{2}} |\square| N^n \right]
 \cdot \kappa N^{n}.
\end{aligned}
 \end{equation}
This holds whether $\Op_{\alpha,N}(\omega) $ is equipped with
Mezincescu or Dirichlet boundary conditions on $\g_{\a,N}$.
\end{theorem}

We shall cast this theorem in a form suitable for the multi-scale analysis in Theorem~\ref{thm:Wegner_with_epsilon} below after introducing the parameter $\e$.
  We stress again that it is a priori not clear whether $\Op(\omega)$ has any spectrum below these $E_0$.
  Some analysis of the infimum of the spectrum of $\Op(\omega)$ will be required.
  This discussion will be done in Section~\ref{sec:MSA}.

\begin{proof} We adapt the strategy of \cite{HislopK-02} to our situation.
Possibly by adding a constant to the potential $V_0$, we may assume that $\Lambda_0 = 0$ in which case we have $\lvert \Lambda_0 - E_0 \rvert \leq \lvert E_0 \rvert$ and the assumption on $\kappa$ becomes $\kappa \leq \lvert E_0 \rvert/4$.
 We shall frequently use that $\Op_{\a,N} - E$ is a positive operator and that $E \leq E_0 < 0$.
 For $E < 0$ and $\omega \in \Omega$, we define
 \[
 X
  :=
  (\Op_{\a,N} - E)^{-\tfrac{1}{2}}
   \left(
    \cL_{\a,N}(\omega)
   \right)
   (\Op_{\a,N} - E)^{-\tfrac{1}{2}}
 \]
where we suppressed the dependence of $X$ on $\a,N,E$ and $\omega$.
Note that due to \eqref{eq:relatively_form_bounded_Weyl}
$X$ is a bounded operator,
as can be seen explicitly by adapting the estimates in \eqref{eq:L-relative-bound}.
 If $E \not \in \sigma(\Op_{\a,N}(\omega))$ we have
 \[
  (\Op_{\a,N}(\omega) - E)^{-1}
  =
  (\Op_{\a,N} - E)^{-\frac{1}{2}} \left( 1 + X \right)^{-1} (\Op_{\a,N} - E)^{-\frac{1}{2}}.
 \]
Using $\lVert (\Op_{\alpha,N} - E)^{-\frac{1}{2}} \rVert = \operatorname{dist} \{ E, \sigma (\Op_{\alpha,N}) \}^{-\frac{1}{2}} \leq \lvert E_0 \rvert^{-\frac{1}{2}}$, we find
 \begin{align*}
 \|   ( \Op_{\a,N}(\omega) - E )^{-1}\|_{L^2(\Pi_{\a,N})\to L^2(\Pi_{\a,N})}
  &\leq
  \frac{\|(1 + X)^{-1}\|_{L^2(\Pi_{\a,N})\to L^2(\Pi_{\a,N})}}{\lvert E_0 \rvert}.
 \end{align*}
This translates into the probability estimate
 \begin{align*}
  \PP ( \dist \{ \sigma(\Op_{\a,N}(\omega)), E \} \leq \kappa)
  = &
  \PP \left( \|(\Op_{\a,N}(\omega) - E)^{-1} \|_{L^2(\Pi_{\a,N})\to L^2(\Pi_{\a,N})} \geq \frac{1}{\kappa} \right)
  \\
  \leq &
  \PP \left( \|(1 + X)^{-1}\|_{L^2(\Pi_{\a,N})\to L^2(\Pi_{\a,N})} \geq \frac{\lvert E_0 \rvert}{\kappa}\right)
  \\
  \leq&
  \PP \left( \dist (\sigma (X), -1) \leq \frac{\kappa}{\lvert E_0 \rvert}\right)
  \leq
  \EE ( \Tr (\chi_{I_\vartheta}(X)))
 \end{align*}
 where $I_\vartheta = [- 1 - \vartheta, -1 + \vartheta ]$ with $\vartheta := \kappa / \lvert E_0 \rvert$.
 By assumption on $\kappa$, we have $- 1 + 2 \vartheta \leq - 1/2$ and consequently $[-1 - 2 \vartheta, -1 + 2 \vartheta] \subset (- \infty, - 1/2]$.

 For  a smooth, antitone function $\vp$ such that $\vp \equiv 1$ on $(- \infty, - \tfrac{\vartheta}{2}]$ and $\vp \equiv 0$ on $[\tfrac{\vartheta}{2}, \infty)$ we have
 \[
    \int_{- \tfrac{3\vartheta}{2}}^{\tfrac{3\vartheta}{2}} \frac{\drm}{\drm t} \vp(x + 1 - t) \drm t
    =\vp(x+1-3 \vartheta/2)-\vp(x+1+3 \vartheta/2)
    \geq \chi_{I_\vartheta}(x)
 \]
implying
 \begin{align*}
  \EE ( \Tr ( \chi_{I_\vartheta}(X)))
  &\leq
  \EE
    \left[
    \Tr \int_{- \tfrac{3\vartheta}{2}}^{\tfrac{3\vartheta}{2}} \frac{\drm}{\drm t} \vp(X + 1 - t) \drm t
    \right]
  =
    \EE
    \left[
    \Tr \int_{- \tfrac{3\vartheta}{2}}^{\tfrac{3\vartheta}{2}} - \vp'(X + 1 - t) \drm t
    \right].
 \end{align*}
 Bearing in mind that $- \vp' \geq 0$,
 \[
  \bigcup_{t \in [- \tfrac{3\vartheta}{2}, \tfrac{3\vartheta}{2}]} \supp \vp'(\cdot + 1 - t)
  \subset
  [- 1 - 2 \vartheta, - 1 + 2 \vartheta],
 \]
 and
 $$
 \frac{X}{(- 1 + 2 \vartheta)} \geq \operatorname{Id}
  \quad \text{on}\quad \operatorname{Ran}  \chi_{[- 1 - 2 \vartheta,  - 1 + 2 \vartheta]}(X),
  $$
  we find
 \begin{align*}
  - \vp'(X + 1 - t)
  &=
  - \vp'(X + 1 - t)  \chi_{[- 1 - 2 \vartheta,  - 1 + 2 \vartheta]}(X)\\
  &\leq
  - \vp'((X + 1 - t) \chi_{[- 1 - 2 \vartheta,  - 1 + 2 \vartheta]}(X) \frac{X}{- 1 + 2 \vartheta}
  \leq
  - \vp'(X + 1 - t) \frac{X}{- 1 + 2 \vartheta}
 \end{align*}
 for all $t \in [- \tfrac{3\vartheta}{2}, \tfrac{3\vartheta}{2}]$.
This yields
\begin{equation}
\label{eq:Wegner_before_vector_field}
 \EE
  \left[
   \Tr \int_{- \tfrac{3\vartheta}{2}}^{\tfrac{3\vartheta}{2}}
   - \vp'(X + 1 - t) \drm t
  \right]
 \leq
 \frac{1}{1 - 2 \vartheta}
 \EE
  \left[
   \Tr \int_{- \tfrac{3\vartheta}{2}}^{\tfrac{3\vartheta}{2}}
   \vp'(X + 1 - t) X \drm t
  \right].
\end{equation}
With the vector field $\vecfield$, defined in~\eqref{eq:definition_vector_field}, we calculate
\[
 \vecfield \vp(X + 1 - t)
 =
 \vp'(X + 1 - t) \vecfield X
 =
 \vp'(X + 1 - t) X
 +
 \vp'(X + 1 - t) (\vecfield X - X)
\]
whence
\[
 \vp'(X + 1 - t) X
 =
 \vecfield \vp(X + 1 - t)
 -
 \vp'(X + 1 - t)(\vecfield X - X).
\]
Plugging this into~\eqref{eq:Wegner_before_vector_field} and using $1 - 2 \vartheta \geq 1/2$ leads to
\begin{equation}
\label{eq:4}
\begin{aligned}
 \EE &\left[ \Tr \int_{-\tfrac{3\vartheta}{2}}^{\tfrac{3\vartheta}{2}} - \vp'(X + 1 - t) \drm t \right]
 \\
 &\leq
 2  \EE \left[ \Tr \int_{-\tfrac{3\vartheta}{2}}^{\tfrac{3\vartheta}{2}}
  \vecfield \vp(X + 1 - t) - \vp'(X + 1 - t) (\vecfield X - X )
 \drm t \right].
\end{aligned}
\end{equation}
Combining inequality~\eqref{eq:relatively_form_bounded_s}, positivity of $- \vp'(X + 1 - t)$, and the fact that $E \leq E_0$ we obtain
\begin{align*}
 - \Tr \vp'(X + 1 - t) (\vecfield X - X)
 &\leq
 - \frac{1}{4} \Tr \vp'(X + 1 - t) (\Op_{\alpha,N} - E)^{-\frac{1}{2}} (\Op_{\alpha,N} - E_0) (\Op_{\alpha,N} - E)^{-\frac{1}{2}}\\
 &\leq
 - \frac{1}{4} \Tr \vp'(X + 1 - t).
\end{align*}
This allows to absorb the second summand on the right hand side of inequality~\eqref{eq:4} on the left hand side.
We have established so far
\begin{align*}
 \PP( \dist \{ \sigma(\Op_{\a,N}(\omega)), E \} < \kappa)
  & \leq
 4 \sum_{j \in \Gamma_{\a,N}} \EE \left[ \Tr \int_{-\tfrac{3\vartheta}{2}}^{\tfrac{3\vartheta}{2}}
  \omega_j \frac{\partial}{\partial \omega_j} \vp(X + 1 - t)  \drm t  \right]
 \\
  & =
 4 \sum_{j \in \Gamma_{\a,N}} \int_{-\tfrac{3\vartheta}{2}}^{\tfrac{3\vartheta}{2}}
  \EE \Tr
  \left[ \omega_j \frac{\partial}{\partial \omega_j} \vp(X + 1 - t)  \drm t  \right]
\end{align*}
Interchanging integrals and the trace is justified by the fact that the trace is actually a finite sum, as will become explicit in \eqref{eq:finite-number}.
The sum over $j \in \Gamma_{\a,N}$ will lead to a $N^n$ factor in Ineq.~\eqref{eq:Wegner}.
Therefore, it suffices to estimate every summand separately with a $j$-independent bound, proportional to $\kappa \cdot (dN^n+1)$.

We use the product structure of the probability space to perform the integration with respect to the particular random variable $\omega_j$ first.
Since the density $\lambda\mapsto h_0(\lambda)$ of the random variables is absolutely continuous, so is $\lambda\mapsto \tilde h(\lambda) := \lambda h_0(\lambda)$.
Using integration by parts for absolutely continuous functions as well as the triangle inequality, we have
\begin{align*}
 &\big\lvert
 \int_b^1
  \tilde h(\omega_j) \frac{\partial}{\partial \omega_j}
  \Tr
    \left\{
     \vp(X + 1 - t)
    \right\}
    \drm \omega_j
    \big\rvert \\
  =
  &\big\lvert
    \tilde h(1)
    \Tr \vp(X^{1,j} + 1 - t)
    -
    \tilde h(b)
    \Tr \vp(X^{b,j} + 1 - t)
    -
    \int_b^1
    \tilde h'(\omega_j)
    \Tr \vp(X^{\omega_j,j} + 1 - t)
    \drm \omega_j
  \big\rvert \\
 \leq
 &
  \left(
   \lvert
   \tilde h(1)
   \rvert
  +
   \lvert
   \tilde h(b)
   \rvert
  +\|
   \tilde h'
   \|_{L_1(b,1)}
  \right)
  \max_{\lambda \in [b,1]} \Tr \vp(X^{\lambda,j} + 1 - t)
\\
 \leq &
2 \Vert h_0 \Vert_{W^{1,1}(\RR)} \,   \max_{\lambda \in [b,1]} \Tr \vp(X^{\lambda,j} + 1 - t)
\end{align*}
 where $X^{\lambda,j}$ denotes the operator $X$ with $\omega_j$ replaced by $\lambda$.
 It remains to bound
 \begin{align}
   \label{eq:Wegner_before_Weyl}
  \int_{- \tfrac{3\vartheta}{2}}^{\tfrac{3\vartheta}{2}}
  \int_{[b,1]^{\lvert \Gamma_{\a,N} \rvert - 1}}
      \max_{\lambda \in [b,1]}
      \Tr \vp(X^{\lambda,j} + 1 - t)
  \left(
  \prod_{i \in \Gamma_{\a,N}, i \neq j} h_0(\omega_i) \drm \omega_i
  \right)
  \drm t.
 \end{align}
 For this purpose, it suffices to establish
 \begin{equation} \label{eq:Weyl}
   \Tr \vp(X + 1 - t)
  \leq
 C_{\text{Weyl}}(n) \left[1 + (4C_b+ \Vert V_0\Vert_\infty)^{\frac{n+1}{2}} |\square| N^n \right]
 \end{equation}
because then,
the integration in~\eqref{eq:Wegner_before_Weyl} with respect to $t$ will yield the $\kappa /\lvert E_0 \rvert$ factor and the integration
with respect to the remaining random variables $\{ \omega_i \}_{i \neq j}$ will amount to one. This will bound in the end (\ref{eq:Wegner_before_Weyl}) by
 \begin{align*}
 C_{\text{Weyl}}(n) \left[1 + (4C_b+ \Vert V_0\Vert_\infty)^{\frac{n+1}{2}} |\square| N^n \right]
\ \frac{3\kappa}{|E_0|}
 \end{align*}
 To prove (\ref{eq:Weyl}), observe that  $\vp \leq 1$, $\supp \vp \subset (- \infty, \tfrac{\vartheta}{2}]$, and $(- \infty, \tfrac{\vartheta}{2} - 1 + t] \subset (- \infty, -\tfrac{1}{2}]$, consequently we have for every $\omega \in \Omega$ and every $t \in [- \tfrac{3\vartheta}{2}, \tfrac{3\vartheta}{2}]$
 \[
  \Tr \vp(X + 1 - t)
  \leq
  \#
  \left\{
  \text{Eigenvalues of $X$ in $(-\infty, -1/2]$}
  \right\}=: k.
 \]
 We claim that this number $k$ is bounded from above by the number of eigenvalues of $(\Op_{\a,N} - E)^{-1}$ in $[(4 C_b)^{-1}, \infty)$, where $C_b$ is the constant from Ineq.~\eqref{eq:relatively_form_bounded_Weyl}.
 In fact, if the $k$-th eigenvalue (counted from above) $\lambda_k^\leftarrow((\Op_{\a,N} - E)^{-1})$ of $(\Op_{\a,N} - E)^{-1}$ was smaller than $(4 C_b)^{-1}$, then for the $k$-th eigenvalue (counted from below) $\lambda_k^\rightarrow(X)$ of $X_{\a,N}$ we would have
 \begin{equation}
  \label{eq:L-relative-bound}
 \begin{aligned}
  \lambda_k^\rightarrow(X)
  &=
  \inf_{\substack{\mathfrak{V} \subset L^2(\Pi_{\a,N}) \\ \dim \mathfrak{V} =k} } \sup_{\substack{\phi \in \mathfrak{V} \\ \lVert \phi \rVert = 1}} \frac{\left\langle (\Op_{\a,N} - E)^{-\frac{1}{2}} \phi, \cL_{\a,N}(\omega) (\Op_{\a,N} - E)^{-\frac{1}{2}} \phi \right\rangle}{\lVert \phi \rVert^2}
  \\ &=
  \inf_{\substack{\mathfrak{W} \subset \Dom\big(h_{\e,\xi,\a, N}\big) \\ \dim \mathfrak{W}=k}} \sup_{\substack{\psi \in \mathfrak{W} \\ \lVert \psi \rVert =1}} \frac{\left\langle \psi, \cL_{\a,N}(\omega) \psi \right\rangle}{\lVert (\Op_{\a,N} - E)^{\frac{1}{2}} \psi \rVert^2}
  \\
 &\geq
  \inf_{\substack{\mathfrak{W} \subset \Dom\big(h_{\e,\xi,\a, N}\big) \\ \dim \mathfrak{W}=k}} \sup_{\substack{\psi \in \mathfrak{W} \\ \lVert \psi \rVert=1}}
    \left(
    - \frac{ \left\langle \psi, \Op_{\a,N} \psi \right\rangle}{4 \lVert (\Op_{\a,N} - E)^{\frac{1}{2}} \psi \rVert^2}
    - \frac{C_b \lVert \psi \rVert^2}{\lVert (\Op_{\a,N} - E)^{\frac{1}{2}} \psi \rVert^2}
    \right)\\
  &\geq
  - \frac{1}{4}
  -
  C_b
  \cdot
  \sup_{\substack{\mathfrak{V} \subset L^2(\Pi_{\a,N}) \\ \dim \mathfrak{V}=k}} \inf_{\substack{\phi \in \mathfrak{V} \\ \lVert \phi \rVert = 1}}
  \frac{\left\langle \phi, (\Op_{\a,N} - E)^{-1} \phi \right\rangle}{\lVert \phi \rVert^2}\\
  &=
  - \frac{1}{4} - C_b \cdot \lambda_k^\leftarrow( (\Op_{\a,N} - E)^{-1} )
  >
  - \frac{1}{2}.
 \end{aligned}
 \end{equation}

Since $\lVert (\Op_{\alpha,N} - E)^{-1} \rVert = \lvert E \rvert^{-1}$
 \begin{align*}
  \lvert \Tr \vp(X + 1 - t)\rvert
  &\leq
  \sharp
 \big\{\text{Eigenvalues of $(\Op_{\a,N} - E)^{-1}$ in $[(4 C_b)^{-1}, \infty)$} \big\}
   \\
  &=
  \sharp\big\{\text{Eigenvalues of $(\Op_{\a,N} - E)^{-1}$ in $[(4 C_b)^{-1}, \lvert E\rvert^{-1}]$} \big\}
 \end{align*}
By the spectral mapping theorem this is equal to
 \begin{align}\label{eq:finite-number}
  \sharp
  \big\{\text{Eigenvalues of $\Op_{\a,N}$ in $[0, 4 C_b + E]$}\big\}
  &\leq
  \sharp
  \big\{\text{Eigenvalues of $\Op_{\a,N}$ in $[0, 4 C_b]$}  \big\}
\end{align}
which is in turn bounded by the number of eigenvalues of the negative Neumann Laplacian
on $\Pi_{\a,N}$ up to energy $4C_b+ \Vert V_0\Vert_\infty$. This can be bounded by
\[
\tilde C_{\text{Weyl}}(n) (4C_b+ \Vert V_0\Vert_\infty)^{\frac{n+1}{2}} |\square| N^n +2^{n+1}
\leq
 C_{\text{Weyl}}(n) \left[1 + (4C_b+ \Vert V_0\Vert_\infty)^{\frac{n+1}{2}} |\square| N^n \right].
\]
\end{proof}

\section{Localization}
\label{sec:MSA}

In this section, we prove Theorems~\ref{thm:localization_C1} and~\ref{thm:localization_C2} on localization for operators on the multidimensional layer.
Theorem~\ref{thm:localization_C1} treats \textbf{Case I}, where the spectrum expands linearly with $\e$.
Theorem~\ref{thm:localization_C2} treats \textbf{Case II}, corresponding to a quadratic expansion of $\Sigma_\e$.

Theorems~\ref{thm:localization_C1'} and~\ref{thm:localization_C2'} for operators acting on the entire space
then follow simply by formula (\ref{5.24b}).
Throughout this section we assume that Assumptions~\ref{A1} and \ref{B1+2} hold.

To prove localization, we shall invoke the multi-scale analysis from~\cite{GerminetK-03}.
For this purpose it is necessary to check a number of basic properties of the model  referred usually to as ``Independence at distance'' (IAD), ``Weyl asymptotics'' (NE), ``Simon-Lieb inequality'' (SLI), ``Eigenfunction decay inequality'' (EDI), ``Strong generalized resolvent expansion'' (SGEE), ``Wegner estimate'' (W) as well as an ``Initial scale estimate'' (ISE).
We shall address their validity in the subsequent subparagraphs.

Given $\a\in\G$ and a natural $N\geqslant 7$, by  $\Ups_{\a,N}$ we denote the set
\begin{equation*}
\Ups_{\a,N}:=\Pi_{\a+e_0,N-2}\setminus\overline{\Pi_{\a+3e_0,N-6}},\qquad e_0:=\sum\limits_{j=1}^{n}e_j.
\end{equation*}
The shape of this set is a `belt' formed by cells $\square_k$ located along the lateral boundary of $\Pi_{\a,N}$. The width of this `belt' is two cells in each direction $e_j$ and this `belt' is separated from the lateral boundary of $\Pi_{\a,N}$ by one cell. By $\chi_{\a,N}$ we denote the characteristic function of $\Ups_{\a,N}$.

\begin{theorem}[{\cite[Theorem~2.4]{GerminetK-03}}]
 \label{thm:ISE_GerminetK}
 Assume that the Assumptions (IAD), (NE), (SLI), (EDI), (SGEE), and (W)  hold in an open interval $\mathcal{I}$.
 Fix a length scale $N_0\in 6 \NN$,
 \[
  N_0
  \geq
  \max
  \left\{
    6,
    3 \rho,
    \eta_\cI^{- \left[ \left( \frac{5}{3} + n \right) \right]^{-1}}
  \right\},
 \]
 where $\rho:=\rho_{\mathrm{IAD}}$ is the parameter from (IAD), and $\eta_\cI$ is a parameter from (W).
 Then, all $E_0 \in \Sigma \cap \mathcal{I}$ such that the initial scale estimate (ISE)
 \begin{equation}
  \label{eq:ISE_GerminetK}
  \PP
  \left\{
    D_\cI
    N_0^{\left( \frac{11 n}{3} \right)}
    \lVert
      \chi_{\Ups_{\a,N_0}} (\Op_{\a,N_0}^\e(\omega) - E_0)^{-1}  \chi_{\a + N_0/2 e_0, N_0/3}
    \rVert
    < 1
  \right\}
  \geq
  1 - \frac{2}{368^n}
 \end{equation}
 holds with
 \[
  D_\cI
  =
  39^{5 n}
  \max
  \left\{
    16 \cdot 60^n Q_I,
    1
  \right\}
  \theta_\cI^2
 \]
 where $Q_\cI$ is the constant in the upper bound of the Wegner estimate and $\theta_\cI$ is the constant\footnote{
In the notation of \cite[Theorem~2.4]{GerminetK-03} the constant $\theta_\cI$ is called $\kappa_\cI$.} in the Simon-Lieb inequality,
 belong to the region in which spectral and dynamical localization holds.
\end{theorem}

\begin{remark}
 $\eta_\cI$ is the maximal value of $\kappa$ in the Wegner estimate and is equal to $D \e^2/4$ in~\eqref{C1} or $D \e^3/4$ in~\eqref{C2}.
\end{remark}

Let us check that all assumptions from Theorem~\ref{thm:ISE_GerminetK} are satisfied, assuming now both Assumptions \ref{A1} and \ref{B1+2}:

\subsection{Independence at a distance (IAD)}

By construction, the restrictions $\Op_{B_1}(\omega)$, $\Op_{B_2}(\omega)$ of $\Op(\omega)$ to disjoint open
$B_1, B_2\subset\Pi$ are independent random operators as soon as the distance between the sets $B_1$ and $B_2$
is larger than $\rho_{\mathrm{IAD}}:= \mathrm{diam}(\square)$.

\subsection{Weyl asymptotics (NE)}
From (\ref{eq:L-relative-bound}) and (\ref{eq:finite-number}) we infer that the following Weyl asymptotics holds:
There is a constant $C_\weyl = C_\weyl(n)$,
depending only on the dimension $n$, and a constant $C_b$, depending on the operator norms of $\cL_1$, $\cL_2$, $\cL_3(\cdot)$, $\cL_3'(\cdot)$ as operators $W^{2,2}(\square) \to L^2(\square)$ such that for all $\omega \in \Omega$, $N \in \NN$, $\alpha \in \Gamma$, $\epsilon \in (0, t_0] \subset [0,1]$ we have
\[
 \Tr \chi_{(- \infty, 0]}
 ( \Op_{\a,N}^\e(\omega))
 \leq
 C_\weyl
 \left[
    1
    +
    \left(
      4 C_b + \lVert V_0 \rVert_\infty
    \right)^{\frac{n + 1}{2}}
    \lVert \square \rvert
    N^n
 \right].
\]
 In particular, for every interval $\cI \subset (- \infty, 0]$ condition (NE) holds with
 \[
  C_\cI
  \leq
  C_\weyl
 \left[
    1
    +
    \left(
      4 C_b + \lVert V_0 \rVert_\infty
    \right)^{\frac{n + 1}{2}}
 \right].
 \]

\subsection{Simon-Lieb inequality (SLI)}\label{ss:SLI}

We recall that $\Op_{\a,N}^\e(\xi)$ and $\cL_{\a,N}^\e(\xi)$ are the restrictions of $\Op^\e(\xi)$ and $\cL^\e(\xi)$ on $L^2(\Pi_{\a,N})$, where $\Op_{\a,N}^\e(\xi)$ is equipped with  Dirichlet or Mezincescu conditions (\ref{4.1}) on the lateral boundary $\g_{\a,N}$ and with condition (\ref{eq:boundary-condition-B}) on $\p\Pi\cap\p\Pi_{\a,N}$.

We choose $L,\ell\in\mathds{N}$, $\ell\geqslant 7$, $L\geqslant 7$, and $\a,\b\in\G$
such that $\Pi_{\b,\ell}\subset \Pi_{\a+3e_0,L-6}$.
For brevity we denote   $U:=\Pi_{\b+3\e_0,\ell-6}$ and $\chi_U$ denotes the characteristic function of this set.
We denote
\begin{equation*}
c_5:=3{\min\limits_{i=1,\ldots,n}}^{-1}|e_i|,\qquad c_{11}:=2\max |V_0|+2.
\end{equation*}

We are going to prove the (SLI) for both types of the boundary conditions on $\g_{\a,N}$.

\begin{lemma}\label{lm-SLI}
Suppose that the operators $\Op_{\a,L}^\e(\xi)$ and $\Op_{\b,\ell}^\e(\xi)$ are equipped with Dirichlet or Mezincescu conditions (\ref{4.1}) on the lateral boundaries $\g_{\a,L}$ and $\g_{\b,\ell}$.
if $\l\not\in\spec(\Op_{\a,L}^\e(\xi))\cup \spec(\Op_{\b,\ell}^\e(\xi))$, the estimate holds:
\begin{align*}
\big\|\chi_{\a,L}&\big(\Op_{\a,L}^\e(\xi)-\l\big)^{-1} \chi_U \big\|_{L^2(\Pi_{\b,\ell})\to L^2(\Ups_{\a,L})}
\\
\leqslant&
\theta_\cI
\big\|\chi_{\b,\ell}\big(\Op_{\b,\ell}^\e(\xi)-\l\big)^{-1} \chi_U \big\|_{L^2(\Pi_{\b,\ell})\to L^2(\Ups_{\b,\ell})}
\big\|\chi_{\a,L}\big(\Op_{\a,L}^\e(\xi)-\l\big)^{-1}\chi_{\b,\ell} \big\|_{L^2(\Ups_{\b,\ell})\to L^2(\Ups_{\a,L})},
\end{align*}
where $\theta_\cI:=(c_6+1)^2\sqrt{2|\l|+ c_{11}+4nc_5}+1$.
\end{lemma}
\begin{proof}
We choose $f\in L^2(\Pi_{\b,\ell})$ arbitrary and denote $$
u_{\a,L}:=\big(\Op_{\a,L}^\e(\xi)-\l\big)^{-1}\chi_U f,\qquad u_{\b,\ell}:=\big(\Op_{\b,\ell}^\e(\xi)-\l\big)^{-1}\chi_U f.
$$
By $\chi_{\mathrm{in}}=\chi_{\mathrm{in}}(x')$, $\chi_{\mathrm{out}}=\chi_{\mathrm{out}}(x')$,  we denote an infinitely differentiable cut-off function such that
\begin{equation}\label{9.15}
\begin{aligned}
&\chi_{\mathrm{in}}\equiv 1\quad\text{on}\quad \Pi_{\b+3e_0,\ell-6},\qquad \chi_{\mathrm{in}}\equiv 0\quad\text{on}\quad \Pi_{\b,\ell}\setminus\Pi_{\b+e_0,\ell-2},
\\
&\chi_{\mathrm{out}}\equiv 0 \quad\text{on}\quad \Pi_{\b,\ell}\setminus\Ups_{\b,\ell}, \qquad \chi_{\mathrm{out}}(x)=1\quad\text{ where }\quad \chi_{\mathrm{in}}(x)\ne0 \quad\text{and}\quad \chi_{\mathrm{in}}(x)\ne 1,
\\
&0\leqslant \chi_\flat\leqslant 1,\quad \left|\frac{\p\chi_\flat}{\p x_i}\right|\leqslant c_5,\quad \left|\frac{\p^2\chi_\flat}{\p x_i\p x_j}\right|\leqslant c_5^2 \quad\text{on}\quad\Ups_{\b,\ell},\quad \flat\in\{\mathrm{in},\mathrm{out}\}.
\end{aligned}
\end{equation}
Such functions obviously exist.

We extend the function $u_{\b,\ell}$ by zero outside $\Pi_{\b,\ell}$ and denote $v:=u_{\a,L}-\chi_{\mathrm{in}} u_{\b,\ell}$.
In view of $\Pi_{\b,\ell}\subset \Pi_{\a+3e_0,L-6}$, the function $v$ satisfies the identity
\begin{equation}\label{9.16}
v=u_{\a,L}\quad\text{on}\quad \Ups_{\a,L},
\end{equation}
belongs to the domain of $\Op_{\a,L}^\e(\xi)$ and solves the equation
\begin{equation}\label{9.17}
\begin{aligned}
&\big(\Op_{\a,L}^\e(\xi)-\l\big) v=g=\chi_{\b,\ell}g,
\\
&g:=2\nabla\chi_{\mathrm{in}}\cdot \nabla u_{\b,\ell}+u_{\b,\ell}\D\chi_{\mathrm{in}} +\big(\chi_{\mathrm{in}}\cL_{\b,\ell}^\e(\xi)-\cL_{\b,\ell}^\e(\xi)\chi_{\mathrm{in}}\big)u_{\b,\ell}.
\end{aligned}
\end{equation}
The function $g$ is supported in the set $\Ups_{\b,\ell}$. Indeed, this is obvious for the first two terms in its definition thanks to properties of $\chi$. And due to the definition of $\cL_{\b,\ell}^\e(\xi)$, the difference $\big(\chi_{\mathrm{in}}\cL_{\b,\ell}^\e(\xi)-\cL_{\b,\ell}^\e(\xi)\chi_{\mathrm{in}}\big)u_{\b,\ell}$ vanishes on each cell $\square_k$, on which $\chi_{\mathrm{in}}$ is identically constant.

It follows from (\ref{9.16}), (\ref{9.17}) that
\begin{align} \nonumber
\|\chi_{\a,L} u_{\a,L}\|_{L^2(\Ups_{\a,L})}
=&\|\chi_{\a,L} v\|_{L^2(\Ups_{\a,L})}
=\big\|\chi_{\a,L} \big(\Op_{\a,L}^\e(\xi)-\l\big)^{-1}\chi_{\b,\ell}g
\big\|_{L^2(\Ups_{\a,L})}
\\ \label{9.18}
\leqslant & \big\|\chi_{\a,L} \big(\Op_{\a,L}^\e(\xi)-\l\big)^{-1}\chi_{\b,\ell}
\big\|_{L^2(\Ups_{\b,\ell})\to L^2(\Ups_{\a,L})} \|g\|_{L^2(\Ups_{\b,\ell})}.
\end{align}

It remains to estimate $\|g\|_{L^2(\Ups_{\b,\ell})}$. Applying Assumption~\ref{B2} and the estimates for the derivatives in (\ref{9.15}), we see immediately that
\begin{equation}\label{9.21}
\|g\|_{L^2(\Ups_{\b,\ell})} \leqslant 2 {c_5} \|\nabla u_{\b,\ell}\|_{L^2(\supp \chi_{\mathrm{in}})} + {c_5}^2 \| u_{\b,\ell}\|_{L^2(\supp \chi_{\mathrm{in}})} + \e c_9 \| u_{\b,\ell}\|_{W^{1,2}(\supp \chi_{\mathrm{in}})},
\end{equation}
where $c_9$ is some absolute constant independent of $\e$ and $u_{\b,\ell}$.

To estimate $\| u_{\b,\ell}\|_{W^{1,2}(\supp \chi_{\mathrm{in}})}$, we write the integral identity associated with the equation for $u_{\b,\ell}$ choosing $\chi_{\mathrm{out}}^2 u_{\b,\ell}$ as a test function:
\begin{equation}\label{9.22}
\begin{aligned}
(\nabla u_{\b,\ell}, \nabla \chi_{\mathrm{out}}^2 u_{\b,\ell})_{L^2(\Ups_{\b,\ell})}&+(V_0 u_{\b,\ell},u_{\b,\ell})_{L^2(\Ups_{\b,\ell})}-\l\|\chi_{\mathrm{out}} u_{\b,\ell}\|_{L^2(\Ups_{\b,\ell})}
\\
&+ (\cL_{\b,\ell}^\e(\xi) u_{\b,\ell}, \chi_{\mathrm{out}}^2 u_{\b,\ell})_{L^2(\Ups_{\b,\ell})}
 = (f\chi_U, \chi_{\mathrm{out}}^2 u_{\b,\ell})_{L^2(\Ups_{\b,\ell})}=0,
\end{aligned}
\end{equation}
where we have also employed that the supports of $\chi_U$ and $\chi_{\mathrm{out}}$ are disjoint. We transform the first term in the above identity as follows:
\begin{equation}\label{9.27}
\begin{aligned}
(\nabla u_{\b,\ell}, \nabla \chi_{\mathrm{out}}^2 u_{\b,\ell})_{L^2(\Ups_{\b,\ell})}=& \|\chi_{\mathrm{out}} \nabla u_{\b,\ell}\|_{L^2(\Ups_{\b,\ell})}^2 + 2 (\nabla u_{\b,\ell}, u_{\b,\ell}\chi_{\mathrm{out}}\nabla \chi_{\mathrm{out}})_{L^2(\Ups_{\b,\ell})}
\\
=& \|\chi_{\mathrm{out}} \nabla u_{\b,\ell}\|_{L^2(\Ups_{\b,\ell})}^2 - \big((|\nabla\chi_{\mathrm{out}}|^2+\chi_{\mathrm{out}} \D \chi_{\mathrm{out}}) u_{\b,\ell}, u_{\b,\ell})_{L^2(\Ups_{\b,\ell})}.
\end{aligned}
\end{equation}
The third term in the left hand side in (\ref{9.22}) can be estimated by means of Assumption~\ref{B1}:
\begin{equation}\label{9.28}
\big|(\cL_{\b,\ell}^\e(\xi) u_{\b,\ell}, \chi_{\mathrm{out}}^2 u_{\b,\ell})_{L^2(\Ups_{\b,\ell})}\big|\leqslant \e \left(c_6 \|\chi_{\mathrm{out}} \nabla u_{\b,\ell}\|_{L^2(\Ups_{\b,\ell})}^2 + c_{10}\| u_{\b,\ell}\|_{L^2(\Ups_{\b,\ell})}^2 \right),
\end{equation}
where $c_{10}$ is some absolute constant independent of $\e$ and $u_{\b,\ell}$. We substitute the obtained relations into (\ref{9.22}) and arrive at the inequality
\begin{equation}\label{9.29}
\begin{aligned}
\|\chi_{\mathrm{out}} \nabla u_{\b,\ell}\|_{L^2(\Ups_{\b,\ell})}^2 &+ \big((\l+V_0-|\nabla\chi_{\mathrm{out}}|^2-\chi_{\mathrm{out}} \D \chi_{\mathrm{out}}) u_{\b,\ell}, u_{\b,\ell})_{L^2(\Ups_{\b,\ell})}
\\
&-\e \left(c_6 \|\chi_{\mathrm{out}} \nabla u_{\b,\ell}\|_{L^2(\Ups_{\b,\ell})}^2 + c_{10}\| u_{\b,\ell}\|_{L^2(\Ups_{\b,\ell})}^2 \right)\leqslant 0,
\end{aligned}
\end{equation}
and for sufficiently small $\e$ we get:
\begin{equation}\label{9.30}
\|\chi_{\mathrm{out}} \nabla u_{\b,\ell}\|_{L^2(\Ups_{\b,\ell})}^2 \leqslant 2 (|\l|+1+\max|V_0|+2n c_5)\|u_{\b,\ell}\|_{L^2(\Ups_{\b,\ell})}^2.
\end{equation}
Since $\chi_{\mathrm{out}}\equiv 1$ on $\supp|\nabla \chi_{\mathrm{in}}|\cup \supp\D\chi_{\mathrm{out}}$, by the latter identity and (\ref{9.21}) we arrive at the final estimate for $g$:
\begin{equation}\label{9.31}
\|g\|_{L^2(\Pi_{\b,\ell})}\leqslant \theta_\cI \|u_{\b,\ell}\|_{L^2(\Ups_{\b,\ell})}
\leqslant \theta_\cI
\big\|\chi_{\b,\ell}\big(\Op_{\b,\ell}^\e(\xi)-\l\big)^{-1}\chi_U
\big\|_{L^2(\Pi_{\b,\ell})\to L^2(\Ups_{\b,\ell})}\|f\|_{L^2(\Pi_{\b,\ell})}.
\end{equation}
The obtained estimate and (\ref{9.18}) imply the statement of the lemma.
\end{proof}

\subsection{Strong generalized eigenfunction expansion (SGEE)}
Here we show that condition (SGEE) is satisfied under our assumptions. We denote
\begin{align*}
&\Dom_+^\e(\xi):=\big\{u:\, u\in\Dom(\Op^\e(\xi))\cap L^2_+(\Pi),\ \Op^\e(\xi)u\in L^2_+(\Pi)\big\},
\\
&L^2_+(\Pi):=\left\{u:\, \int\limits_{\Pi} |u|^2 (1+|x'|^2)^{2m}\di x\right\},
\end{align*}
where $m>\tfrac{n}{4}$ is some fixed integer number to be chosen later. Let $\cT$ be an operator in $L^2(\Pi)$ of multiplication by $(1+|x'|^2)^{-2m}$ and $\cP_\bot^\e(\xi)$ be the orthogonal projector in $L^2(\Pi)$ onto the orthogonal complement to the kernel of $\Op^\e(\xi)$, that is, to the closure of the range of $\Op^\e(\xi)$.
According to the formulation of (SGEE) in \cite[Sect. 2.3]{GerminetK-01a} and equation (2.36) in \cite{GerminetK-01a}, it is sufficient to prove the following lemma.

\begin{lemma}\label{lm:SGEE}
For $m>\tfrac{n}{4}$, the set $\Dom_+^\e(\xi)$ is dense in $L^2_+(\Pi)$ and an operator core for $\Op^\e(\xi)$,
for any $\xi \in \Omega$.
Furthermore, there exists a bounded continuous function $f$ strictly positive on the spectrum of $\Op^\e(\xi)$ such that
\begin{equation*}
\Tr\big(\cT f(\Op^\e(\xi))\cP_\bot^\e(\xi)\cT\big)\leqslant c_{12},
\end{equation*}
where $c_{12}$ is a some constant independent of $\e$ and $\xi$. \end{lemma}

\begin{proof}
Let $\Dom_+^0$ be the set of functions $u\in C^\infty(\overline{\Pi})$ obeying boundary conditions (\ref{eq:boundary-condition-B}) and vanishing for $|x'|$ large enough. It is clear that $\Dom_+^0$ is a subset of $\Dom_+^\e(\xi)$, is dense in $L^2_+(\Pi)$ and is an operator core of the operator $\Op^0$. By the assumed properties, the operator $\cL^\e(\xi)$ is $\Op^0$-bounded with a relative bound less than one as soon as $\e$ is small enough. Then by Kato-Rellich theorem, the set $\Dom_+^0$ is also a core for $\Op^\e(\xi)$ and this proves the first statement of the lemma.

We proceed to prove the bound for the trace.
For each $u\in\Dom_+^\e(\xi)$ and each bounded, continuous, positive function $f$ we have:
\begin{align*}
\big(\cT f(\Op^\e)(\xi)\cP_\bot^\e(\xi)\cT u,u\big)_{L^2(\Pi)}
=&\big(\sqrt{f(\Op^\e)(\xi)}\cT u, \cP_\bot^\e(\xi)\sqrt{f(\Op^\e)(\xi)}\cT u\big)_{L^2(\Pi)}
\\
\leqslant& \big( \cT u, f(\Op^\e)(\xi)\cT u\big)_{L^2(\Pi)}=\big(\cT f(\Op^\e)(\xi) \cT u, u\big)_{L^2(\Pi)}.
\end{align*}
Thus by the definition of the trace,
\begin{equation*}
\Tr\big(\cT f(\Op^\e(\xi))\cP_\bot^\e(\xi)\cT\big)\leqslant
\Tr\big(\cT f(\Op^\e(\xi)) \cT\big).
\end{equation*}

Thanks to Assumption~\ref{B1} with $\vp\equiv 1$, we obtain for all $u\in\Dom(\Op^\e(\xi))$
\begin{align*}
\big((\Op^\e(\xi)+c)u,u\big)_{L^2(\Pi)}\geqslant \|\nabla u\|_{L^2(\Pi)}^2 + (V_0u,u)_{L^2(\Pi)} + c\|u\|_{L^2(\Pi)}^2 -\e \left(c_6\|\nabla u\|_{L^2(\Pi)}^2+c_7 \|u\|_{L^2(\Pi)}^2\right).
\end{align*}
If we set $c:= 1+\|V_0\|_\infty+2Tc_7$ and choose $\e\leqslant t_0\leqslant(2c_6)^{-1}$
then we have for all $\xi\in\Omega$
\begin{equation*}
\Op^\e(\xi)+c\geqslant \frac{1}{2} (\Op_{\cB}+1),
\end{equation*}
where $\Op_{\cB}$ is the operator $-\D$ in $L^2(\Pi)$ subject to boundary conditions (\ref{eq:boundary-condition-B}); its domain is the same as of $\Op^0$.

Now set $f(x):=(x+c)^{-2 m}$, with $m\in\mathds{N}$ a constant, to be chosen later.
Since $f$ is monotone decreasing on $[0,+\infty)$
by the trace properties we conclude that
\begin{equation}\label{9.33}
\Tr\big(\cT f(\Op^\e(\xi)) \cT\big)\leqslant 2^{2 m} \Tr(\cT (\Op_{\cB}+1)^{-2 m}\cT)=2^{2 m} \Tr\big( ((\Op_{\cB}+1)^{- m}\cT)^* (\Op_{\cB}+1)^{- m}\cT\big).
\end{equation}

Let us show that $(\Op_{\cB}+1)^{-m}\cT$ is as Hilbert-Schmidt operator.
We consider the operator $-\frac{d^2\hphantom{x_{n1}}}{dx_{n+1}^2}$ in $(0,d)$ subject to boundary conditions (\ref{eq:boundary-condition-B}).
By $\lmb_q$, $q\in\mathds{N}$, we denote its eigenvalues taken in increasing order and by $\vp_q=\vp_q(x_{n+1})$ the associated eigenfunctions orthonormalized in $L^2(0,d)$.
For all possible choices of the operator $\cB$ in (\ref{eq:boundary-condition-B}), the eigenvalues $\lmb_q$ and eigenfunctions $\vp_q$ can be found explicitly. These formulae imply the lower bounds
\begin{equation}\label{9.34}
\lmb_q\geqslant \frac{\pi^2 (q-1)^2}{d^2},\qquad q\in\mathds{N}.
\end{equation}
For each $u\in L^2(\Pi)$, we have the representation
\begin{equation*}
u(x)=\sum\limits_{q=1}^{\infty} u_q(x')\vp_q(x_{n+1}),\qquad \|u\|_{L^2(\Pi)}^2=\sum\limits_{q=1}^{\infty} \|u_q\|_{L^2(\mathds{R}^n)}^2
\end{equation*}
and hence,
\begin{equation}\label{9.36}
(\Op_{\cB}+1)^{- m}\cT u=\sum\limits_{q=1}^{\infty} \big( (-\D_{x'}+\lmb_q+1)^{- m}\cT' u_q\big)\vp_q.
\end{equation}
Here $-\D_{x'}$ is the Laplacian and $\cT'$ is the multiplication operator by $(1+|x'|^2)^{-2m}$, both acting in $L^2(\mathds{R}^n)$.
As in the proof of Theorem~4.1 in \cite[Ch. 4]{Simon-05}, the operator $(-\D_{x'}+\lmb_{q}+1)^{-m}\cT'$ is integral and its kernel is $(2\pi)^{-\frac{n}{2}}(1+|y'|^2)^{-2m}g_q(x'-y')$, where $g_q$ is the inverse Fourier transform of the function $z\mapsto (|z|^2+\lmb_q+1)^{-m}$, $z\in\mathds{R}^n$.
Hence, by (\ref{9.36}), the operator $(\Op_{\cB}+1)^{- m}\cT$ is also integral with the kernel
\begin{equation*}
K(x,y):=(2\pi)^{-\frac{n}{2}}\sum\limits_{q=1}^{\infty} \vp_q(x_{n+1}) (1+|y'|^2)^{-2m}g_q(x'-y')\vp_q(y_{n+1}).
\end{equation*}
Let us find the $L^2(\Pi\times\Pi)$-norm of this kernel. By $|\mathds{S}^{n-1}|$ we denote the area of the unit sphere in $\mathds{R}^{n}$.
Thanks to identity (4.7) in the proof of Theorem~4.1 in \cite[Ch. 4]{Simon-05}, we have:
\begin{align*}
\|K\|_{L^2(\Pi\times\Pi)}^2=&(2\pi)^{-n}\sum\limits_{q=1}^{\infty}
\int\limits_{\mathds{R}^n\times\mathds{R}^n} (1+|y'|^2)^{-4m}|g_q(x'-y')|^2\di x'\di y'
\\
=&
(2\pi)^{-n}\int\limits_{\mathds{R}^n} \frac{1}{(1+|y'|^2)^{4m}}
\sum\limits_{q=1}^{\infty} \
\int\limits_{\mathds{R}^n}  |g_q(x'-y')|^2\di x'\di y'
\\
=&
(2\pi)^{-n}\int\limits_{\mathds{R}^n} \frac{dy'}{(1+|y'|^2)^{4m}}
\sum\limits_{q=1}^{\infty} \
\int\limits_{\mathds{R}^n}  \frac{dx'}{(|x'|^2+\lmb_q+1)^{2 m}}
\\
=&
(2\pi)^{-n}|\mathds{S}^{n-1}|^2
\int\limits_{0}^{+\infty} \frac{r^{n-1} dr}{(1+r^2)^{4m}}
\sum\limits_{q=1}^{\infty} \
\int\limits_{0}^{+\infty} \frac{s^{n-1}ds}{(s^2+\lmb_q+1)^{2 m}}
\\
=&
(2\pi)^{-n}|\mathds{S}^{n-1}|^2
\int\limits_{0}^{+\infty} \frac{r^{n-1} dr}{(1+r^2)^{4m}}
\int\limits_{0}^{+\infty}
  \frac{t^{n-1}dt}{(t^2+1)^{2m}}
\sum\limits_{q=1}^{\infty} (\lmb_q+1)^{-2m+\tfrac{n}{2}},
\end{align*}
where we have passed to the spherical coordinates and made the rescaling $s=t\sqrt{\lmb_q+1}$
to obtain the last identity. Now we employ estimate (\ref{9.34}) and Theorem~2.11 in \cite[Ch. 2]{Simon-05} and we conclude that
\begin{align*}
\Tr\big( &(\Op_{\cB}+1)^{-m}\cT)^* (\Op_{\cB}+1)^{-m}\cT\big) = \|K\|_{L^2(\Pi\times\Pi)}^2
\\
&=(2\pi)^{-n}|\mathds{S}^{n-2}|^2\int\limits_{0}^{+\infty} \frac{r^{n-1} dr}{(1+r^2)^{2m}}\int\limits_0^{+\infty}  \frac{r^{n-1}dr}{(r^2+1)^{2m}}
\sum\limits_{q=1}^{\infty} (\lmb_q+1)^{-2m+\tfrac{n}{2}}  <\infty \quad\text{for}\quad m > \frac{n}{4}+\frac{1}{2}.
\end{align*}
The obtained estimate and (\ref{9.33}) complete the proof.
\end{proof}

\subsection{Eigenfunction decay inequality (EDI)}

Next we prove (EDI). Let $\psi$ be a generalized eigenfunction of $\Op^\e(\xi)$ associated with some $\l\in\spec(\Op^\e(\xi))$. Then for each $\a\in\G$, $L\geqslant 7$, the restriction of $\psi$ on $\Pi_{\a,L}$ (still denoted by $\psi$) belongs to $W^{2,2}(\Pi_{\a,L})$ and solves the equation
\begin{equation*}
-\D\psi+V_0\psi+ \cL_{\a,L}^\e(\xi) \psi=\l\psi,
\end{equation*}
which is treated as the identity for two functions in $L^2(\Pi_{\a,L})$.
In this subsection we continue using the notation from Section~\ref{ss:SLI}.
The (EDI) condition is established by the next

\begin{lemma}
\label{lm-EDI} Suppose that the operator  $\Op_{\a,L}^\e(\xi)$ is equipped with Dirichlet or Mezincescu conditions (\ref{4.1}) on the lateral boundary $\g_{\a,L}$.
Then, if $\l\in\spec(\Op_{\a,L}^\e(\xi))$, the estimate holds:
\begin{equation*}
\|\chi_U\psi\|_{L^2(\Pi_{\a,L})}
\leqslant
\theta_\cI
%
\big\|\chi_{\a,L}\big(\Op_{\a,L}^\e(\xi) -\l\big)^{-1}\chi_U\big\|_{L^2(\Pi_{\b,\ell})\to L^2(\Ups_{\a,L})} \big\|\chi_{\a,L}\psi\|_{L^2(\Ups_{\a,L})}.
\end{equation*}
\end{lemma}

\begin{proof}
We introduce cut-off functions $\chi_{\mathrm{in}}$, $\chi_{\mathrm{out}}$ as in the proof of Lemma~\ref{lm-SLI} with properties (\ref{9.15}), where $\b$, $\ell$ are replaced by $\a$, $L$. Since $\l\notin \spec\big(\Op_{\a,L}^\e(\xi)\big)$ and $\psi\chi_{\mathrm{in}}\in \Dom\big(\Op_{\a,L}^\e(\xi)\big)$, we have the following chain of identities:
\begin{align*}
\chi_U\psi=&\chi_U\chi_{\mathrm{in}}\psi=\chi_U \big(\Op_{\a,L}^\e(\xi)-\l\big)^{-1} \big(\Op_{\a,L}^\e(\xi)-\l\big) \chi_{\mathrm{in}}\psi
\\
=&\chi_U\big(\Op_{\a,L}^\e(\xi)-\l\big)^{-1} g=\chi_U\big(\Op_{\a,L}^\e(\xi)-\l\big)^{-1}\chi_{\a,L}g,
\end{align*}
where $g$ is defined by the formula in (\ref{9.17}) with $u_{\b,\ell}$ replaced by $\psi$. Then we immediately get:
\begin{equation}\label{9.26}
\|\chi_U\psi\|_{L^2(\Pi_{\a,L})} \leqslant \big\|\chi_U\big(\Op_{\a,L}^\e(\xi)-\l\big)^{-1} \chi_{\a,L}\big\|_{L^2(\Ups_{\a,L})\to L^2(\Pi_{\a,L})} \|g\|_{L^2(\Ups_{\a,L})}.
\end{equation}
The norm $\|g\|_{L^2(\Ups_{\a,L})}$  satisfies estimate (\ref{9.21}) with $u_{\b,\ell}$ and $\Ups_{\b,\ell}$ replaced by $\psi$ and $\Ups_{\a,L}$. Estimates (\ref{9.22}), (\ref{9.27}), (\ref{9.28}), (\ref{9.29}), (\ref{9.30}) remain true with $u_{\b,\ell}$ and $\Ups_{\b,\ell}$ replaced by $\psi$ and $\Ups_{\a,L}$. As in (\ref{9.31}), we then obtain
\begin{equation*}
\|g\|_{L^2(\Ups_{\a,L})}\leqslant \theta_\cI
\|\chi_{\a,L}\psi\|_{L^2(\Ups_{\a,L})}.
\end{equation*}
Substituting this inequality into (\ref{9.26}), we complete the proof.
\end{proof}

\subsection{Wegner estimate (W)}
\label{ss-Wegner-for-localization}
We formulate now a version of Theorem~\ref{thm:Wegner} where the disorder parameter $\e> 0$ is re-introduced,
i.e., we replace every $\omega_j$ by $\e \omega_j$.
The following lemma verifies the hypotheses of Theorem~\ref{thm:Wegner} for sufficiently small $t_0$.
\begin{lemma}
 \label{lem:conditions_Wegner}
  There is $\e_{\max} > 0$, depending only on the norms from ${\H2}(\square)$ into $L^2(\square)$ of the operators $\cL_1$, $\cL_2$, $\cL_3(t)$, and $\partial_t \cL_3(t)$, $t \in [-T,T]$,
 Let $t_0$ be sufficiently small,
 depending only on the norms from ${\H2}(\square)$ into $L^2(\square)$ of the operators $\cL_1$, $\cL_2$, $\cL_3(t)$, and 
 Then there exists
 $D \geq 0$  depending only on the norms from ${\H2}(\square)$ into $L^2(\square)$ of the operators $\cL_2$, $\cL_3(t)$, and $\partial_t \cL_3(t)$, $t \in [-T,T]$, such that for all $\e \in (0, t_0]$
 \begin{enumerate}[(i)]
  \item
  we have
  \[
    \lvert \left\langle \phi, \cL_{\a, N}(\omega)^\e \phi \right\rangle \rvert
    \leq
    \frac{1}{4}
    \left\langle \phi, \Op_{\a, N} \phi \right\rangle + \lVert \phi \rVert_{L^2(\Pi_{\a,N})}^2;
  \]
  \item
   for all $E_0 \leq \Lambda_0 - D \e^2$, we have
  \begin{align*}
  \left\langle \phi, \left( \vecfield \cL_{\a,N}^\e(\omega) - \cL_{\a,N}^\e(\omega) \right) \phi \right\rangle
  \leq
  \frac{1}{4}
  \left\langle \phi, (\Op_{\alpha,N} - E_0 ) \phi \right\rangle;
 \end{align*}
  \item
  if $\cL_2 \leq 0$, then for all $E_0 \leq \Lambda_0 - D \e^3$, we have
    \begin{align*}
  \left\langle \phi, \left( \vecfield \cL_{\a,N}^\e(\omega) - \cL_{\a,N}^\e(\omega) \right) \phi \right\rangle
  \leq
  \frac{1}{4}
  \left\langle \phi, (\Op_{\alpha,N} - E_0 ) \phi \right\rangle.
 \end{align*}
 \end{enumerate}
This holds whether $\Op_{\alpha,N} $ is equipped with
Mezincescu or Dirichlet boundary conditions on  $\g_{\a,N}$.
\end{lemma}

\begin{proof}
Statement (i) follows from the fact that $\cL^\e(t) = \e(\cL_1 + \e \cL_2 + \e^2 \cL_3(\e t))$ is relatively bounded, uniformly with a bound of order $\e$.

To see (ii), we note that by the definition of $\cL(t)$ and an elementary calculation we have
 \[
  t \frac{\partial}{\partial t} \cL^\e(t) - \cL^\e(t)
  =
  \e^2 t^2 \cL_2 + 2 \e^3 t^3 \cL_3(\e t) + \e^4 t^4 \cL_3'(\e t)
  =
  \e^2
  (
  t^2 \cL_2 + 2 \e t^3 \cL_3(\e t) + \e^2 t^4 \cL_3'(\e t)
  )
  .
 \]
 The operators $t^2\cL_2$, $\e t^3 \cL_3(\e t)$ and $\e^2 t^4 \cL_3'(\e t)$ are bounded from $\H2(\square)$ into $L^2(\square)$ uniformly for $t \in [b,1]$ and $\e \in (0,1]$.
 This implies in particular that they are relatively bounded with respect to $\Op_\square^0 - E_0$.
 Therefore, $\vecfield \cL_{\a,N}^\e(\omega) - \cL_{\a,N}^\e(\omega)$ is also relatively bounded with respect to the positive definite operator $\Op_{\a,N} - E_0$  uniformly for $\omega \in \Om$ and $\e \in (0,1]$ with relative bound proportional to $\e^2$.
 Relative boundedness implies relative form boundedness, cf.~\cite[Thm. X.18]{ReedS-75}, thus, there are $D_a,D_b > 0$  independent of $\e \in (0,1]$ such that for all $\phi \in \Dom\big(h_{\e,\xi,\a, N}\big)$ we have
 \begin{align*}
  \left\langle \phi, ( \vecfield \cL_{\a,N}^\e(\omega) - \cL_{\a,N}^\e(\omega) ) \phi \right\rangle
  &\leq
  \e^2\
  \left(
  D_a
  \left\langle \phi, (\Op_{\a,N} - E_0)   \phi \right\rangle
  +
  D_b   \lVert \phi \rVert^2
  \right)
 \end{align*}
 Note that for both Mezincescu and Dirichlet boundary conditions
 $\Op_{\a,N} \geq \Sigma_0$. Choosing $D := 8 D_b$,
 we find for $E_0 \leq \Lambda_0 - D \e^2$
 \[
  \e^2 D_b \lVert \phi \rVert^2
  \leq
  \frac{\e^2 D_b }{\Lambda_0 - E_0} \left\langle \phi, (\Op_{\a,N} - E_0) \phi \right\rangle
  \leq
  \frac{\e^2 D_b}{D \e^2} \left\langle \phi, (\Op_{\a,N} - E_0) \phi \right\rangle
  =
  \frac{1}{8} \langle \phi, (\Op_{\a,N} - E_0) \phi \rangle
 \]
 and obtain
 \[
  \left\langle \phi, ( \vecfield \cL_{\a,N}^\e(\omega) - \cL_{\a,N}^\e(\omega) ) \phi \right\rangle
  \leq
  \left(
    D_a \e^2
    +
    \frac{1}{8}
  \right)
  \langle \phi, (\Op_{\a,N} - E_0) \phi \rangle.
 \]
For sufficiently small $\e >0$ we have $\e^2 D_a \leq 1/8$ and conclude the statement (i).

For for non-positive $\cL_2$ as in claim (iii) we have
 \[
  t \frac{\partial}{\partial t} \cL^\e(t) - \cL^\e(t)
  \leq
  \e^3
  ( 2 t^3 \cL_3(\e t) + \e t^4 \cL_3'(\e t)  )
\]
We find for $E_0 \leq \Lambda_0 - 8 D_b \e^3$ and sufficiently small $\e >0$ again
 \[
  \left\langle \phi, ( \vecfield \cL_{\a,N}^\e(\omega) - \cL_{\a,N}^\e(\omega) ) \phi \right\rangle
 \leq
  \left(    D_a \e^3  +\frac{1}{8}\right)   \langle \phi, (\Op_{\a,N} - E_0) \phi \rangle
 \leq
    \frac{1}{4}     \langle \phi, (\Op_{\a,N} - E_0) \phi \rangle.
 \]
\end{proof}

Combining Lemma~\ref{lem:conditions_Wegner} and Theorem~\ref{thm:Wegner}, we find the following Wegner estimate:

\begin{theorem}
\label{thm:Wegner_with_epsilon}
  Assume that $t_0 > 0$ is sufficiently small.
  Then, there exists $D > 0$ (carrying the same dependencies as in Lemma~\ref{lem:conditions_Wegner}), $C_{n,h_0}$ depending exclusively on $n$ and $h_0$, as well as $C_{n,V_0}$ depending merely on $n$ and $V_0$,
  such that for all $\e \in (0, t_0]$, all $\alpha \in \Gamma$, and all $N \in \NN$
  the following hold:
\begin{enumerate}[(i)]
  \item
  For all $E \leq \Lambda_0 - D \e^2$ and  all $\kappa \leq D \e^2/ 4$ we have
  \begin{align*}
    \PP ( \dist(\sigma(\Op_{\a,N}^\e(\omega)), E) \leq \kappa )
    \leq
    \frac{C_{n,h_0}}{D \e^2}  [1+C_{n,V_0}|\square| N^n]\cdot \kappa \, N^{n}.
  \end{align*}
  \item
  Assume that $\cL_2 \leq 0$.
  Then for all $E \leq \Lambda_0 - D \e^3$ and all $\kappa \leq D \e^3/ 4$,  we have
  \begin{equation*}
    \PP ( \dist(\sigma(\Op_{\a,N}^\e(\omega)), E) \leq \kappa )
    \leq
    \frac{C_{n,h_0}}{D \e^3}  [1+C_{n,V_0}|\square| N^n]\cdot \kappa \, N^{n}.
  \end{equation*}
 \end{enumerate}
\end{theorem}

We note that the constant in the Wegner estimate is
\[
 Q_\cI
 =
 \frac{C_{n,h_0}}{D}  [1+C_{n,V_0}|\square| N^n]
 \cdot
 \begin{cases}
  \e^{-2}
  &
  \text{in~\eqref{C1}},\\
  \e^{-3}
  &
  \text{in~\eqref{C2}}.
 \end{cases}
\]
In particular, we conclude that the constant $D_\cI$ in Theorem~\ref{thm:ISE_GerminetK} is of order $\e^{-2}$ or $\e^{-3}$, respectively.

\subsection{Initial scale estimate~\eqref{eq:ISE_GerminetK}}

The initial scale estimates in Theorem~\ref{thm:ISE} hold for ranges of $\e$ which~\textsl{depend on the scale $N$}.
We aim for a localization statement for \textsl{all sufficiently small $\e$} whence we shall choose a sufficiently small $t_0$ and for every $\e \in (0, t_0]$ a corresponding $\e$-dependent initial scale $N_0(\e)$.

We shall first discuss~\eqref{C1}.
Choosing $\tau = 5$ in Thm.~\ref{thm:ISE} the initial scale estimate thereof holds for all
\[
  \e
  \in
  J_N
  =
  \left[
  \frac{8}{\sqrt{|\L_1|\mathbb{E}(|\om_0|)}}
  \frac{1}{N^{\frac{1}{2}}},
  \frac{c_0}{N^{\frac{2}{5}}}
  \right]
  =:
  \left[
  \frac{\tilde c_1}{{N^{\frac{1}{2}}}},
  \frac{c_0}{N^{\frac{2}{5}}}
  \right].
\]
Let us now additionally require that
\[
  t_0
  \leq
  \min
  \left\{
  c_0^5 /(2 \tilde c_1)^4,
  \tilde c_1/ \sqrt{2}
  \right\}
\]
For $\e \in (0, t_0]$, define then
\[
 N_0:=N_0(\e)
 :=
  \left\lceil
  \left( \frac{\tilde c_1}{\e} \right)^2
  \right\rceil_6
\]
where $\lceil x \rceil_6$ denotes the least multiple of six larger or equal than $x$.
With this choice, one can check that every $\e \in (0, t_0]$ satisfies $\e \in J_{N_0}$.

Now, if $t_0$ is furthermore chosen so small that for all $\e \in (0, t_0]$ we have
\[
  N_0
  \geq
    N_1^5
    +
    \left( \frac{\tilde c_1}{c_0} \right)^{10}
   =
    N_1^\tau
    +
    K_1^\tau
   \geq
    \max \{ N_1^\tau, K_1^\tau \}
\]
where $N_1$ and $K_1$ are the parameters from Theorem~\ref{thm:ISE}, then this theorem and \eqref{2.20} imply
\[
 \mathbb{P}
  \left(
    \forall \l\leqslant \L^\e + \frac{1}{4} \frac{\e}{\tilde c_1}:\,
      \| \chi_{B_1}(\Op_{\a,N_0}^\e(\xi)-\l)^{-1}\chi_{B_2} \|
      \leqslant
      2\sqrt{N_0}
      e^{-c_2 \frac{\dist(B_1,B_2)}{\sqrt{N_0}}}
  \right)
\geqslant
  1 - N_0^{n \frac{4}{5}}
  e^{-c_1 N_0^{\frac{n}{5}}}.
\]
(Recall that $\min\Sigma_\e=\L^\e$.)
We now choose
\[
 B_1 = \Pi_{\a + \tfrac{1}{2} N_0 e_0, \tfrac{1}{3} N_0}
 \quad
 \text{and}
 \quad
 B_2 = \Pi_{\beta,1}
\]
where
$\beta$ is a lattice point in $\Upsilon_{\a, N_0}$.
Assuming $N_0 \in 6\NN$, $N_0 \geq 12$, which can be ensured by the choice of $t_0$, we have in particular $\dist(B_1, B_2) \geq C_\Gamma N_0$ for a constant $C_\Gamma$ only depending on the lattice $\Gamma$.
Thus
$
 \frac{\dist(B_1,B_2)}{\sqrt{N_0}}  \geq C_\Gamma \sqrt{N_0}
$.
Since
\begin{align*}
 &\Big\lVert \chi_{\Upsilon_{\a, N_0}} (\Op_{\a,N_0}^\e(\xi)-\l)^{-1} \chi_{\Pi_{\a + {\tfrac{1}{2}N_0 e_0,\tfrac{1}{3}N_0}}} \Big\rVert
 \leq
 &\sum_{\substack{\beta \in \Gamma \cap \Upsilon_{\a,N_0}}}
 \Big\lVert \chi_{\Pi_{\beta,1}} (\Op_{\a,N_0}^\e(\xi)-\l)^{-1}
 \chi_{\Pi_{\a + \tfrac{1}{2}N_0 e_0, \tfrac{1}{3}N_0}} \Big\rVert
\end{align*}
and the number of lattice points in $\Upsilon_{\a, N_0}$ can be bounded by $4 n N_0^{n-1}$,
we have by a union bound
\begin{align*}
  \mathbb{P} &
  \left(
    \forall \l\leqslant \L^\e + \frac{1}{4} \frac{\e}{\tilde c_1}:\,
      \| \chi_{B_1}(\Op_{\a,N_0}^\e(\xi)-\l)^{-1}\chi_{\Upsilon_{\a, N_0}} \|
      \leqslant
      2\sqrt{N_0}
      e^{-c_2 \frac{\dist(B_1,B_2)}{\sqrt{N_0}}}
  \right)
  \\
&\hphantom{\bigg(\forall \l }  \geq
  1  -   4 n
  N_0^{n-1}       N_0^{n \frac{4}{5}}     e^{-c_1 N_0^{\frac{n}{5}}}
  = 1   -   4 n N_0^{\frac{9 n}{5} - 1}    e^{-c_1  N_0^{\frac{n}{5}}}
\end{align*}
In order to obtain inequality~\eqref{eq:ISE_GerminetK} in Theorem~\ref{thm:ISE_GerminetK}
it suffices to ensure that
\[
 e^{c_2 C_\Gamma \sqrt N_0}  >
 2 D_\cI N_0^{\frac{25 n}{6}}
  \quad
  \text{and}
  \quad
 4 n N_0^{\frac{9n}{5} - 1}   < \frac{2}{368^n}  e^{c_1 N_0^{\frac{n}{5}}}.
\]
Since $N_0 \sim \e^{-2}$ and $D_\cI \sim \e^{-2}$,  this is true for all $\e \in (0, t_0]$ as soon as $t_0$ is chosen sufficiently small, depending on $n$, $h_0$, $V_0$, $\lvert \square' \rvert$, $\cL(\cdot)$, $\Lambda_1$, $\EE \lvert \omega_0 \rvert$, $\theta_\cI$.
Thus, in~\eqref{C1}, we obtain Anderson localization for a sufficiently small $t_0 > 0$ and for every $\e \in (0, t_0]$ in the energy region
\begin{equation*}
 \cI_\e
 :=
 \left( - \infty, \Lambda_0 - D \e^2 \right]
 \cap
 \left( - \infty, \L^\e + \frac{\e}{4 \tilde c_1} \right]
 \cap \Sigma_\e.
\end{equation*}
By Corollary \ref{c:expansion} $\L^\e \leq \Lambda_0 + \tfrac{\L_1}{2} \e$
and  for sufficiently small $t_0$ we have the equality of sets
$ \cI_\e
  =
 \left[\L^\e, \min\{\Lambda_0 - D \e^2, \L^\e +\frac{\e}{4 \tilde c_1} \}\right]
 \cap \Sigma_\e $.
We see that
\[
\cI_\e \supset
 \left[\L^\e, \min\{\L^\e +\tfrac{|\L_1|}{2} \e - D \e^2, \L^\e +\frac{\e}{4 \tilde c_1} \}\right]
 \cap \Sigma_\e\ni \L^\e.
 \]
Thus,
there exists $C > 0$ such that for all $\e\in(0,t_0]$ the set
$\Sigma_\e\cap[\L^\e,\L^\e+C \e]$ is almost surely non empty and exhibits dynamical localization,
which proves Theorem~\ref{thm:localization_C1}.

In~\eqref{C2}, we proceed analogously. We set $\tau = 17$ in Theorem~\ref{thm:ISE}, require
\[
  t_0
  \leq
  \min
  \left\{
  c_0^{17} / (2 \tilde c_2)^{16},
  (5/2)^{1/4} \tilde c_2
  \right\},
  \quad
  \text{where}
  \quad
  \tilde c_2
  =
  \frac{\sqrt{2}}{\sqrt{\eta \EE ( \lvert \omega_0 \rvert)}}
\]
and define
\[
 N_0:=N_0(\e)  :=
 \left\lceil
  \left( \frac{\tilde c_2}{\e} \right)^4
 \right\rceil_6
\]
By analogous calculations as in~\eqref{C1}, we find Anderson localization
for every $\e \in (0, t_0]$ in the energy interval
\[
 \cI_e
 :=
 \left(- \infty, \L_0 - D \e^3\right]
 \cap
 \left(- \infty, \L^\e + \frac{\e^2}{8 \tilde c_2^2} \right]   \cap \Sigma_\e
\]
Note that $\L_0 - D \e^3 \geq \L^\e + \frac{|\L_2|}2\e^2 -D\e^3$. Thus,
there exists a constant $C>0$ such that for all $\e \in(0,t_0]$
\[
\cI_\e\supset \left[ \L^\e, \L^\e + C \e^2\right] \cap \Sigma_\e
\]
and the latter set is non-empty and exhibits dynamical localization, almost surely. This completes the proof of Theorem~\ref{thm:localization_C2}.

\appendix

\section{Random magnetic field with non-zero electric potential}
\label{appendix:V_0}

\begin{remark}
We show here that for a random magnetic field as in Section~\ref{ss:simple-magnetic}, an arbitrary measurable and bounded $V_0$ and no random electric potential (i.e. $W_1 = W_2 = 0$ in the notation of Section~\ref{ss:simple-magnetic}) we have $\L_1 = 0$ and $\L_2 \geq 0$, i.e. we are neither in~\eqref{C1} nor~\eqref{C2}.
For that purpose, we recall parts of the calculation in~\cite[Sec.~3.3]{BorisovGV-16} and study the effect of adding a non-zero background potential $V_0$.
Here we also assume that the considered magnetic field is non-trivial in the sense that it can not be removed by an appropriate gauge transformation. As it is known, this is equivalent to assuming that the magnetic potential $A$ is not a gradient of some scalar function.

 \par
 Recall that $\L_0$ is the smallest eigenvalue of the operator
\begin{equation*}
-\frac{d^2}{dx_{n+1}^2}+V_0\quad \text{on}\quad (0,d)
\end{equation*}
subject to Dirichlet or Neumann boundary condition and
$\Psi_0=\Psi_0(x_{n+1})$ is the associated positive eigenfunction, extended to $\square$ by $\Psi_0(x',x_{n+1})=\Psi_0(x_{n+1})$, and normalized appropriately.
We then have
\begin{equation*}
\L_1 = (\cL_1\Psi_0,\Psi_0)_{L^2(\square)},
\qquad
\L_2 = (\cL_2\Psi_0,\Psi_0)_{L^2(\square)}  + (\Psi_1,\cL_1\Psi_0)_{L^2(\square)},
\end{equation*}
where $\Psi_1$ is the unique solution to
\begin{equation*}
(\Op^0_\square-\L_0)\Psi_1=-\cL_1\Psi_0 + \L_1\Psi_0, \qquad (\Psi_1,\Psi_0)_{L^2(\square)}=0.
\end{equation*}
As in Section~\ref{ss:simple-magnetic}, the operators $\cL_1$ and $\cL_2$ can be written as
\begin{equation} \label{eq:A-L1_L2}
 \cL_1
 =
 \iu [\nabla \cdot A + A \cdot \nabla]
 =
 2 \iu A \cdot \nabla + \iu \operatorname{div}A,
 \quad
 \cL_2
 =
 \lvert A \rvert^2.
\end{equation}
We observe that the formula for $\cL_1$ implies in particular that the function $\Psi_1$ is pure imaginary.

\par
Since $\Lambda_0$ is a ground state of a $1$-dimensional Schr\"odinger equation, it is non-degenerate and the corresponding eigenfunction $\Psi_0$ is up to a phase real-valued (real part and imaginary part are linearly dependent because else they would yield two linearly independent ground states).
Thus, without loss of generality, we assume that $\Psi_0$ is real-valued 
and since $A$ is also real-valued, we can calculate by employing integration by parts:
\begin{equation}\label{MFL1}
 \L_1
 =
 \int\limits_{\square}  \iu \Psi_0   \left[ \nabla \cdot A + A \cdot \nabla \right] \Psi_0 \drm x
 =
 \iu\int\limits_{\square} \big((A \Psi_0) \cdot  \nabla \Psi_0 - \nabla \Psi_0 \cdot (A \Psi_0)\big) \drm x=0.
\end{equation}
Hence, also for non-zero $V_0$, random perturbations consisting purely of magnetic fields always imply $\L_1 = 0$.
Thus, we are not in~\eqref{C1}
\par

Worse, we also have $\L_2 > 0$, such that we are not even in~\eqref{C2}.
To see this, let us first note that
\[
 \Lambda_2
 =
 ( \cL_2 \Psi_0, \Psi_0)_{L^2(\square)} + (\Psi_1, \cL_1 \Psi_0)_{L^2(\square)}
 =
 \int_\square \lvert A \rvert^2 \cdot \lvert \Psi_0 \rvert^2 - 2 \iu \Psi_1 A \cdot \nabla \Psi_0 - \iu \Psi_0 \Psi_1 \operatorname{div} A \drm x.
\]
Now, to see that $\Lambda_2 > 0$ it suffices to establish
\begin{equation}\label{A.1}
 \Lambda_2
 =
 \int_{\square}
 \lvert \Psi_0 \rvert^2 \left\lvert A + \iu \nabla \frac{\Psi_1}{\Psi_0} \right\rvert^2 \drm x.
\end{equation}
Indeed, the right hand side is obviously non-negative and the sum $A + \iu \nabla \frac{\Psi_1}{\Psi_0}$ cannot vanish since otherwise $A$ would be the gradient of a real-valued scalar function because $\Psi_1$ is purely imaginary.

Identity (\ref{A.1}) is proved by calculations which are explicitly performed in~\cite[Section~3.3]{BorisovGV-16}  using in particular that
\[
 \Psi_0 \nabla \frac{\Psi_1}{\Psi_0} = \nabla \Psi_1 - \frac{\Psi_1}{\Psi_0} \nabla \Psi_0
\]
Note that since $\Psi_0$ is a ground state, it is bounded away from zero by Harnack's inequality and we can divide by $\Psi_0$ without any trouble.
\end{remark}

\section{Chasing the Wegner estimate: An example}
\label{ap:CaseIII}

\begin{remark}
  In~\cite{HislopK-02}, operators with random magnetic fields are studied.
  In particular, Section~6 of~\cite{HislopK-02} treats random operators of the form
  \begin{equation}\label{apB1}
   \Op^\e(\omega)
   =
   \big(\iu \nabla + A_0 + A^\e(\omega)\big)^2 + V_0,\qquad A^\e(\omega):=\e\sum\limits_{k\in\G} \om_k A(x'-k,x_n),
  \end{equation}
with a random magnetic field $A^\e(\omega)$, a deterministic magnetic field $A_0$ and  a deterministic electric potential $V_0$.
We assume that
$A_0\colon \overline{\Pi} \to \RR^{n+1}$,  $A\colon \overline{\Pi} \to \RR^{n+1}$ and $V_0\colon\overline{\Pi}\to \mathds{R}$
 are  $\square$-periodic, the potential $A$
vanishes on the boundary of $\square$,
and all these functions are twice continuously differentiable.

 Note that in this case, the zero disorder limit of the random operator is the magnetic Schr\"odinger operator $(\iu \nabla + A_0)^2 + V_0$.
 This is a more general situation than considered in the main body of this paper,
 where the corresponding limit operator is $- \Delta + V_0$.
 In~\cite[Theorem~6.1.(a)]{HislopK-02}, a Wegner estimate is proved, however only in an energy region \textsl{strictly below the infimum of the spectrum of the unperturbed operator} and at small disorder.

  In such a situation it is crucial to investigate whether at small disorder there is any spectrum in the region where the Wegner estimate can be proven. Else it would concern the resolvent set and would be a trivial statement.
 Unfortunately from the discussion in Appendix~\ref{appendix:V_0} it follows that at least in the special case where $A_0 = 0$, we have $\Lambda_1 = 0$ and $\Lambda_2 > 0$.
 Thus, in this special case, at small disorder there is no spectrum at all below the infimum of the unperturbed operator.

 Furthermore, even if the spectrum expanded below the infimum of the deterministic operator, another issue would arise:

More precisely, Theorem~6.1.(a) of~\cite{HislopK-02} states the following:
 Fix parameters $E_0 < \min  \Sigma_0:=\min\sigma(\Op^0)$ and $\eta \in (0, \eta_{\sup})$,
 where $\eta_{\sup} = \dist(E_0, \Sigma_0)/2$ and
 $\Op^0=(i \nabla + A_0)^2 + V_0$ denotes the unperturbed operator.
 Then there exists $\e_0 > 0$ (depending on $E_0$ and $\eta$)
 such that for all disorder strengths\footnote{In~\cite[Theorem~6.1 (a)]{HislopK-02}, the disorder is denoted by $\lambda$ and the maximal disorder strength by $\lambda_0$. We call them $\epsilon$ and $\epsilon_0$ here in order to be consistent with the notation in the rest of the paper} $\e \leq \e_0$,
 we have a Wegner estimate in $[E_0 - \eta, E_0 + \eta]$.
 Clearly, since the spectrum expands continuously with the disorder strength and since $\eta < \dist(E_0, \Sigma_0)/2$, the region of the Wegner estimate $[E_0 - \eta, E_0 + \eta]$ will contain no spectrum for small $\e$.
 Whether there is a parameter $\e \in (0, \e_0]$ such that region where the Wegner estimate holds contains any spectrum at all depends on the rate of expansion of the spectrum with respect to $\e$ and on the interplay of $\eta$, $E_0$, and $\e$.

In fact, in the proof of~\cite[Theorem~6.1 (a)]{HislopK-02} (text between Formulas (6.13) and (6.14))
once $E_0 < \min  \Sigma_0$ is chosen, the disorder strength $\e$ must satisfy
\begin{equation}\label{eq:HK-disorder}
 \e^2 \leqslant\e_0^2
 :=
 \frac{1 - 2 \eta / \dist(E_0, \Sigma_0)}{2 \lVert R_0(E_0)^{\frac{1}{2}} \cL_2 R_0(E_0)^{\frac{1}{2}} \rVert}
 =
 \frac{
 \eta_{\sup} - \eta}
 { 2 \eta_{\sup} \lVert R_0(E_0)^{\frac{1}{2}} \cL_2 R_0(E_0)^{\frac{1}{2}} \rVert}
\end{equation}
where $R_0(E_0) = (\Op^0 - E_0)^{-1}$.
If we choose $\eta$ close to $\eta_{\sup}$, we have $\e_0 \sim 0$,
such that $\sigma(\Op^\e(\omega))$ does not intersect  $[E_0 - \eta, E_0 + \eta]$, cf.~Figure~\ref{fig:expansion_of_spectrum}

\begin{figure}

 \begin{tikzpicture}

\draw[->, thick] (-6,0) -- (1.5,0);
\draw[->, thick] (0,0) -- (0,5);
\draw (.4,4.75) node {$\e$};
\draw (0,-.5) node {$\min \Sigma_{0}$};

\draw[black!80] (0,0) -- (-4,4);
\draw[black!80] (-5.5,4) -- (0,0) -- (-2.5,4);
\fill[pattern = dots, pattern color = black!40]  (-5.5,4) -- (0,0) -- (-2.5,4);

\draw (-4,4.25) node {\tiny $E_0(\e)$};
\draw (-5.5,4.25) node {\tiny $(E_0 - \eta)(\e)$};
\draw (-2.5,4.25) node {\tiny $(E_0 + \eta)(\e)$};


\draw[thick]
	(0,0) {[rounded corners = 20]--
	(-1.25,1.25)  --
	(-1.7,2.6)} --
	(-2.8,4);
	
\fill[pattern = north east lines, pattern color = black!40]
	(0,0) {[rounded corners = 20]--
	(-1.25,1.25)  --
	(-1.7,2.6)} --
	(-2.8,4) --
	(1,4) -- (1,0) -- (0,0);	


\fill[ thick, pattern = north west lines, pattern color = black]
	(0,0) {[rounded corners = 20]--
	(-1.25,1.25)  --
	(-1.7,2.6)} --
	(-2.8,4) --(-2.5,4) -- (0,0);

  \fill[pattern = dots, pattern color = black!40] (-6,-2.5) rectangle (-4,-1);
  \draw (-5,-3) node {Wegner holds here};
  \fill[pattern = north east lines, pattern color = black!40] (-3.5,-2.5) rectangle (-1.5,-1);
  \draw (-2.5,-3) node {$\Sigma_{\e}$};
  \fill[thick, pattern = north west lines, pattern color = black] (-1,-2.5) rectangle (1,-1);
  \draw (0,-2.8) node {Localization is};
  \draw (0,-3.2) node {meaningful here};
 \end{tikzpicture}

\caption{{The maximal possible expansion of the spectrum as a function of $\eta$ and the area where the Wegner estimate holds}}
\label{fig:expansion_of_spectrum}
\end{figure}
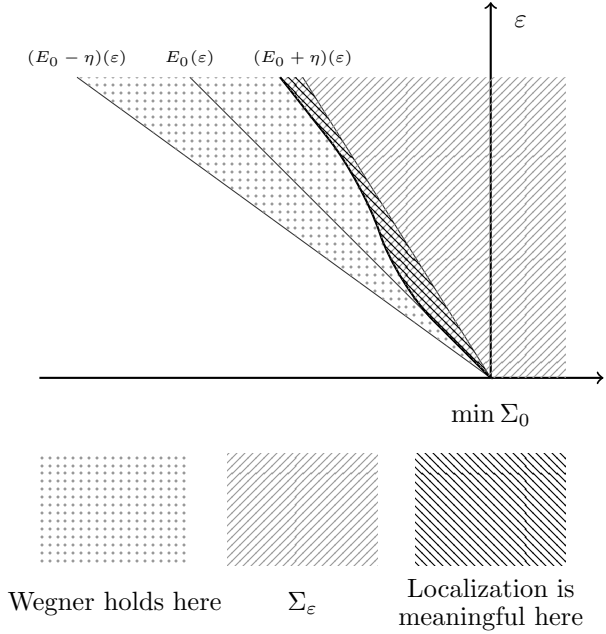

This shows that it is more natural to choose the interval $[E_0 - \eta, E_0 + \eta]$ depending on the disorder $\e$, in particular $E_0=E_0(\e)$ and $\eta = \eta(\e)$.
To make sure that this interval intersects $\Sigma_\e$ we need to have $\min \Sigma<E_0(\e)+\eta(\e)$.
Possibly adding a constant to $\Op^0$, we can assume w.l.o.g~$\min \Sigma_0=0$.
Since $\eta(\e) < \eta_{\sup}= |E_0(\e)|/2$ we have $E_0(\e)+\eta(\e) \leq E_0(\e)/2$.

Let us first consider the case that $\min\Sigma_\e$ decreases linearly for small $\e>0$, i.e.~there is an $c_{\ell}>0$ such that  $\min\Sigma_\e\leq -c_{\ell} \e$
(in analogy to Case (I) in the main body of the paper).
In order to ensure $E_0(\e)+\eta(\e) \in \Sigma_\e$,  we choose $E_0(\e) := -c_{\ell}\e$  and $\eta(\e) = \frac{c_{\ell}}{4}\e$.
Then
\[
\eta_{\sup}  = \dist(E_0(\e), 0)/2=\frac{c_{\ell}}{2}\e> \eta(\e)
\]
and, indeed, $E_0(\e)+\eta(\e) = -\frac{3c_{\ell}}{4}\e > -c_{\ell}\e\geq \min \Sigma_\e$.
Hence for this choice of $\eta(\e)$, in the light of~\eqref{eq:HK-disorder}, \cite[Theorem~6.1 (a)]{HislopK-02} allows disorder strengths
\[
 \e^2
 \leq
 \frac{1}{4 \lVert R_0(E_0(\e))^{\frac{1}{2}} \cL_2 R_0(E_0(\e))^{\frac{1}{2}} \rVert}.
\]
Let us bound the denominator, using representation (\ref{eq:A-L1_L2})
\begin{multline} \nonumber
\lVert R_0(E_0(\e))^{\frac{1}{2}} \cL_2 R_0(E_0(\e))^{\frac{1}{2}} \rVert
\leqslant
\lVert R_0(E_0(\e))^{\frac{1}{2}}\rVert^2 \lVert \cL_2 \rVert
=|E_0(\e)|^{-1} \lVert |A|^2 \rVert_\infty  
=\frac{\sup |A|^2 }{c_{\ell}\e}
\end{multline}
Thus Ineq.~\eqref{eq:HK-disorder} is satisfied if $\e^2 \leqslant    c_{\ell} \e/( 4\sup |A|^2)$.
This is true for sufficiently small $\e$.

\bigskip

Let us turn to the case that $\min\Sigma_\e$ decreases quadratically for small $\e>0$, i.e.~there is an $c_{q}>0$ such that  $\min\Sigma_\e\leq -c_{q} \e^2$
(in analogy to Case (II) in the main body of the paper).

In order to ensure $E_0(\e)+\eta(\e) \in \Sigma_\e$,  we choose $E_0(\e) := -c_{q}\epsilon^2$  and $\eta(\e) = \frac{c_{q}}{4}\epsilon^2$.
Then
\[
\eta_{\sup}  = \dist(E_0(\e), 0)/2=\frac{c_{q}}{2}\epsilon^2> \eta(\e)
\]
and, indeed, $E_0(\e)+\eta(\e) = -\frac{3c_{q}}{4}\epsilon^2 > -c_{q}\epsilon^2\geq \min \Sigma_\e$.
Hence for this choice of $\eta(\e)$, in the light of~\eqref{eq:HK-disorder}, \cite[Theorem~6.1 (a)]{HislopK-02} allows disorder strengths
\[
 \e^2
 \leq
 \frac{1}{4 \lVert R_0(E_0(\e))^{\frac{1}{2}} \cL_2 R_0(E_0(\e))^{\frac{1}{2}} \rVert}.
\]
Let us bound the denominator, using representation (\ref{eq:A-L1_L2})
\begin{multline} \nonumber
\lVert R_0(E_0(\e))^{\frac{1}{2}} \cL_2 R_0(E_0(\e))^{\frac{1}{2}} \rVert
\leqslant
\lVert R_0(E_0(\e))^{\frac{1}{2}}\rVert^2 \lVert \cL_2 \rVert
=|E_0(\e)|^{-1} \lVert |A|^2 \rVert_\infty  
=\frac{\sup |A|^2 }{c_{q}\epsilon^2}
\end{multline}
Thus Ineq.~\eqref{eq:HK-disorder} is satisfied if $\e^2 \leqslant    c_{q}\e^2 /( 4\sup |A|^2)$, i.~e.~$4\sup |A|^2 \leqslant    c_{q} $
To decide whether this condition holds, one needs to provide a lower bound on the expansion coefficient $c_{q}$.
\end{remark}

In any case, merely assuming small disorder is not sufficient to ensure that~\cite[Theorem~6.1 (a)]{HislopK-02} is a non-trivial statement. Additional arguments or assumptions are required.
To elucidate this further we want to exhibit situations where indeed we have linear expansion, i.e.~$\min\Sigma_\e\leq -c_{\ell} \e$  with $c_{\ell} >0$.


\begin{remark}
The operator (\ref{apB1}) does not fit the assumptions of the present paper.
Indeed, for $\e=0$ the unperturbed operator becomes $\Op^0=(\iu\nabla +A_0)^2$, possibly with $A_0\ne0$.
Nevertheless, the issue on how the spectrum expands can be analyzed in this situation, as well, by applying the general results of \cite{BorisovHEV-18}.
In this context Theorem \ref{th2.1} is replaced by Theorem 2.1 in \cite{BorisovHEV-18} still ensuring that there exist an almost sure spectral set $\Sigma_\e$.
To fit the setting of \cite{BorisovHEV-18}  we introduce some notation and assumptions.
We begin with the Floquet-Bloch expansion for the unperturbed operator, namely, we consider the operator $\Op^0_{per}:=(\iu\nabla +A_0)^2$
on the periodicity cell $\square$ subject to $\theta$-quasiperiodic boundary conditions on the (lateral) boundaries $\gamma$
and denote its lowest eigenvalue by $E_0(\theta)$.
We assume further that there exists a unique quasimomentum $\theta_0$ such that $E_0(\theta_0)=\min_\theta E_0(\theta)=\L_0 :=\min\Sigma_\e$,
that $\L_0$ is a simple eigenvalue of $\Op^0$ on $\square$ with $\theta_0$-periodic boundary conditions on $\g$ and $\Psi_0$ is the associated eigenfunction normalized in $L^2(\square)$.
Continuous dependence of the Floquet eigenvalues on the quasimomentum implies that there is closed ball $U$ around $\theta_0$ such that $\min_{\theta \in U}\dist(E_0(\theta), \sigma(\Op^0)\setminus \{E_0(\theta)\})>0$.
Such a scenario indeed occurs for electromagnetic Schr\"odinger operators, see e.g.~\cite{Shterenberg-05,FilonovK-18}.
Since here we focus on the effects of magnetic vector potentials we will assume in the following $V_0\equiv0$.
Assume also that the function $\varpi:=\frac{\Psi_0}{|\Psi_0|}$ belongs to $C^2(\overline{\square})$.

Then one can prove along the lines of \cite{BorisovHEV-18}
 that the bottom of the spectrum of the operator $\Op^\e(\omega)$ satisfies the asymptotic formula
\begin{equation}\label{apB2}
\inf\spec\big(\Op^\e(\omega)\big)\leqslant\L_0 + \e \min\{b \L_1, \L_1\} +O(\e^2),
\end{equation}
where the constant $\L_1$ is determined by a formula similar to (\ref{MFL1}):
\begin{equation}\label{apB3}
 \L_1
=
 \int\limits_{\square}  \overline{\Psi_0} \left[ (\iu\nabla+A_0) \cdot A + A \cdot (\iu\nabla+A_0) \right] \Psi_0 \drm x.
\end{equation}
We also note that if the function $\frac{1}{\Psi_0}\frac{\p\Psi_0}{\p\nu}$ belongs to $C^1(\p\square)$ then according Theorem~2.3 in  \cite{BorisovHEV-18}, the inequality in (\ref{apB2}) can be replaced by the identity.

Integrating by parts in (\ref{apB3}) and taking into consideration that $A$ vanishes on the boundary of $\square$, we obtain:
\begin{equation}\label{apB4}
 \L_1
=
 \int\limits_{\square}  \Big(\overline{\Psi_0 A} \cdot (\iu\nabla+A_0) +
\Psi_0 A  \cdot \overline{(\iu\nabla+A_0)\Psi_0}\Big) \drm x
=2\RE \big(\Psi_0 A,(\iu\nabla+A_0)\Psi_0 \big)_{L^2(\square)}.
\end{equation}
Thanks to the assumed smoothness of $A_0$, $A$ and $V_0$ and by the Schauder estimates, we have $\Psi_0\in C^2(\overline{\square})$. Then we can rewrite formula (\ref{apB4}) as
\begin{equation}\label{apB5}
\L_1=2\RE \big(A,\overline{\Psi_0}(\iu\nabla+A_0)\Psi_0 \big)_{L^2(\square)}.
\end{equation}
This representation for $\L_1$ shows that once $A_0$ is fixed, the vector function $\overline{\Psi_0}(\iu\nabla+A_0)\Psi_0$ is fixed as well, and we can vary the potential $A$ to achieve $\L_1\ne0$.
Let us prove this fact rigorously.

We are going to show that there exists at least one potential $A$ such the constant $\L_1$ defined by (\ref{apB5}) is non-zero. We argue by contradiction assuming that $\L_1$ vanishes for all $A$. Then formula (\ref{apB5}) implies that
\begin{equation}\label{apB6}
0\equiv \RE \overline{\Psi_0}(\iu\nabla+A_0)\Psi_0=-\IM \overline{\Psi_0}\nabla\Psi_0+ A_0|\Psi_0|^2=0\quad\text{in}\quad\square.
\end{equation}
Employing the definition of the function $\varpi$, it is straightforward to check that the above identity is rewritten equivalently as
\begin{equation}\label{apB7}
A_0=-\iu\overline{\varpi}\nabla \varpi.
\end{equation}
The right hand side in the above formula is a real-valued vector function. Indeed, since $\varpi\overline{\varpi}\equiv 1$, we easily confirm that
\begin{equation*}
\overline{\varpi}\nabla \varpi + \overline{\overline{\varpi}\nabla \varpi}=\overline{\varpi}\nabla \varpi + \varpi\nabla \overline{\varpi}=\nabla \varpi \overline{\varpi}\equiv 0
\end{equation*}
and hence, $\overline{\varpi}\nabla \varpi $ is a pure imaginary vector function.

On $L^2(\square)$, we introduce a unitary transformation by the formula
$\cU u:=\varpi u$, and $\cU^{-1}u=\overline{\varpi} u$. It is straightforward to check that the operator $\Op^0_{per}$ is unitary equivalent to
\begin{equation*}
\cU^{-1}\Op^0_{per}\cU=\big(\iu \nabla + \iu \overline{\varpi}\nabla \varpi +A_0\big)^2=-\D,
\end{equation*}
where the differential expressions in the right hand side are treated as operators in $L^2(\square)$ subject to the same boundary conditions as $\Op^0_{per}$. This means that $\L_1$  can vanish simultaneously for all $A$ only in the above discussed case $A_0=0$. Once $A_0$ describes a non-trivial magnetic field, there exists at least one potential $A$, for which $\L_1$ is non-zero. Once such potential $A$ is found and fixed, we can consider small variation of $A$ of the form $A+\d A$ and in view of the continuity of $\L_1$ in $A$, we conclude immediately that $\L_1$ is also non-zero for $A+\d A$ provide $\d A$ is small enough. Hence, there exist infinitely many examples of $A$, for which $\L_1$ is non-zero.

Moreover, given some $A_0$, $A$, for which $\L_1$ is non-zero, we easily see that replacing $A$ by $-A$, we change the sign of $\L_1$.
This means that in general, $\L_1$ is non-zero and can be both negative and positive. Hence, returning back to formula (\ref{apB2}), we conclude that in general,
the term of order $O(\e)$ in this formula is non-zero and can be both positive and negative.

Let us discuss the assumption $\frac{\Psi_0}{|\Psi_0|}=\varpi\in C^2(\overline{\square})$.
This is surely true once we know that $\Psi_0$ does not vanish in $\square$ and behaves ``well enough'' at the boundaries of $\square$.
Apriori, we failed trying to prove this fact for general $A_0$, but the set of potentials $A_0$ obeying this assumption is non-empty.
Indeed, we can start from $A_0\equiv 0$, then, as it was discussed in Appendix~\ref{appendix:V_0}, the ground state $\Psi_0$ is real and sign,
hence in this case the function $\varpi$ is identically constant: $\varpi\equiv 1$.
Now, consider a (sufficiently) small non-trivial magnetic potential $A_0$.
Its ground state will be a small perturbation of the one for $A_0\equiv0$, i.~e.~a (sufficiently) small variation of the function $\varpi\equiv 1$
The perturbed ground state will obviously still belong to $C^2(\overline{\square})$.
\end{remark}
%
%

\section*{Acknowledgments}

The scientific contribution by D.I.B. in the results presented in Sections~2,~4,~5,~7,~8 is financially supported by Russian Science Foundation (project No. 17-11-01004).


\def\polhk#1{\setbox0=\hbox{#1}{\ooalign{\hidewidth
  \lower1.5ex\hbox{`}\hidewidth\crcr\unhbox0}}}

\end{document}